\documentclass[reqno,11pt]{amsart}
\usepackage[latin1]{inputenc} 
\usepackage{listings, tcolorbox}

\usepackage{amssymb}
\usepackage{}
\usepackage{graphicx}  
\usepackage{float}  
\usepackage{subfigure}  
\usepackage{amssymb, amsmath, amsthm}
\usepackage[colorlinks=true,linkcolor=blue,citecolor=red]{hyperref}
\usepackage{color}
\usepackage{float}
\usepackage{enumerate}
\usepackage{graphics}
\usepackage{epstopdf}
\usepackage{tikz}
\usepackage{cite}
\usepackage{epsfig}
\usepackage{amsmath}
\usepackage{amssymb}
\usepackage{amscd}
\usepackage{graphicx}
\usepackage{mathrsfs}
\usepackage[T1]{fontenc}
\DeclareMathOperator*{\res}{Res}
\numberwithin{equation}{section}
\newtheorem{thm}{Theorem}[section]
\newtheorem{lem}[thm]{Lemma}
\newtheorem{rem}[thm]{Remark}
\newtheorem{prop}[thm]{Proposition}
\newtheorem{assum}[thm]{Assumption}
\newtheorem{RHP}[thm]{RH Problem}
\newtheorem{DRHP}[thm]{$\overline{\partial}$-Problem}
\newtheorem{cor}[thm]{Corollary}

\providecommand{\D}[1]{\mathbb{#1}}
\usepackage{epsfig}

\numberwithin{equation}{section}
\numberwithin{thm}{section}

\renewcommand{\Im}{\operatorname{Im}}
\renewcommand{\Re}{\operatorname{Re}}
\renewcommand{\res}{\operatorname{Res}}

\DeclareMathOperator{\di}{diag}



\begin{document}

\title[On the long-time asymptotic of the mCH equation]{On the long-time asymptotic of the modified Camassa-Holm equation with nonzero boundary conditions in space-time solitonic regions}


\author[Yang]{Jin-Jie Yang}

\author[Tian]{Shou-Fu Tian$^{*}$}
\address{Jin-Jie Yang, Shou-Fu Tian (Corresponding author) and Zhi-Qiang Li \newline
School of Mathematics, China University of Mining and Technology, Xuzhou 221116, People's Republic of China}
\thanks{$^{*}$Corresponding author(sftian@cumt.edu.cn, shoufu2006@126.com).
\email{sftian@cumt.edu.cn, shoufu2006@126.com (S.F. Tian)}}

\author[Li]{Zhi-Qiang Li}
%
%
%
%
%

\begin{abstract}
{We investigate the long-time asymptotic behavior for the Cauchy problem of the modified Camassa-Holm (mCH) equation with nonzero boundary conditions in different regions
\begin{align*}
&m_t+\left((u^2-u_x^2)m\right)_x=0,\ m=u-u_{xx}, \ (x,t)\in\mathbb{R}\times\mathbb{R}^+,\\
&u(x,0)=u_0(x),\ \ \lim_{x\to\pm\infty} u_0(x)=1,~~u_0(x)-1\in H^{4,1}(\D{R}),
\end{align*}
where  $m(x,t=0):=m_{0}(x)$ and $m_0(x)-1\in H^{2,1}(\D{R})$.
 Through spectral analysis, the initial value problem of the mCH equation is transformed into a matrix RH problem on a new plane $(y,t)$, and then using the $\overline{\partial}$ nonlinear steepest descent method, we analyze the different asymptotic behaviors of the four regions divided by the interval of $\xi=y/t$ on plane $\{(y,t)|y\in(-\infty,+\infty), t>0\}$.  There is no steady-state phase point corresponding to the regions $\xi\in(-\infty,-1/4)\cup(2,\infty)$. We prove that the solution of mCH equation is characterized by $N$-soliton solution and error on these two regions. In $\xi\in(-1/4,0)$ and $\xi\in(0,2)$, the phase function $\theta(z)$ has eight and four steady-state phase points, respectively. We prove that the soliton resolution conjecture holds, that is, the solution of the mCH equation can be expressed as the soliton solution on the discrete spectrum, the leading term on the continuous spectrum, and the residual error. Our results also show that soliton solutions of the mCH equation with nonzero condition boundary are asymptotically stable.}
\end{abstract}

\maketitle
\tableofcontents

\section{Introduction}\quad
The Camassa-Holm (CH) equation is an integrable nonlinear dispersion equation in the sense of having Lax representations \cite{CH-PRL,CH-AAM}
\begin{align}\label{CH-1}
m_t+(um)_x+u_xm=0,\ \ \ \ m:=u-u_{xx},
\end{align}
which can be used to describe the unidirectional propagation of shallow water waves on a flat bottom and has a rich mathematical structure, which has attracted widespread attention in \cite{CL-2009,Jo-2002}.  In particular,  Boutet and Shepelsky in \cite{BD-2008} developed the inverse scattering theory to study the long-time asymptotic behavior of the solution of the CH equation \eqref{CH-1} based on the nonlinear steepest descent method \cite{DZ-AM1993} by constructing the matrix Riemann-Hilbert (RH) problem \cite{BC-1984}, including Cauchy problem and initial-boundary value problem in \cite{BD-SIAM-2010,BD-SIAM-2009,BD-AMSc2008,BD-AIF}. Also the method is widely applied to other integrable systems to study soliton solutions and asymptotic solutions of initial value problems and initial-boundary value problems \cite{RHP-1,RHP-2,Tian-PAMS,DS-JDE,XJ-2015,Geng-CMP}.

On the other hand, the CH equation \eqref{CH-1} has a remarkable feature with fast decay initial value, which has weak solution and orbital stability \cite{Liu-AM-2014,Qu-CMP}. The CH equation has smooth soliton solution with non-zero dispersion term $\kappa u_x$ under the condition of fast decay initial value
\begin{align}\label{dis-CH}
m_t+\left((u^2-u_x^2)m\right)_x +\kappa u_x=0, \ m:=u-u_{xx},
\end{align}
where $\kappa>0$ represents the effect of the linear dispersion, which is also called mCH equation \cite{Chen-AM,Matsuno-JPA}. Under the initial value of weighted Sobolev space supporting solitons, Fan's team \cite{YF-AM} proved the soliton resolution conjecture of the mCH equation with fast decaying initial value, and gave the asymptotic stability of the soliton solution. It is worth noting that for another version of the mCH equation without nonlinear dispersion term  $\kappa u_x$, which is derived from  some literatures, e.g., \cite{YQZ-DCDS} and references therein.

Specifically, in 2009,  Novikov classified integrable equations by utilizing perturbation symmetry methods \cite{Novikov-JPA}
\begin{align}
\left(1-\partial_x^2\right)u_t=F(u,u_x,u_{xx},\ldots),\ u=u(x,t),
\end{align}
where $F$ is a homogeneous differential polynomial over $\D{C}$ (also \cite{MN-JPA}). Among a series of equations presented by Novikov, equation (32) in \cite{Novikov-JPA} is the second cubic nonlinear equation with the form (also called the modified Camassa-Holm (mCH) equation)
\begin{equation}\label{mCH-1}
m_t+\left((u^2-u_x^2)m\right)_x = 0, \quad m:=u-u_{xx},
\end{equation}
which has equivalent form given by Fokas in \cite{Fokas-1995}
(also see \cite{FGLQ-2013,OR-1996}). Shiff regarded equation \eqref{mCH-1} as the duality of the modified  Korteweg-de Vries (mKdV) equation, and gave the Lax pair of \eqref{mCH-1} based on the Lax pair of mKdV equation in \cite{Schiff-1996}. An equivalent Lax pair for \eqref{mCH-1} was proposed by Qiao in \cite{Qiao-JMP}, thus the mCH
equation is also called the Fokas-Olver-Rosenau-Qiao (FORQ) equation in \cite{HFQ-2017}. The mCH equation  \eqref{mCH-1} has an interesting feature, that is, it has peakon solution \cite{GLOQ-2013}. The dynamic stability and the stability of  peakons were studied in \cite{Qu-CMP,Liu-2014}. Chang and Szmigielski studied multipeakon solutions by using inverse spectrum method \cite{Chang-JNMP}. Besides, there are many results about mCH equation, including local well posedness \cite{Chen-JFA,Chen-AM,FGLQ-2013,LL-AA}, wave breaking mechanism \cite{GLOQ-2013}, algebraic geometric quasi-periodic solutions \cite{HFQ-2017},  global weak solutions \cite{Gao-DCDS}, Hamiltonian structure and Liouville integrability of peak systems \cite{Anco-2018,CH-AAM,GLOQ-2013,OR-1996}. In terms of asymptotic solution analysis, the asymptotic analysis of the dispersionless CH equation \eqref{CH-1} in the zero background (where the spectrum is completely discrete) requires different tools \cite{Eckhardt-I,Eckhardt-II,Eckhardt-III}, because the inverse scattering method requires the spatial equation of the Lax pair related to the CH equation to have a continuous spectrum. For the similar situation of the mCH equation \eqref{mCH-1}, Boutet, Karpenko, and Shepelsky developed the RH method to study the Cauchy problem of the mCH equation under the non-zero background \cite{mCH-JMP}, and studied the long-time asymptotic behavior of the solution in different regions \cite{mCH-NZBCs} by using the classical nonlinear steepest descent method.

Compared with equations \eqref{dis-CH} and \eqref{mCH-1}, the significantly different feature is that our work considers the version of the mCH equation  \eqref{mCH-1} without dispersion term $\kappa u_x$, which leads to the fact that the mCH equation  \eqref{mCH-1} has only peakon solutions and no smooth soliton solution when it has an attenuation initial value. Therefore, in order to study whether the mCH equation \eqref{mCH-1} has similar characteristics, soliton resolution conjecture, to the mCH equation \eqref{dis-CH} with  dispersion term $\kappa u_x$ and fast decaying initial value, we consider the initial value of finite density, which is also one of the main motivations of this work, that is, to study the long-time asymptotic behavior of the solution of the mCH equation  \eqref{mCH-1} in different soliton regions.

Recently, an asymptotic analysis method, the so-called $\overline{\partial}$-steepest descent   method, was proposed by McLaughlin and Miller in \cite{McLaughlin-1,McLaughlin-2} to investigate the oscillation RH problem.  Compared with the classical nonlinear steepest descent method, the advantage lies in improving the error accuracy on the one hand, and avoiding the complex norm   estimation on the other hand. This method has been used to solve the asymptotic solution        analysis of the initial value problem of nonlinear Schr\"{o}dinger (NLS) equation without       soliton since it was proposed in \cite{DM-2008}. Subsequently, the soliton resolution conjecture of focusing NLS equation fast decay initial value was solved by $\overline{\partial}$-steepest  descent method from the perspective of integrable system \cite{AIHP}. The defocusing NLS        equation with finite density initial value also had similar results and \cite{Cu-CMP} gave the  asymptotic stability of the $N$-soliton solution. In addition, some progress has been made in   proving the conjecture of soliton resolution of other nonlinear integrable models, including    modified CH equation \cite{YF-AM}, short-pulse equation \cite{Fan-1}, Fokas-Lenells equation    \cite{CF-JDE},  Wadati-Konno-Ichikawa equation \cite{WKI-2021} and other equations.

A recent work by Boutet de monvel, Karpenko and Shepelsky in  \cite{mCH-NZBCs} investigated the
long-time asymptotic solutions of Cauchy problem for the mCH equation with nonzero zero boundary conditions
\begin{subequations}\label{mCH1-ic}
\begin{alignat}{4}
&m_t+\left((u^2-u_x^2)m\right)_x=0,&\quad&m:=u-u_{xx},&\quad&t>0,&\;&-\infty<x<+\infty,\\
&u(x,0)=u_0(x), && u_0(x)\to 1, &&&&x\to\pm\infty,\label{IC}
\end{alignat}
\end{subequations}
in the absence of the discrete spectrum based on the classical nonlinear steepest descent method, where $u_0(x)-1\in H^{4,1}(\D{R})$. Subsequently, Yang, Fan and Liu \cite{YFL-2022} studied the existence of the global solution of the mCH equation \eqref{mCH1-ic} based on the inverse scattering theory. The purpose of our work
 is to study
the long-time asymptotic behavior of the solution of the mCH equation \eqref{mCH1-ic} under the condition of finite density initial value $ m_0(x)-1\in H^{2,1}(\D{R})$ and the presence of discrete spectrum, based on the work \cite{mCH-NZBCs} and in combination with the
$\overline{\partial}$-steepest  descent method. Furthermore, we will prove that
the soliton resolution conjecture of the mCH equation is also valid under the initial value
condition of finite density.

As in \cite{mCH-JMP}, the mCH equation \eqref{mCH1-ic} with nonzero zero boundary conditions can be transformed into zero background under the new function $q:=q(x,t)$ by
\begin{align}\label{q-new}
u(x,t)=q(x-t,t)+1.
\end{align}
Then the mCH equation is
\begin{subequations}\label{mCH-2}
\begin{align}
&\breve{m}_t+(\breve{\omega}\breve{m})_x=0,\\
&\breve{m}:=q-q_{xx}+1,\\
&\breve{\omega}:=q^2-q_x^2+2q.
\end{align}
\end{subequations}
In what follows, we will study long-time asymptotic behavior of mCH equation \eqref{mCH-2} with zero background $q(x,0):=q_0(x)\to 0$ as $x\to\pm\infty$ and the initial data $\breve{m}_0-1\in H^{2,1}(\D{R})$.

The soliton resolution conjecture can generally be expressed as that the evolution of the general initial data of globally well posed nonlinear dispersion equations will be decomposed into a finite soliton plus a dispersive radiation component. For most nonlinear dispersion equations, it is an open and very active research field \cite{Cuccagna-2014,Martel-SIAM,Martel-MA}. The interesting point is that soliton resolution is better understood in integrable systems, and the solution provided by RH problem is more accurate than that obtained by pure analytic techniques \cite{Cuccagna-AA,Deift-1996,DP-2011}.
Although \cite{mCH-NZBCs} gives the asymptotic behavior of the solution of the mCH equation under non-zero boundary conditions from the classical nonlinear steepest descent method, we will further consider the asymptotic behavior of the solution of the mCH equation under non-zero boundary conditions in different regions from the perspective of $\overline{\partial}$-nonlinear steepest descent method, and verify the validity of the soliton resolution conjecture. \\


The basic framework of this work: In Section \ref{sec2}, we do spectral analysis of the mCH equation under the initial value condition of finite density, and establish the analytical, asymptotic and symmetric properties of the eigenfunctions and scattering coefficients. In addition,   the RH problem corresponding to the mCH are established with finite density initial value \eqref{IC}.  Section \ref{sec3} mainly considers the properties of the phase function $\theta(z)$, including the distribution of steady-state phase points and the symbol table that affects the attenuation of the oscillation terms $e^{\pm2it\theta(z)}$.

In Section \ref{sec4}, the decomposition of RH problem on regions $\xi\in(-\infty,-1/4)$ and $\xi\in(2,\infty)$ depends on the change of the symbol table in Fig.\ref{Fig1} of the phase function discussed in Section \ref{sec3}, corresponding to the absence of steady-state phase points.  Section \ref{subsec4.1} mainly removes the jump of $M^{(1)}(z)$ on the real axis $\mathbb{R}$, and analyzes the attenuation region of the imaginary part of the phase function to prepare for continuous extension. In Section \ref{subsec4.2},
the second transformation obtains a new matrix-valued function $M^{(2)}(z)$ corresponding to the mixed RH problem \ref{RHP-M2}, which includes soliton component and pure $\overline{\partial}$-problem.  Pure soliton component $M^{(rhp)}(z)$ satisfied RH problem \ref{RHP-Mrhp} is considered, corresponding to RH problem  \ref{RHP-M2} with $\overline{\partial}R^{(2)}(z)\equiv0$ discussed in Section \ref{subsec4.3}. The error function $M^{(err)}(z)$ satisfies RH problem \ref{RHP-Merr}  solved by small-norm RH problem presented in Section \ref{subsec4.4}. In the last Subsection \ref{subsec4.5}, we discuss the contribution of pure RH problem to the solution of mCH equation.

In Section \ref{sec5}, the decomposition of RH problem on regions $\xi\in(-1/4,0)$ and $\xi\in(0,2)$ depends on the change of the symbol table in Fig.\ref{Fig1} of the phase function discussed in Section \ref{sec3}, corresponding to eight steady-state phase points and four steady-state phase points, respectively. Similar to Section \ref{sec4}, we make a series of deformations to the original RH problem, and finally turn it into a solvable RH problem. The difference from Section \ref{sec4} is that the steady-state phase points appear, and the contribution of the steady-state phase points to the solution of the mCH equation needs to be handled carefully, which is reflected in Appendix \ref{AppA}. In Section \ref{sec6}, we summarize the results corresponding to the four cases, and give the different properties of the solution of the mCH equation under the initial value condition of finite density.

\section{Spectral analysis and eigenfunctions} \label{sec2}\quad
This section will briefly review the spectral analysis in \cite{mCH-JMP} reported by Boutet
 de Monvel, Karpenko and Shepelsky and make appropriate adjustments to facilitate the subsequent long-time asymptotic analysis.
\subsection{Lax pairs}
The Lax representation of the mCH equation \eqref{mCH-1} has the following form in \cite{Qiao-JMP}
\begin{align}\label{Lax-O}
\Phi_x=X_0\Phi,\qquad\qquad\Phi_x=T_0\Phi,
\end{align}
where $\Phi:=\Phi(x,t,\lambda)$, $X_0:=X_0(x,t,\lambda)$, and $T_0:=T_0(x,t,\lambda)$ with
\begin{align*}
X_0&=\frac{1}{2}\begin{pmatrix}
-1 & \lambda  m \\
-\lambda m & 1
\end{pmatrix},\qquad\quad  m:=u-u_{xx},\\
T_0&=\begin{pmatrix}\lambda^{-2}+\frac{u^2-u_x^2}{2} &
-\lambda^{-1}(u-u_x)-\frac{\lambda(u^2-u_x^2)m}{2}\\
\lambda^{-1}(u+u_x)+\frac{\lambda(u^2- u_x^2)m}{2} &
		 -\lambda^{-2}-\frac{u^2-u_x^2}{2}\end{pmatrix},
\end{align*}
from which one obtains a new Lax representation under the transformation \eqref{q-new}
\begin{subequations}\label{Lax2}
\begin{align}
\Phi_x&=X_1\Phi,\label{Lax-x}\\
\Phi_t&=T_1\Phi,\label{Lax-t}
\end{align}
\end{subequations}
where the coefficients $X_1:=X_1(x,t,\lambda)$ and $T_1:=T_1(x,t,\lambda)$ are defined by
\begin{subequations}
\begin{align}
X_1&=\frac{1}{2}\begin{pmatrix} -1 & \lambda \breve{m} \\
-\lambda \breve{m} & 1 \end{pmatrix}, \quad \breve{m}:=q-q_{xx}+1, \quad \breve{\omega}:=q^2-q_x^2+2q, \label{Lax-X1}\\
T_1&=\begin{pmatrix}\lambda^{-2}+\frac{\breve{\omega}}{2} &
-\lambda^{-1}(q-q_x+1)-\frac{\lambda\breve{\omega}\breve{m}}{2}\\
\lambda^{-1}(q+q_x+1)+\frac{\lambda\breve{\omega}\breve{m}}{2} & -\lambda^{-2}-\frac{\breve{\omega}}{2}\end{pmatrix}.\label{Lax-T1}
\end{align}
\end{subequations}
Then the mCH equation \eqref{mCH-2} can be directly calculated from the compatibility condition $T_{1,t}-X_{1,x}+[T_1,X_1]=0$, where the operator $[A,B]=AB-BA$. By observing the spectral problem \eqref{Lax-X1} of mCH equation, it is found that there are two spectral singularities, $\lambda=0$ and $\lambda=\infty$, respectively. In order to control the asymptotic behavior, we need to discuss them respectively.

As $\lambda\to\infty$, we introduce the transformation
\begin{align}\label{Lax-psi}
\Psi(x,t,\lambda)=D(\lambda)\Phi(x,t,\lambda),
\end{align}
where  $D(\lambda)$ is the eigenvector matrix with
\begin{align*}
D(\lambda)=\left(
             \begin{array}{cc}
               1 & -\frac{1}{1+\sqrt{1-\lambda^2}} \\
               -\frac{1}{1+\sqrt{1-\lambda^2}} & 1 \\
             \end{array}
           \right),
\end{align*}
to transform the Lax pairs \eqref{Lax2} into
\begin{subequations}\label{Lax3}
\begin{align}
&\Psi_x(x,t,\lambda)+P_x\Psi_x(x,t,\lambda)=X_2\Psi_x(x,t,\lambda),\\
&\Psi_t(x,t,\lambda)+P_t\Psi_x(x,t,\lambda)=T_2\Psi_x(x,t,\lambda),
\end{align}
\end{subequations}
where $X_2=X_2(x,t,\lambda)$ and $T_2=T_2(x,t,\lambda)$ are given by
\begin{align*}
&X_2=\frac{\lambda(\breve{m} -1)}{2\sqrt{1-\lambda^2}}
\begin{pmatrix}
0 & 1 \\
-1 & 0 \\
\end{pmatrix}
+\frac{\breve{m} -1}{2\sqrt{1-\lambda^2}}\sigma_3,\\
&T_2=
\frac{i}{2 \sqrt{1-\lambda^2}}\left(\lambda\breve{\omega}(\breve{m} - 1)
+\frac{2 q}{\lambda}\right)
\sigma_2
+\frac{ q_x}{\lambda}  \sigma_1
 -\frac{1}{\sqrt{1-\lambda^2}}\left(q +\frac{1}{2}(\breve{m}-1)\breve{\omega}\right) \sigma_3,\\
&P:=P(x,t,\lambda)=p(x,t,\lambda)\sigma_3, \text {with}\\
&p(x,t,\lambda)=-\frac{1}{2}\sqrt{1-\lambda^2}\left(\int_x^{+\infty} (\breve{ m}(\xi,t)-1)\,d\xi-\frac{1}{2}x+\frac{1}{\lambda^2}t\right).
\end{align*}
\subsection{The eigenfunctions}
The Lax pairs in form \eqref{Lax3} allows us to determine the special solution with good control behavior through the relevant integral equation. Indeed, introducing the transformation $\widetilde{\Psi}(z)=\widetilde{\Psi}(x,t,z)$
\begin{align}\label{psi-tilde}
\widetilde{\Psi}(z)=\Psi(z)e^{P},
\end{align}
yields that
\begin{subequations}\label{Lax4}
\begin{align}
&\widetilde{\Psi}_x(z)+[P_x,\widetilde{\Psi}(z)]=X_2\widetilde{\Psi}(z),\label{Lax4a}\\
&\widetilde{\Psi}_t(z)+[P_t,\widetilde{\Psi}(z)]=T_2\widetilde{\Psi}(z).
\end{align}
\end{subequations}
As $x\to\infty$, the systems \eqref{Lax4} are determined by the Volterra integral equations
\begin{equation}\label{E-T-3}
\begin{split}
\widetilde\Psi_+(x,t,\lambda)&=I-\int_x^{+\infty}
e^{\frac{\sqrt{1-\lambda^2}}{2}
\int_x^\xi \breve{m}(\eta,t)\,d\eta\,\sigma_3} X_2(\xi,t,\lambda)\widetilde\Psi_+(\xi,t,\lambda)e^{-\frac{\sqrt{1-\lambda^2}}{2}\int_x^\xi\breve{m}
(\eta,t)\,d\eta\,\sigma_3}\,d\xi,\\
\widetilde\Psi_{-}(x,t,\lambda)&=I+\int_{-\infty}^xe^{\frac{\sqrt{1-\lambda^2}}{2}\int_x^\xi\breve{m}
(\eta,t)\,d\eta\,\sigma_3} X_2(\xi,t,\lambda)\widetilde\Psi_{-}(\xi,t,\lambda)e^{-\frac{\sqrt{1-\lambda^2}}{2}\int_x^\xi\breve{m}
(\eta,t)\,d\eta\,\sigma_3}\,d\xi,
\end{split}
\end{equation}
where $\widetilde\Psi_\pm\equiv\widetilde\Psi_\pm(x,t,\lambda)$ are also called Jost functions. In order to eliminate the multivaluence of functions in exponential functions in \eqref{E-T-3}, the following transformation is further introduced
\begin{align*}
  \lambda^2=4k^2+1,
\end{align*}
then the exponential functions in \eqref{E-T-3} become  $e^{\pm ik\int_x^\xi\breve{m}
(\eta,t)\,d\eta\,\sigma_3}$. Also the coefficient matrices $X_2$ and $T_2$ have  the following new forms
\begin{align*}
X_2(x,t,k) &= \frac{\breve{m}-1}{2}\left(\frac{1}{2ik}
            \begin{pmatrix}
                1 & 0 \\ 0 & -1
            \end{pmatrix}
            +\frac{\lambda(k)}{2ik}
            \begin{pmatrix}
                0 & 1 \\ -1 & 0
            \end{pmatrix}
            \right),\\
T_2(x,t,k)&=
\frac{1}{4k}\left(\lambda(k)\breve{\omega}(\breve{m}-1)
+\frac{2 q}{\lambda(k)}\right)
\sigma_2
+\frac{ q_x}{\lambda(k)}\sigma_1
-\frac{1}{2ik}\left(q +\frac{1}{2}(\breve{m}-1)\breve{\omega}\right) \sigma_3.
\end{align*}

It is observed that the new spectral parameter $k$ on which the coefficient matrices $X_2$ and $T_2$ depend  are not rational because of the existence of $\lambda(k)$, which is also a great difference between the mCH equation and CH equation. In order to avoid discussing on the Riemann surface $\lambda^2=4k^2+1$, we introduce a new spectral parameter $z$, then both $\lambda$ and $k$ are single valued functions of $z$ with
\begin{equation}\label{la-k-z}
\lambda=-\frac{1}{2}\left(z+\frac{1}{z}\right), \qquad
k = \frac{1}{4}\left(z-\frac{1}{z}\right).
\end{equation}

Now combining the transformation \eqref{la-k-z} with the expressions of $p(x,t,\lambda)$ and $X_{2}$ obtains
\begin{align}
p(x,t,z)&=-\frac{ i(z^2-1)}{4z}\left(\int_x^{+\infty} (\breve{m}(\xi,t)-1)\,d\xi-x+\frac{8z^2}{(z^2+1)^2}t\right),\label{e1}\\
X_2(x,t,z)&=
\frac{ i(z^2+1)(\breve{m}-1)}{2(z^2-1)}
\begin{pmatrix}
0 & 1 \\
-1 & 0 \\
\end{pmatrix}
-\frac{ iz(\breve{m} - 1)}{z^2-1}
\begin{pmatrix}
1 & 0 \\
0 & -1 \\
\end{pmatrix},\label{e2}
\end{align}
from which we obtain the new form of Volterra integral equations
\begin{equation}\label{e3}
\widetilde\Psi_{\pm}(x,t,z)=I+\int_{\pm\infty}^xe^{\frac{ i(z^2-1)}{4z}\int_x^\xi\breve{m}(\eta,t)\,d\eta\,\sigma_3} X_2(\xi,t,z)\widetilde\Psi_{\pm}(\xi,t,z)e^{-\frac{ i(z^2-1)}{4z}\int_x^\xi\breve{m}(\eta,t)\,d\eta\,\sigma_3}\,d\xi.
\end{equation}

Applying the Neumann series expansions for the solutions of $\widetilde\Psi_{\pm}(x,t,z)$ in \eqref{e3}, we can investigate the analytic and asymptotic properties of eigenfunctions $\widetilde\Psi_{\pm}(x,t,z)$  by analogy with the case of the CH equation \cite{BD-2006,BD-2008}. As the results reported in \cite{mCH-JMP}, we obtain the following proposition.
\begin{prop}\label{pp-ana}
The eigenfunctions $\widetilde\Psi_{\pm}(x,t,z)$ have the following properties:
\begin{enumerate}[$\blacktriangleright$]
\item The eigenfunctions $\widetilde\Psi_{-,1}$ and $\widetilde\Psi_{+,2}$ are analytic in ${\mathbb{C}}^+=\{z\in\mathbb{C}\mid \operatorname{Im} z>0\}$.
\item The eigenfunctions $\widetilde\Psi_{+,1}$ and $\widetilde\Psi_{-,2}$ are analytic in ${\mathbb{C}}^-=\{z\in\mathbb{C}\mid \operatorname{Im} z <0\}$.
\item The eigenfunctions $\widetilde\Psi_{-,1}$, $\widetilde\Psi_{+,2}$, $\widetilde\Psi_{+,1}$, and $\widetilde\Psi_{-,2}$ are continuous up to the real line except at $z=\pm 1$, where
 $\widetilde\Psi_{\pm,j} (j=1,2)$ denote the $j$-th column of $\widetilde\Psi_{\pm}$.
\end{enumerate}
\end{prop}

Additionally, for $z\neq\pm1$ and $z\neq0$, the symmetries of the coefficient matrices $X_2(z)$ and $T_2(z)$
\begin{align*}
X_2(z^{*})&=\sigma_1 X_2^{*}(z) \sigma_1,&~~&X_2(-z)=\sigma_2X_2(z) \sigma_2,&~~&X_2(z^{-1})=\sigma_1X_2(z)\sigma_1,\\
T_2(z^{*})&=\sigma_1 T_2^{*}(z) \sigma_1,&&T_2(-z)=\sigma_2T_2(z)\sigma_2, &&T_2(z^{-1})=\sigma_1T_2(z)\sigma_1,
\end{align*}
allow us to establish the symmetries of the eigenfunctions $\widetilde\Psi_{\pm}(z)$
\begin{equation}\label{sym-Psi}
\widetilde\Psi_\pm(z^{*})=\sigma_1 \widetilde\Psi_{\pm}^{*}(z) \sigma_1, \quad \widetilde\Psi_\pm(-z)=\sigma_2\widetilde\Psi_\pm(z)\sigma_2, \quad \widetilde\Psi_\pm(z^{-1})=\sigma_1\widetilde\Psi_\pm(z)\sigma_1,
\end{equation}
for $\pm \operatorname{Im} z\leq 0$ and  $z\neq\pm1$, or equivalently,
$$\widetilde\Psi_{\pm,1}(z)=\sigma_1\widetilde\Psi^{*}_{\pm,2}(z)=\sigma_3 \sigma_1\widetilde\Psi_{\pm,2}(-z)=\sigma_1\widetilde\Psi_{\pm,2}(z^{-1}).$$

Furthermore, the combination coefficient matrix $U_2(z)$ is traceless matrix and the relation of equation \eqref{e3} produces the following relation.
\begin{prop}\label{pp-asy}
Asymptotic properties and asymptotic behavior  of eigenfunctions $\widetilde\Psi_{\pm}$ at singular points $z=\pm1$:
\begin{enumerate}[$\blacktriangleright$]
\item
As $z\to\infty$ $($is equivalent to $z\to0$ by the symmetry \eqref{sym-Psi}$)$, the eigenfunctions $\left(\begin{smallmatrix}
\widetilde\Psi_-^{(1)} &
\widetilde\Psi_+^{(2)}\end{smallmatrix}\right)\to I$ for $\Im z\geq 0$.
\item
As $z\to\infty$ $($is equivalent to $z\to0$ by the symmetry \eqref{sym-Psi}$)$, the eigenfunctions $\left(\begin{smallmatrix}
\widetilde\Psi_+^{(1)} &
\widetilde\Psi_-^{(2)}\end{smallmatrix}\right)\to I$ for  $\Im z\leq 0$.
\item
As $z\to 1$, $\widetilde\Psi_\pm(x,t,z)=\frac{i}{2(z-1)}\alpha_\pm(x,t)\left(\begin{smallmatrix}-1&1\\ -1&1\end{smallmatrix}\right)+O(1)$ with $\alpha_\pm(x,t)\in \mathbb{R} $ $($understood column-wise, in the corresponding half-planes$)$.
\item
As $z\to -1$, $\widetilde\Psi_\pm(x,t,z)=-\frac{i}{2(z+1)}\alpha_\pm(x,t)\left(\begin{smallmatrix}1 & 1 \\ -1 & -1\end{smallmatrix}\right)+O(1)$ with the same $\alpha_\pm(x,t)$ as the previous ones $($by symmetry \eqref{sym-Psi}$)$.
\end{enumerate}
\end{prop}

\subsection{Scattering matrix}
As usual, there is a matrix $S(z)=\{(s_{ij})\}_{i,j=1}^2$ that only depends on the spectral parameter $z$ and associates the eigenfunctions $\widetilde{\Psi}_+$ and $\widetilde{\Psi}_-$ on the real line $\D{R}$
\begin{align}\label{scattering}
\widetilde{\Psi}_+(x,t,z)=\widetilde{\Psi}_-(x,t,z)S(z), ~~~~z\in\D{R},~z\neq\pm1,
\end{align}
which is also called scattering matrix. Simple calculations show that
\begin{align}\label{ele-sca}
s_{22}(z)=\det[\widetilde{\Psi}_{-,1},\widetilde{\Psi}_{+,2}],\qquad
s_{12}(z)=e^{2p}\det[\widetilde{\Psi}_{+,2},\widetilde{\Psi}_{-,2}].
\end{align}
It follows from the symmetries of eigenfunctions \eqref{sym-Psi} that
\begin{align}\label{S}
S(z)=\left(
       \begin{array}{cc}
         s_{11}(z) & s_{12}(z) \\
         s_{21}(z) & s_{22}(z) \\
       \end{array}
     \right)
\triangleq  \left(
        \begin{array}{cc}
          \overline{a(z)} & b(z) \\
          \overline{b(z)} & a(z) \\
        \end{array}
      \right).
\end{align}
The properties of entires of the scattering matrix $S(z)$ in \eqref{S} are derived by Propositions \ref{pp-ana}, \ref{pp-asy} and the expressions \eqref{ele-sca} as follow:
\begin{prop}\label{pp-S}
The scattering matrix $S(z)$ is of the following properties:
\begin{enumerate}[$\blacktriangleright$]
\item $a(z)$ is analytic in $\mathbb{C}^+$, $a(0)=1$ and $a(z)\to1$ as $z\to\infty$.
\item $a(z)=\overline{a(-\overline{z})}=a(-z^{-1})$ for $\Im z>0$.
\item $b(z)$ is continuous for $z\in\D{R}\setminus\{\pm1\}$, $b(0)=0$, and $b(z)\to0$ as $z\to\infty$.
\item As $z\to1$, there exists a constant $\gamma\in\D{R}$ such that
\begin{align}
a(z)=\frac{i\gamma}{2(z-1)}+O(1),~~b(z)=\frac{i\gamma}{2(z-1)}+O(1).
\end{align}
\item As $z\to-1$, there exists a constant $\gamma\in\D{R}$ with the same as previous one such that
\begin{align}
a(z)=\frac{i\gamma}{2(z+1)}+O(1),~~b(z)=-\frac{i\gamma}{2(z+1)}+O(1).
\end{align}
\item For $z\in\D{R}\setminus\{\pm1\}$, $|a(z)|^2-|b(z)|^2=1.$
\item The reflection coefficient $r(z)=\frac{b(z)}{\overline{a(z)}}$ such that
\begin{align}
r(z)=-\overline{r(-\overline{z})}(\text{or}\ -\overline{r(-z)})=r(-z^{-1}),\ \ z\in\D{R}.
\end{align}
\end{enumerate}
\end{prop}

We cannot rule out that the scattering coefficient $s_{22}(z)$ has zeros on the real axis $\D{R}$, which are called spectral singularities. In this case, it will bring great difficulties to our work. Therefore, in order to avoid this situation, we make the following assumptions.
\begin{assum}
For the scattering coefficient $s_{22}(z)$ generated by the initial data $q_0(x)\in H^{2,1}(\D{R})$ such that
\begin{enumerate}[(1)]
\item The scattering coefficient $s_{22}(z)$ has no zeros on the real axis $\D{R}$.
\item The scattering coefficient $s_{22}(z)$ has finite simple zeros.
\end{enumerate}
\end{assum}

\subsection{Construction of Riemann-Hilbert problem}
Applying the analytical properties of the eigenfunctions $\widetilde{\Psi}_\pm$ and the scattering matrix $S(z)$, a piecewise analytical function $\widetilde{M}(z):=\widetilde{M}(x,t,z)$ can be introduced as usual (it depends on time and space at the same time. In fact, it should be consistent with classical inverse scattering transformation. A piecewise analytical function only related to space is introduced, and the time part is processed through time evolution. The two processing methods are essentially equivalent here).
\begin{align}\label{M-tilde}
\widetilde{M}(z):=\widetilde{M}(x,t,z)=\left\{\begin{aligned}
\left(\frac{\widetilde{\Psi}_{-,1}(x,t,z)}{a(z)}~~\widetilde{\Psi}_{+,2}(x,t,z)\right),&&z\in\D{C}^+,\\
\left(\widetilde{\Psi}_{+,1}(x,t,z)~~\frac{\widetilde{\Psi}_{-,2}(x,t,z)}{\overline{a(z)}}\right),&&z\in\D{C}^-.
\end{aligned}\right.
\end{align}
Then the piecewise analytical function $\widetilde{M}(z)$ solves the following RH problem.
\begin{RHP}
Find a $2\times2$ matrix-valued function
$\widetilde{M}(z)$ satisfying
\begin{enumerate}[(I)]
\item Analytic condition: $\widetilde{M}(z)$ is meromorphic in $\D{C}\setminus\D{R}$.
\item Symmetry condition: $\widetilde{M}(z)=\overline{\widetilde{M}(\bar{z}^{-1})}=
    \sigma_1\overline{\widetilde{M}(\overline{z})}\sigma_1=
\sigma_3\overline{\widetilde{M}(-\bar{z})}\sigma_3$.
\item Jump condition:
\begin{align}\label{Jp-tilde}
\widetilde{M}_+(z)=\widetilde{M}_-(z)\widetilde{V}(z),~~~~z\in\D{R}\setminus\{\pm1\},
\end{align}
where  the matrix $\widetilde{V}(z)$ is defined by
\begin{align}
\widetilde{V}(z)=\left(
                   \begin{array}{cc}
                     1-r(z)\overline{r}(z) & r(z)e^{-2p(z)} \\
                     -\overline{r}(z)e^{2p(z)} & 1 \\
                   \end{array}
                 \right),
\end{align}
with $p(z):=p(x,t,z)$ in \eqref{e1}.
\item Asymptotic conditions: $\widetilde{M}(z)\to I$ as $z\to\infty$ $(z\to0)$.
\item Singularity conditions:
\begin{align}
\widetilde{M}(z)=\left\{\begin{aligned}
&\frac{i\alpha_+}{2(z-1)}\left(
                          \begin{array}{cc}
                            -c & 1 \\
                            -c & 1 \\
                          \end{array}
                        \right)+O(1),&&z\to1, \Im z>0,\\
&-\frac{i\alpha_+}{2(z+1)}\left(
                          \begin{array}{cc}
                            c & 1 \\
                            -c &-1 \\
                          \end{array}
                        \right)+O(1),&&z\to-1, \Im z>0,
\end{aligned}\right.
\end{align}
where $\alpha_+\in\D{R}$ and
\begin{align*}
c=\left\{\begin{aligned}
&0,&&\gamma\neq0,\\
&\frac{a(1)+b(1)}{a(1)},&&\gamma=0.
\end{aligned}\right.
\end{align*}
\item Residue conditions: $\widetilde{M}(z)$ has simple poles in $\mathcal {Z}=\{\eta_n,\overline{\eta}_n\}_{n=1}^{8N_1+4N_2}$ with
\begin{align}\label{res-tilde}
&\mathop{\res}\limits_{\eta_n} \widetilde{M}(z)=\lim_{z\to\eta_n}\widetilde{M}(z)\left(
                                                  \begin{array}{cc}
                                                    0 & 0 \\
                                                    c_ne^{2p(\eta_n)} & 0 \\
                                                  \end{array}
                                                \right),\\
&\mathop{\res}\limits_{\overline{\eta}_n} \widetilde{M}(z)=\lim_{z\to\overline{\eta}_n}\widetilde{M}(z)\left(
                                                  \begin{array}{cc}
                                                    0 & \overline{c}_ne^{-2p(\overline{\eta}_n)} \\
                                                    0 & 0 \\
                                                  \end{array}\right),
\end{align}
where the discrete spectrums $\eta_n=z_n$, $\eta_{n+N_1}=-\overline{z}_{n}$, $\eta_{n+2N_1}=-z_n^{-1}$, and $\eta_{n+3N_1}=\bar{z}_n^{-1}$ for $n=1,2,\cdots,N_1$, the discrete spectrums on the unit circle are expressed as $\eta_{m+4N_1}=\kappa_m$ and $\eta_{m+4N_1+N_2}=-\overline{\kappa}_m$ for $m=1,2,\cdots,N_2$.
\end{enumerate}
\end{RHP}
\begin{proof}
See \cite{mCH-JMP}.
\end{proof}

In order to recover the potential function, we need to further consider the behavior at singular point $z=i$ corresponding to $\lambda=0$, which has been reported in \cite{mCH-JMP}. Here we will briefly review it. The Lax pair \eqref{Lax3} under the new spectrum parameter $z$ becomes
\begin{subequations}\label{Lax-z}
\begin{align}
  &\Psi_x+\frac{ i(z^2-1)}{4z}\sigma_3 \Psi = X_3 \Psi,\label{R-3}\\
  &\Psi_t-\frac{2 i(z^2-1)z}{(z^2+1)^2}\sigma_3 \Psi=T_3 \Psi,\label{R-4}
\end{align}
\end{subequations}
where
\begin{align*}
X_3&= \frac{(z^2+1)(\breve{m} - 1)}{2 (z^2-1)}\sigma_2
-\left(\frac{ iz(\breve{m}-1)}{z^2-1}+\frac{ i(z^2-1)\breve{m}}{4z}-\frac{ i(z^2-1)}{4z}\right)\sigma_3,\\
T_3&=\frac{ i(z^2-1)}{4z}(q^2-q_x^2+2q)\breve{m}\sigma_3 + T_2(x,t,z).
\end{align*}

Introducing the variable
\begin{align}
p_0(x,t,z)=\frac{i(z^2-1)}{4z}x-\frac{2i(z^2-1)}{(z^2+1)^2}t,
\end{align}
and $P_0=p_0\sigma_3$, $\widetilde{\Psi}_0=\Psi e^{P_0}$ transforms the Lax pair \eqref{Lax-z}  into
\begin{subequations}\label{Lax-P}
\begin{align}
&\widetilde{\Psi}_{0,x}+[P_{0,x},\widetilde{\Psi}_0]=X_3\widetilde{\Psi}_0,\\
&\widetilde{\Psi}_{0,t}+[P_{0,t},\widetilde{\Psi}_0]=T_3\widetilde{\Psi}_0,
\end{align}
\end{subequations}
from which one obtains the solution to Lax pair \eqref{Lax-P}
\begin{align}
\widetilde\Psi_{0,\pm}(x,t,z)=I+\int_{\pm\infty}^xe^{-\frac{ i(z^2-1)}{4z}(x-\xi)\sigma_3}  X_3(\xi,t,z)\widetilde\Psi_{0,\pm}(\xi,t,z)e^{\frac{ i(z^2-1)}{4z}(x-\xi)\sigma_3}\,d\xi,
\end{align}
which is also called Jost functions.
Similar to CH and SP equations, the Jost functions $\widetilde\Psi_{0,\pm}$ and $\widetilde\Psi$ are related by matrices $C_{\pm}(z)$  independent of $x$ and $t$:
\begin{equation}\label{Related}
\widetilde\Psi_{\pm}(x,t,z)=\widetilde\Psi_{0\pm}(x,t,z)e^{-p_0(x,t,z) \sigma_3}C_\pm(z)e^{p(x,t,z)\sigma_3},
\end{equation}
a calculation reveals that $C_\pm(z)=e^{(p_0(\pm\infty,t,z)-p(\pm\infty,t,z))\sigma_3}$ and $p(x,t,z)-p_0(x,t,z)=-\frac{i(z^2-1)}{4z}\int_{x}^{-\infty}(\breve{m}(\xi,t)-1)\,d\xi$, thus
\begin{align*}
C_+(z)\equiv I,~~~~C_-(z)=e^{\frac{i(z^2-1)}{4z}\int_{-\infty}^{\infty}(\breve{m}(\xi,t)-1)\,d\xi\sigma_3}.
\end{align*}

Additionally, observing that $\widetilde\Psi_{0,\pm}(x,t,i)\equiv I$ from the fact $X_3(x,t,i)\equiv0$, then we have
\begin{align}
\widetilde\Psi_{+}(x,t,i)=e^{\frac{1}{2}\int_{x}^{+\infty}(\breve{m}(\xi,t)-1)\,d\xi\,\sigma_3}, \quad \widetilde\Psi_{-}(x,t,i)=e^{-\frac{1}{2}\int^{x}_{-\infty}(\breve{m}(\xi,t)-1)\,d\xi\,\sigma_3},
\end{align}
which leads to
\begin{align}
a(i)=e^{-\frac{1}{2}\int_{-\infty}^{+\infty}(\breve{m}(\xi,t)-1)\,d\xi}
\end{align}
and
\begin{align}
 M(x,t,i)=\left(
             \begin{array}{cc}
               e^{\frac{1}{2}\int_x^{+\infty}(\tilde m(\xi,t)-1)\,d\xi} & 0 \\
               0 & e^{-\frac{1}{2}\int_x^{+\infty}(\tilde m(\xi,t)-1)\,d\xi} \\
             \end{array}
           \right),
\end{align}
from expressions \eqref{ele-sca} and \eqref{M-tilde}.

By analogy with CH equation \cite{BD-2006,BD-2008}, we transform the ($x,t$)-plane into ($y,t$)-plane by
\begin{align}
y(x,t)=x-\int_{x}^{-\infty}(\breve{m}(\xi,t)-1)d\xi,
\end{align}
and introduce a new matrix-valued function $M(z):=M(y(x,t),t,z)$ satisfying
\begin{align}\label{M-O}
\widetilde{M}(x,t,z)=M(y(x,t),t,z)
\end{align}
which solves the following RH problem.
\begin{RHP}\label{RHP-M}
Find a matrix $M(y,t,z)$ such that the following conditions:
\begin{enumerate}[$(a)$]
\item $M(z)$ is analytic in $\D{C}\setminus\{\D{R}\cup\mathcal {Z}\}$.
\item $M(z)=\overline{M(\bar{z}^{-1})}=
    \sigma_1\overline{M(\overline{z})}\sigma_1=
\sigma_3\overline{M(-\bar{z})}\sigma_3$.
\item Jump condition: The non-tangential limits $M_{\pm}(z)=\lim\limits_{\varepsilon\to0^+}M(z+\pm i\varepsilon)$ for $z\in\D{R}\setminus\{\pm1\}$
\begin{align}\label{JM}
M_{+}(z)=M_{-}(z)V(z),~~~~~z\in\D{R}\setminus\{\pm1\},
\end{align}
where the jump matrix $V(z)$ is defined by
\begin{align}\label{Jump-M}
V(z)=\left(
       \begin{array}{cc}
         1-|r(z)|^2 & r(z)e^{-2it\theta(z)} \\
         -\overline{r}(z)e^{2it\theta(z)} & 1 \\
       \end{array}
     \right)
\end{align}
with the phase function $\theta(z)=-\frac{i(z^2-1)}{4z}\left(-y+\frac{8t}{(z^2+1)^2}\right)$.
\item Asymptotic behavior:
\begin{align}\label{Asy-M}
&M(z)=I+O(z^{-1}), z\to\infty~ (\text{and}~ z\to0 ~\text{by symmetry}),\\
&M(z)=\left(
        \begin{array}{cc}
          a_1(y,t) & 0 \\
          0 & a_1^{-1}(y,t) \\
        \end{array}
      \right)+\left(
        \begin{array}{cc}
          0 & a_2(y,t) \\
          a_3(y,t) & 0 \\
        \end{array}
      \right)(z-i)+O((z-i)^2),
\end{align}
where $a_j(y,t)$, $j=1,2,3$ are real-valued functions from the expansion of $M(z)$ at $z=i$.
\item The singularity:
\begin{align}
&M(z)=\frac{i\widetilde{\alpha}_+}{2(z-1)}\left(
                                           \begin{array}{cc}
                                             -c & 1 \\
                                             -c & 1 \\
                                           \end{array}
                                         \right)+O(1),&&z\to1,~ \Im z>0,\\
&M(z)=-\frac{i\widetilde{\alpha}_+}{2(z+1)}\left(
                                           \begin{array}{cc}
                                             -c & 1 \\
                                             c & -1 \\
                                           \end{array}
                                         \right)+O(1),&&z\to-1, ~\Im z<0.
\end{align}
\item Residue conditions: The matrix $M(z)$ has simple poles at $\eta_n$ and $\overline{\eta}_n$ with
\begin{subequations}\label{res-M}
\begin{align}
&\mathop{\res}\limits_{\eta_n} M(z)=\lim_{z\to\eta_n} M (z)\left(
                                                  \begin{array}{cc}
                                                    0 & 0 \\
                                                    c_ne^{2it\theta(\eta_n)} & 0 \\
                                                  \end{array}
                                                \right), \label{res1} \\
&\mathop{\res}\limits_{\overline{\eta}_n} M (z)=\lim_{z\to\overline{\eta}_n} M (z)\left(
                                                  \begin{array}{cc}
                                                    0 & \overline{c}_ne^{-2it\overline(\bar{\eta}_n)} \\
                                                    0 & 0 \\
                                                  \end{array}\right).
\end{align}
\end{subequations}
\item The solution to the mCH equation \eqref{mCH-1} is expressed by
\begin{align}\label{q-sol}
\begin{split}
q(x,t)=q(y(x,t),t)=&-\partial_zM_{12}(z)|_{z=i}M_{11}(i)-
\partial_zM_{21}(z)|_{z=i}M_{11}(i)^{-1},\\
&x(y,t)=y+\ln M_{11}(i).
\end{split}
\end{align}
\end{enumerate}
\end{RHP}

It follows from the results \cite{bij-1998,YF-AM,YFL-2022} that the following map is Lipschitz:
\begin{align}
\mathcal {D}: H^{2,1}(\D{R})\ni q_0(x)\mapsto r(z)\in H^{1,1}(\D{R}),
\end{align}
where some spaces are defined as follows:
\begin{enumerate}[(I)]
\item The Sobolev space: \begin{align}\label{S1}
W^{j,k}(\mathbb{R})=\left\{f(x)\in L^k(\mathbb{R}):\partial_x^mf(x)\in L^k(\mathbb{R}),
~m=1,\ldots,j~\right\}.
\end{align}
\item The weighted Sobolev space
\begin{align}\label{WS}
H^{j,k}(\mathbb{R})=\left\{f(x)\in L^2(\mathbb{R}):\partial_x^jf(x)\in L^{2}(\mathbb{R}),
 |x|^kf(x)\in L^{2}(\mathbb{R})\right\}.
\end{align}
\item The weighted $L^{p}(\mathbb{R})$ space
\begin{align}\label{WL}
L^{p,s}(\mathbb{R})=\left\{f(x)\in L^p(\mathbb{R}):|x|^sf(x)\in L^p(\mathbb{R})\right\},
\end{align}
\item  Define $A\lesssim B$ for any quantities $A$ and $B$, if there exists a constant $c>0$
 such that $|A|\leq c|B|$.
Additionally, the norm of $f(x)\in L^{p}(\mathbb{R})$ $(L^{p,s}(\mathbb{R}))$
are abbreviated to $||f||_p$ $(||f||_{p,s})$, respectively.
\end{enumerate}

\section{Distribution of phase points and symbol table}\label{sec3}
The jump matrix $V(z)$ in \eqref{Jump-M} is of the triangular decomposition
\begin{align}\label{Jump-de}
V(z)=\left\{\begin{aligned}
&\left(
  \begin{array}{cc}
    1 & r(z)e^{-2it\theta(z)} \\
    0 & 1 \\
  \end{array}
\right)\left(
         \begin{array}{cc}
           1 & 0 \\
           -\overline{r}(z)e^{-2it\theta(z)} & 1 \\
         \end{array}
       \right), && z\in\Sigma_a(\xi),   \\
&\left(
   \begin{array}{cc}
     1 & 0 \\
     -\frac{\overline{r}(z)}{1-|r(z)|^2}e^{2it\theta(z)} & 1 \\
   \end{array}
 \right)(1-|r(z)|^2)^{\sigma_3}\left(
                 \begin{array}{cc}
                   1 & \frac{r(z)}{1-|r(z)|^2}e^{-2it\theta(z)} \\
                   0 & 1 \\
                 \end{array}
               \right),&& z\in\Sigma_b(\xi).
               \end{aligned}\right.
\end{align}

Following the basic idea of the nonlinear steepest descent method \cite{DZ-AM1993,Deift-1994}, the triangular decomposition \eqref{Jump-de} of the jump matrix is used so that the (oscillation) jump matrix of the modified RH problem on steepest descent line can be exponentially attenuated to the identity matrix in appropriate regions as $t\to+\infty$. Whereas,
the analysis of the oscillation RH problem is mainly based on the growth and decay properties of the exponential functions $e^{\pm2it\theta(z)}$ appearing in the jump matrix $V(z)$ \eqref{Jump-M} and residue conditions \eqref{res-M}, and whether there is a phase point also affects the leading term of the solution. As  indicated in \cite{mCH-NZBCs}, the structure of $\Sigma_a(\xi)$ and $\Sigma_b(\xi)$ is affected by the ranges of values of $\xi$. That is, the four value ranges of $\xi$ can be distinguished, $\Sigma_a(\xi)$ and $\Sigma_b(\xi)$ have different structures, which means four different types of large time asymptotics: ($i$) $\xi>2$,
($ii$) $0<\xi<2$,
($iii$) $-\frac{1}{4}<\xi<0$,
($iv$) $\xi<-\frac{1}{4}$.

Thus, we now analyze the exponential function in the jump matrix. Observing that the phase function $\theta(z)$ defined by \eqref{Jump-M}
\begin{align}\label{phe}
\theta(z)=-\frac{i(z^2-1)}{4z}\left(-y+\frac{8t}{(z^2+1)^2}\right).
\end{align}
Let $\xi=\frac{y}{t}$, then the imaginary part of the phase function can be shown in Fig.\ref{Fig1}.
The stationary points of phase function $\theta(z)$ are determined by $\frac{d\theta(z)}{dz}=0$, where
\begin{align}\label{theta-dao}
\frac{d\theta(z)}{dz}=\frac{\xi z^8+4\xi z^6+8z^6+6\xi z^4-48z^4+4\xi z^2+8z^2+\xi}{z^2(z^2+1)^3}.
\end{align}
\begin{prop}\label{pp-point}
The  distribution of stationary points for the phase function $\theta$ is based on the value range of $\xi=\frac{y}{t}$:
\begin{enumerate}[$($i$)$]
\item As $\xi>2$ and $\xi<-1/4$, there is no stationary points on the real axis $\D{R}$.
\item As $0<\xi<2$, there are four stationary points on the real axis  $\D{R}$, which are respectively recorded as $\xi_1$, $\xi_2$, $\xi_3$, and $\xi_4$, and satisfy $\xi_4<\xi_3<\xi_2<\xi_1$ as well as $\xi_1>1$, $\xi_4\xi_3=1$,
$\xi_2\xi_1=1$ and $\xi_1\xi_3=-1$ shown in Fig. \ref{Fig1}.
\item As $-1/4<\xi<0$, there are eight stationary points on the real axis $\D{R}$, which are respectively recorded as $\xi_1$, $\xi_2$, $\xi_3$, $\ldots$, $\xi_8$ satisfying
$$\xi_8<\ldots<\xi_2<\xi_1(>1), \quad \xi_1\xi_4=\xi_5\xi_8=-\xi_2\xi_6=-\xi_3\xi_7=1,\quad \xi_j=-\xi_{9-j},~j=1,2,3,4.$$
\end{enumerate}
This results lead to the structure of $\Sigma_b(\xi)$ expressed as
\begin{align}
\Sigma_b(\xi)=\left\{\begin{aligned}
&\emptyset, &&\xi>2,\\
&(\xi_4,\xi_3)\cup(\xi_2,\xi_1), &&0<\xi<2,\\
&(-\infty,\xi_8)\cup(\xi_7,\xi_6)\cup(\xi_5,\xi_4)\cup(\xi_3,\xi_2)\cup(\xi_1,+\infty),&&-\frac{1}{4}<\xi<0,\\
&(-\infty,+\infty),&&\xi<-\frac{1}{4},
\end{aligned}\right.
\end{align}
and $\Sigma_a(\xi)=\D{R}\setminus\Sigma_b(\xi).$
\end{prop}
\begin{proof}
It follows from $\frac{d\theta(z)}{dz}=0$ \eqref{theta-dao} that
\begin{align}
\xi z^8+4\xi z^6+8z^6+6\xi z^4-48z^4+4\xi z^2+8z^2+\xi=0,
\end{align}
from which one obtains
\begin{align}
\xi=\frac{2-8z^2}{(1+4z^2)^2}.
\end{align}
As in \cite{BD-2008,BD-AMSc2008,mCH-NZBCs}, the distribution of corresponding steady-state phase points in four cases can be calculated, respectively.
\end{proof}
\begin{rem}
As $\xi\in(-\infty,-1/4)\cup(2,+\infty)$, there is no steady-state phase point. Whereas, as $\xi\in(-1/4,0)\cup(0,2)$, there exists eight and four phase points, respectively. Each phase points have different properties. Specifically,  as $\xi\in(0,2)$, $\theta''(\xi_j)>0$ for $j=1,3$ and $\theta''(\xi_j)<0$ for $j=2,4$. As $\xi\in(-1/4,0)$, $\theta''(\xi_j)>0$ for $j=2,4,6,8$ and $\theta''(\xi_j)<0$ for $j=1,3,5,7$, which shown in Fig.\ref{fa}.
For convenience, we introduce the following notation
\begin{align}\label{notation-eta}
\eta_j=\text{sgn}(\theta''(\xi_j))=\left\{\begin{aligned}
&1, & j=1,3\ \text{as}\ \xi\in(0,2)\ \text{and}\ j=2,4,6,8\ \text{as}\ \xi\in(-1/4,0),\\
&-1,& j=2,4\ \text{as}\ \xi\in(0,2)\ \text{and}\ j=1,3,5,7\ \text{as}\ \xi\in(-1/4,0),
\end{aligned}\right.
\end{align}

\begin{figure}[H]
	\centering
\subfigure[]{
\includegraphics[width=0.38\linewidth]{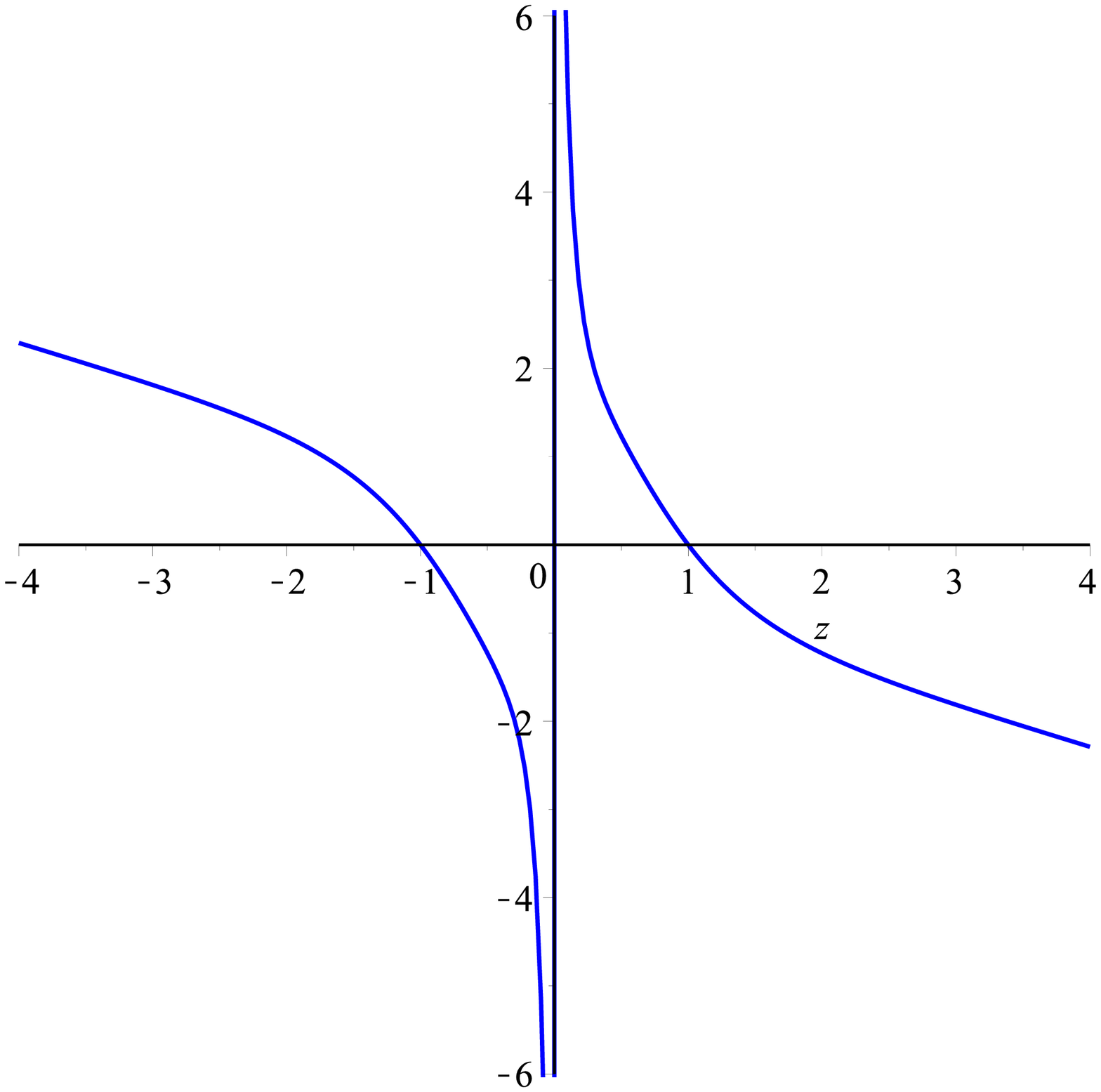}\hspace{0.6cm}\label{F1}}
\subfigure[]{
\includegraphics[width=0.38\linewidth]{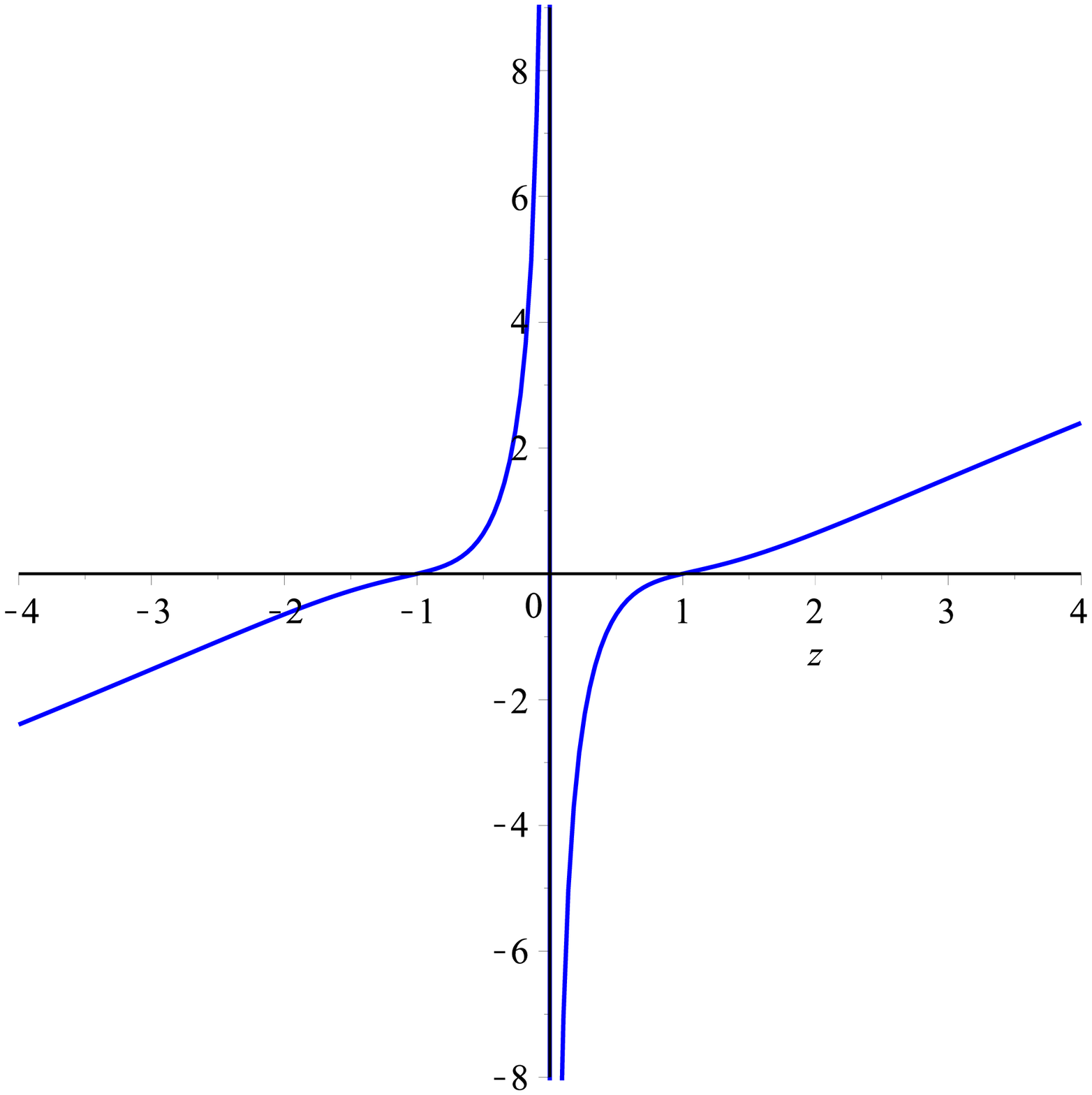}\label{F2}}
\end{figure}

\begin{figure}[H]
	\centering
\subfigure[]{
\includegraphics[width=0.38\linewidth]{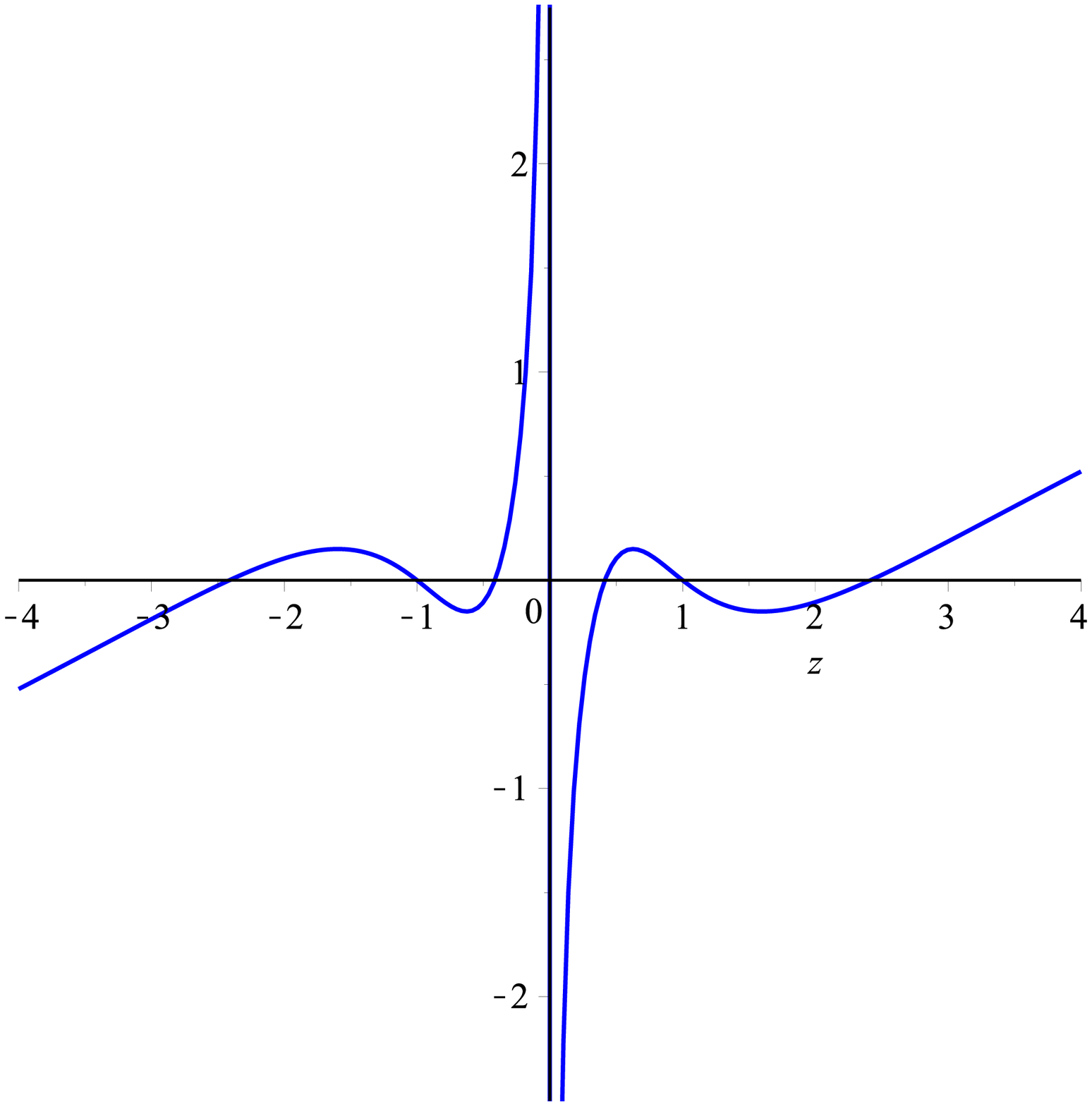}\hspace{0.6cm}\label{F3}}
\subfigure[]{
\includegraphics[width=0.38\linewidth]{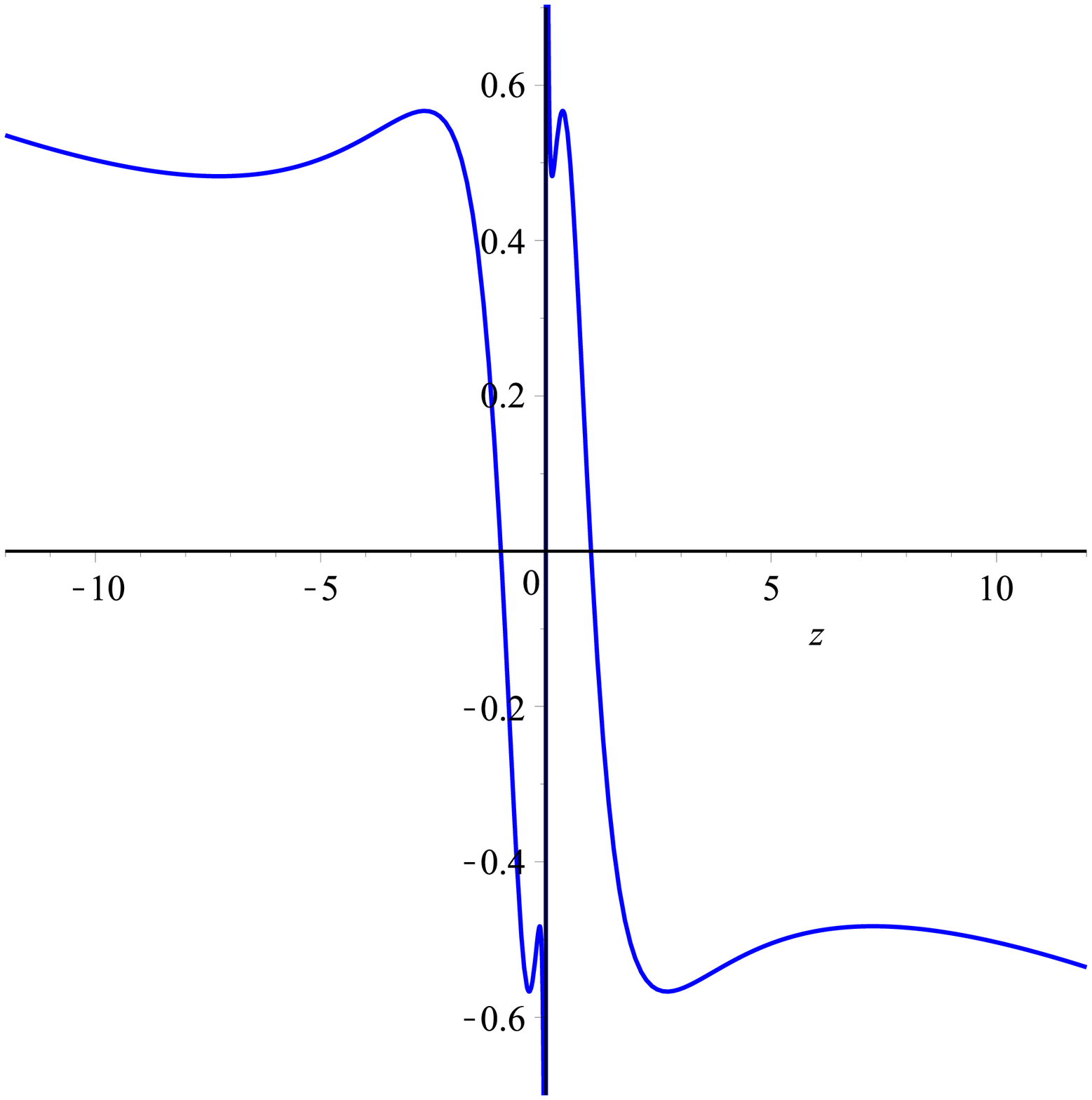}\label{F4}}
	\caption{Variation diagram of phase function $\theta$ with $\xi$ value range. ($a$) and ($b$) correspond to $\xi\in(-\infty,-1/4)\cup(2,+\infty)$. ($c$) corresponds to $\xi=1\in(0,2)$. ($d$) corresponds to $\xi=-1/8\in(-1/4,0)$.}
	\label{fa}
\end{figure}

\end{rem}

\begin{figure}[H] 
\centering  
\subfigbottomskip=12pt 
\subfigcapskip=2pt 
\subfigure[$\xi>2$]{
\includegraphics[width=0.38\linewidth]{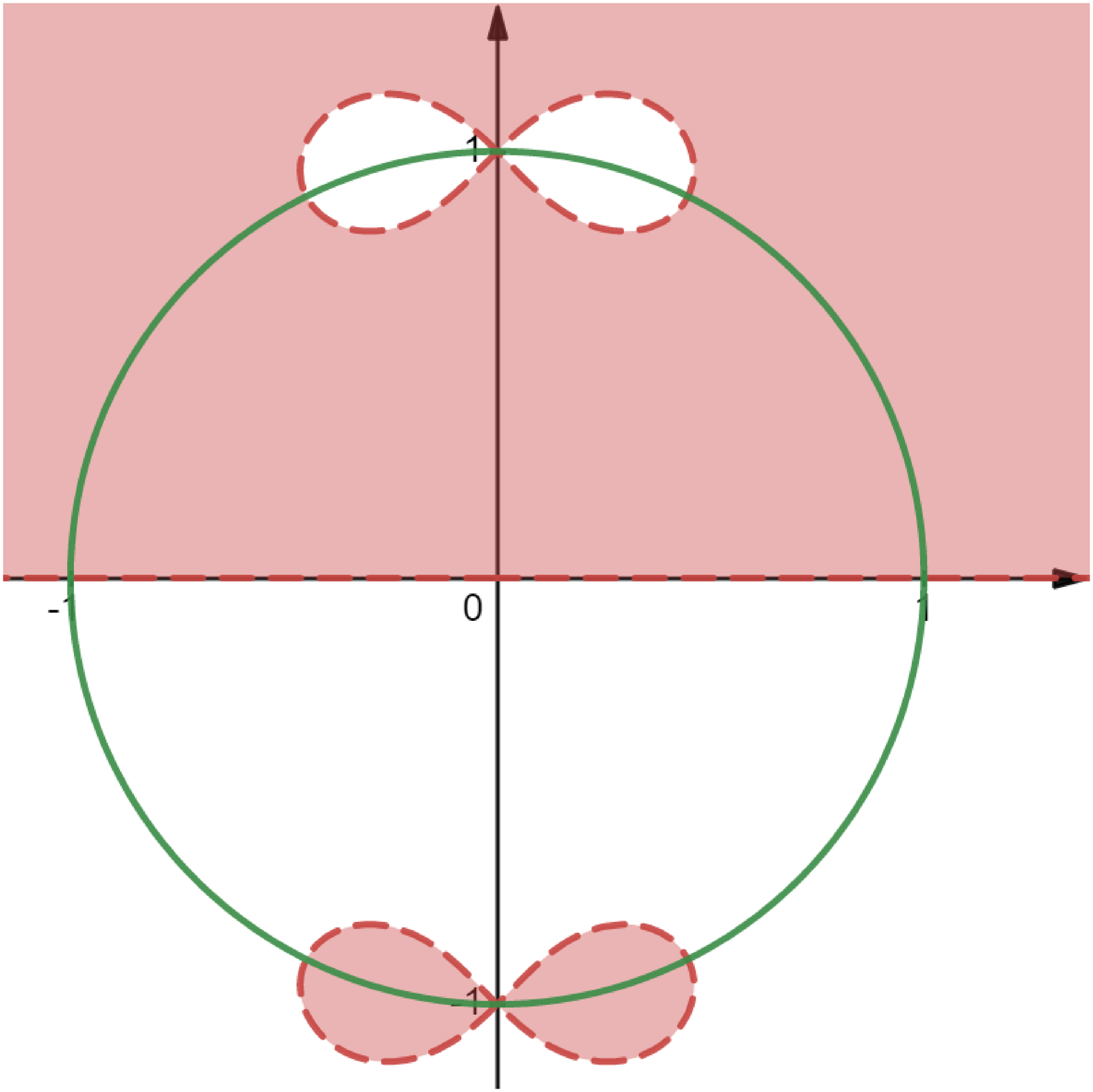}\hspace{0.6cm}\label{Fig-2a}}
\subfigure[$\xi<-1/4$]{
\includegraphics[width=0.38\linewidth]{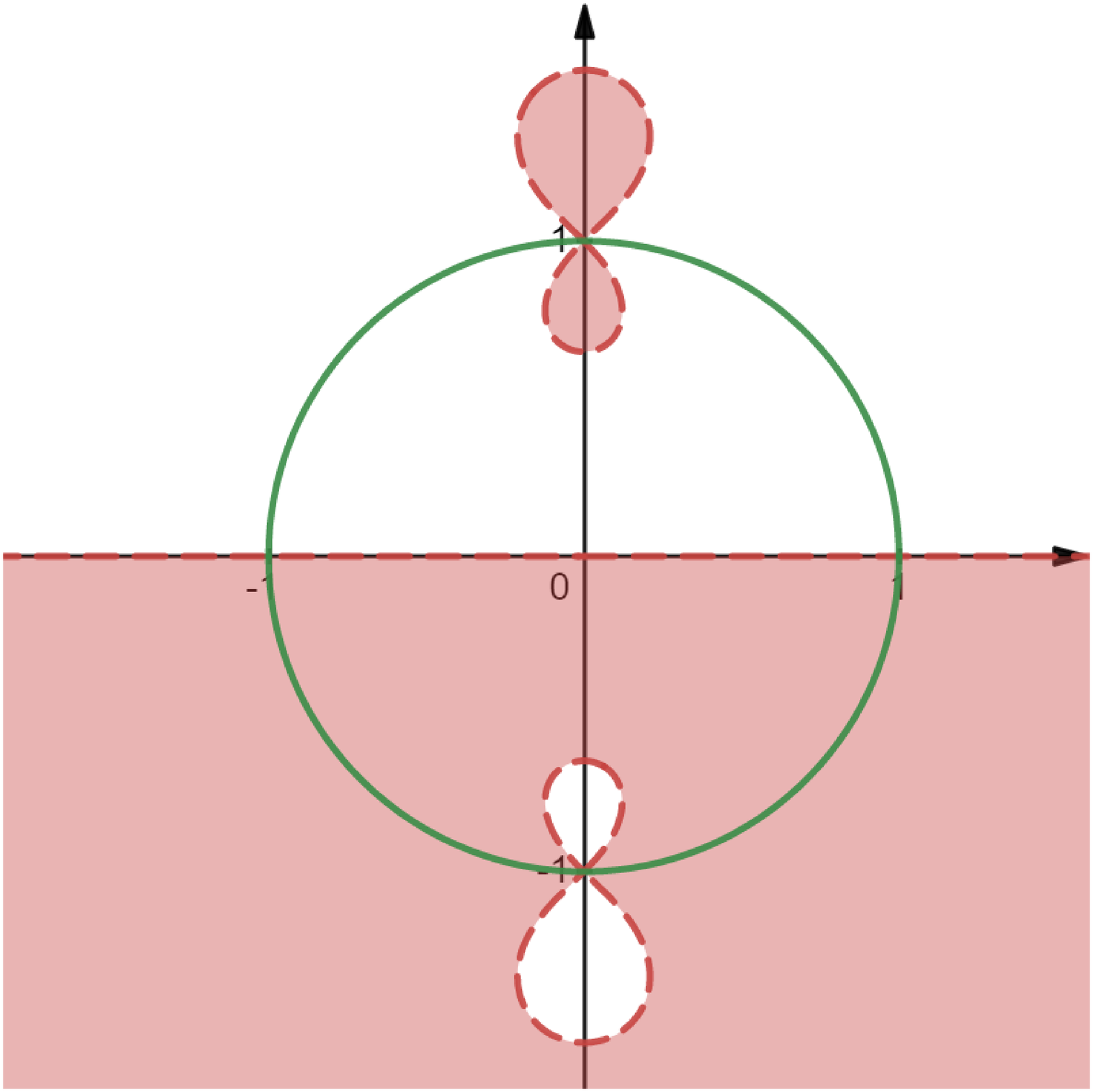}\label{Fig-2b}}\\
\subfigure[$0<\xi<2$]{
\includegraphics[width=0.38\linewidth]{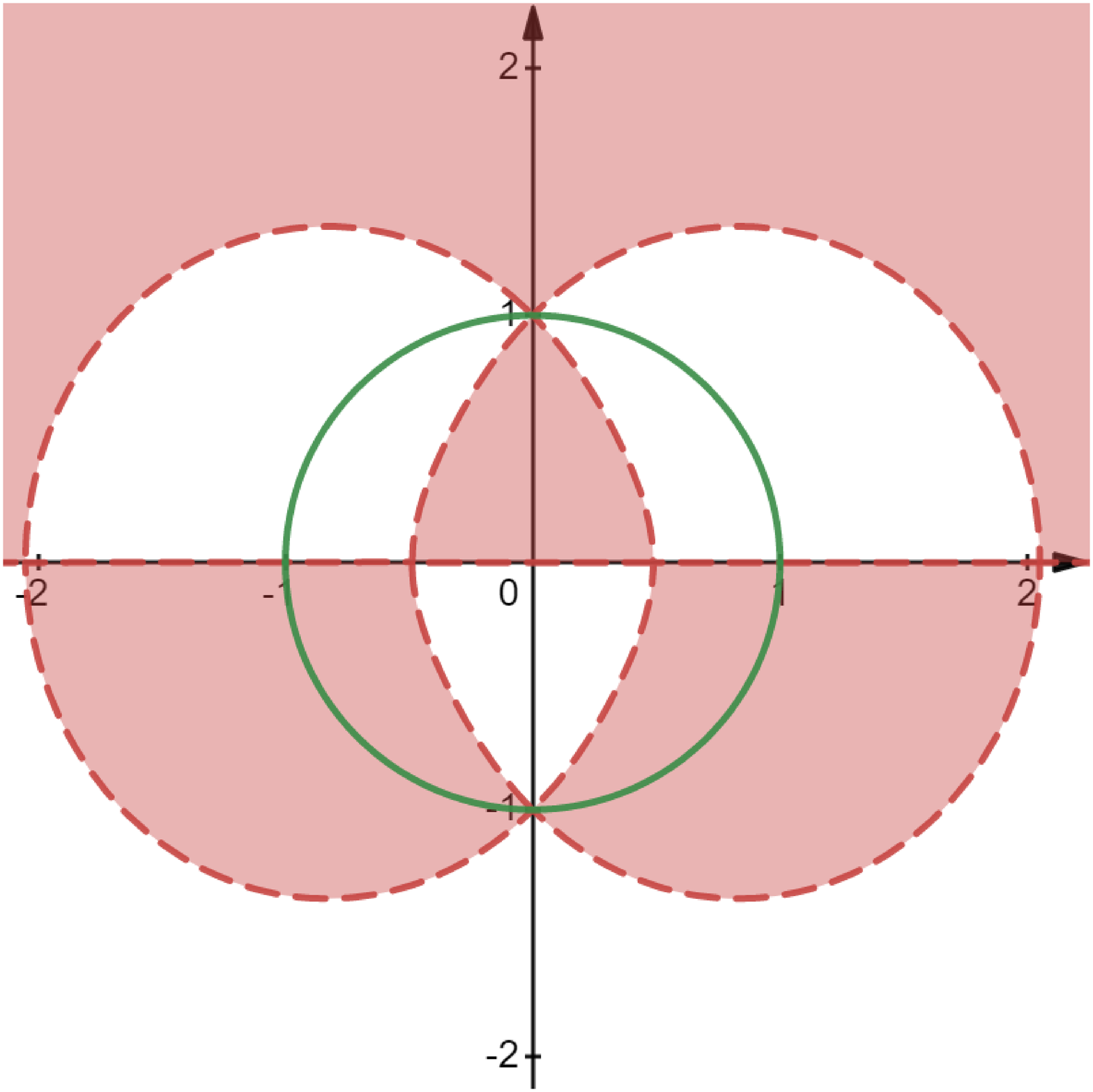}\hspace{0.6cm}\label{Fig-2c}}
\subfigure[$-1/4<\xi<0$]{
\includegraphics[width=0.38\linewidth]{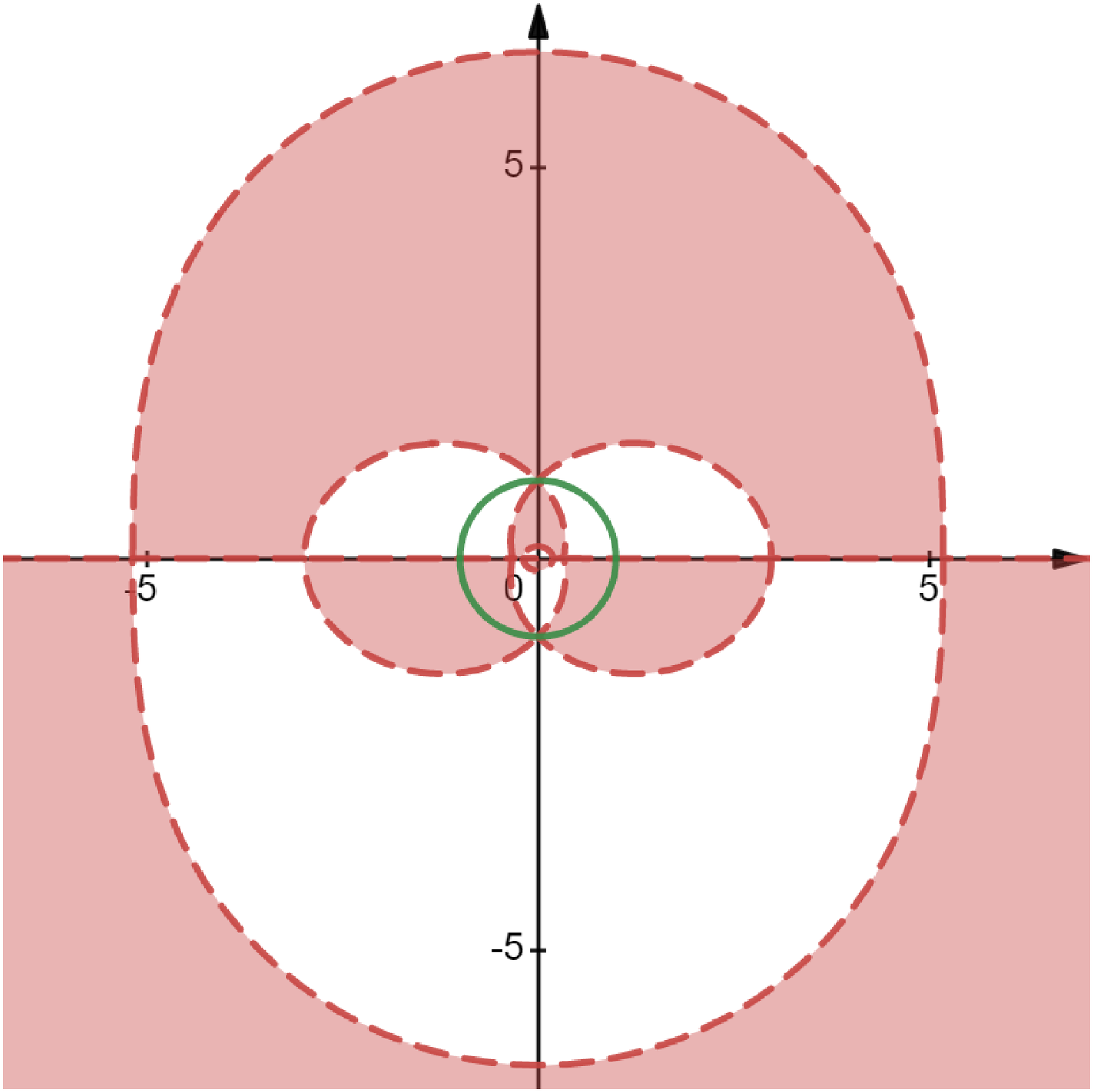}\label{Fig-2d}}
\caption{Plots of the imaginary part $\Im\theta$ of phase function $\theta$: (a) $\xi>2$. (b) $\xi<-1/4$.
(c) $0<\xi<2$. (d) $-1/4<\xi<0$, where $\Im\theta>0$ in red region and $\Im\theta<0$ in white region. }\label{Fig1}
\end{figure}

$\blacktriangleright$ For $\xi>2$ and $\xi<-\frac{1}{4}$, we will study the long-time asymptotic behavior of mCH equation \eqref{mCH-2}, and prove that its solution can be characterized by $N$-soliton solution and error term shown in Section \ref{sec4}.

$\blacktriangleright$ For $0<\xi<2$ and $-\frac{1}{4}<\xi<0$, which correspond to four and eight phase points respectively, we prove that the soliton resolution conjecture is valid for mCH equation  \eqref{mCH1-ic} with non-zero boundary, that is, when the time tends to infinity, the solution can be characterized by the soliton solution corresponding to the discrete spectrum, the leading term corresponding to the continuous spectrum and error term  corresponding to pure $\overline{\partial}$-problem.

In order to facilitate the later analysis, the sign of the number of steady-state phase points is introduced by
\begin{align}\label{number}
n(\xi)=\left\{\begin{aligned}
&0,&&\xi>2,\xi<-\frac{1}{4},\\
&4,&&0<\xi<2,\\
&8,&&-\frac{1}{4}<\xi<0,
\end{aligned}\right.
\end{align}
and define $N=\{1,2,\ldots,4N_1+2N_2\}$, and fix a sufficiently small positive number $\varepsilon_0$ to satisfy
\begin{align*}
\Delta=\{n\in N|\Im\theta<0\},~~\nabla=\{n\in N|\Im\theta>0\},\
\Lambda=\{n\in N||\Im\theta|\leq\varepsilon_0\}.
\end{align*}
Observing that the oscillatory term $e^{2it\theta}$ appearing in residue condition \eqref{res1} grows as $t\to+\infty$ for $\eta_n$ with $n\in\Delta$, also for $\eta_n$ with $n\in\nabla$ the residue condition \eqref{res1} decays as $t\to+\infty$.

\section{Deformation of RH problem for $\xi>2$ and $\xi<-1/4$}\label{sec4}
In this section, we turn our attention to $\xi>2$ and $\xi<-1/4$ corresponding to Figures \ref{Fig-2a} and \ref{Fig-2b}, which corresponds to the case where there is no steady-state phase point. In order to perform a long-time asymptotic analysis through the nonlinear steepest descent method, we need to deal with the original RH problem with the following conditions:
\begin{enumerate}[(i)]
\item Interpolating poles are transformed to jump conditions on each pole circle by trading their residues for $\eta_n$ with $n\in\Lambda$.
\item The jump of the original RH problem \ref{RHP-M} on the real axis $\D{R}$ is traded to the exponentially decaying contour.
\end{enumerate}

In order to eliminate the diagonal matrix in \eqref{Jump-de} in the triangular decomposition on interval $\Sigma_b(\xi)$, as \cite{mCH-NZBCs}, the scalar RH problem satisfied by $\delta(z):=\delta(z,\xi)$, which is analytic in $\D{C}\setminus\overline{\Sigma_b(\xi)}$ needs to be introduced
\begin{subequations}\label{del}
\begin{align}
&\delta_+(z)=\delta_-(z)(1-|r(z)|^2),&&z\in\Sigma_b(\xi),\\
&\delta(z)\to I,&&z\to\infty.
\end{align}
\end{subequations}
The solution to \eqref{del} can be derived by Pelemlj formula
\begin{align}\label{del-ex}
\delta(z)=\exp\left\{\frac{1}{2\pi i}\int_{\Sigma_b(\xi)}\frac{\ln(1-|r(s)|^2)}{s-z}\ ds\right\}.
\end{align}

Define some functions as
\begin{align}
\nu(z)&=-\frac{1}{2\pi}\ln(1-|r(z)|^2), ~~\delta(z)=\exp\left\{-i\int_{\Sigma_b(\xi)}
\frac{\nu(s)}{s-z}\ ds\right\},\\
T(z)&:=T(z,\xi)=\prod_{n\in\Delta}\frac{z-\eta_n}{\overline{\eta}_n^{-1}z-1}\delta(z)\nonumber\\
&=\prod_{j\in \Delta_1}\dfrac{z-z_j}{\bar{z}_j^{-1}z-1}\dfrac{z+\bar{z}_j}{z_j^{-1}z+1}\dfrac{z-\bar{z}_j^{-1}}{z_jz-1}\dfrac{z+z_j^{-1}}{\bar{z}_jz+1}\prod_{i\in \Delta_2}\dfrac{z-w_i}{w_iz-1}\dfrac{z+\overline{w}_i}{\overline{w}_iz+1}\delta (z,\xi) \label{T},
\end{align}
where the intervals $\Delta_j$, $\nabla_j$ and $\Lambda_j$ for $j=1,2$ are defined as
\begin{subequations}
\begin{align}
	&\nabla_1=\left\lbrace j \in \left\lbrace 1,\ldots,N_1\right\rbrace  \left|\Im\theta(z_j)> 0\right.\right\rbrace,
	&&\Delta_1=\left\lbrace j \in \left\lbrace 1,\ldots,N_1\right\rbrace  \left|\Im\theta(z_j)< 0\right.\right\rbrace,\\
	&\nabla_2=\left\lbrace i \in \left\lbrace 1,\ldots,N_2\right\rbrace  \left|\Im\theta(w_i)> 0\right.\right\rbrace,
	&&\Delta_2=\left\lbrace i \in \left\lbrace 1,\ldots,N_2\right\rbrace  \left|\Im\theta(w_i)< 0\right.\right\rbrace,\\
	&\Lambda_1=\left\lbrace j_0 \in \left\lbrace 1,\ldots,N_1\right\rbrace  \left||\Im\theta(z_{j_0})|\leq\delta_0\right.\right\rbrace,&&\Lambda_2=\left\lbrace i_0 \in \left\lbrace 1,\ldots,N_2\right\rbrace  \left||\Im\theta(w_{i_0})|\leq\delta_0\right.\right\rbrace.
\end{align}
\end{subequations}
\begin{prop}\label{PP-T}
The function $T(z)$ defined by \eqref{del} such that the following conditions
\begin{enumerate}[(i)]
\item The function $T(z)$ is meromorphic in $\mathbb{C}\setminus \mathbb{R}$, and for $n\in\Delta$, $T(z)$ has a simple zeros at $\eta_n$ and simple poles at $\overline{\eta}_n$.
\item $T(z)=\overline{T(-\bar{z})}=T(-z^{-1})=\overline{T^{-1}(\bar{z})}$.
\item For $z\in \Sigma_b(\xi)$, the non-tangential boundary values $T_\pm(z):=T_{\pm}(z,\xi)$, which satisfy:
\begin{equation}
T_+(z)=(1-|r(z)|^2)T_-(z),\hspace{0.5cm}z\in \Sigma_b(\xi).
\end{equation}
\item $\lim\limits_{z\to \infty}T(z)\triangleq T(\infty)$ with $T(\infty)=1$.
\item As $z=i$,
\begin{equation}
T(i)=\prod_{j\in \Delta_1}\left( \dfrac{i-z_j}{i-\overline{z}_j}\dfrac{i+\overline{z}_j}{i+z_j}\right) \prod_{h\in \Delta_2}\dfrac{i-w_h}{i+w_h}\dfrac{i+\overline{w}_h}{i-\overline{w}_h}\delta (i,\xi)\label{Ti},
\end{equation}
and as $z\to i$, $T(z)$ has asymptotic expansion  as
\begin{equation}
T(z)=T(i)\left( 1-T_0(\xi)(z-i)\right) + \mathcal{O}((z-i)^2) \label{expT0},
\end{equation}
with
\begin{equation}
T_0(\xi)=\frac{1}{2\pi i}\int _{\Sigma_b(\xi)}\dfrac{ \ln (1-|r(s)|^2)}{(s-i)^2}ds;
\end{equation}
\item $T(z)$ is continuous at $z=0$, and
\begin{equation}
T(0)=T(\infty)=1 \label{T0}.
\end{equation}
\item As $z\to \xi_j$ along any ray $\xi_j+e^{i\phi}\mathbb{R}^+$ with $|\phi|<\pi$,
\begin{align}
|T(z,\xi)-T_j(\xi)(z-\xi_j)^{i\nu(\xi_j)}|\lesssim \parallel r\parallel_{H^{1}(\mathbb{R})}|z-\xi_j|^{1/2},\label{T-TJ}
\end{align}
where $T_j(\xi) $ are determined by
\begin{align}
T_j(\xi)=\prod_{n\in \Delta}\dfrac{z-\eta_n}{\overline{\eta}_n^{-1}z-1}e^{i\beta(\xi_j,\xi)},
\end{align}
for  $j=1,\ldots,n(\xi)$ and
\begin{align}
\beta_j(z,\xi)=\int_{\Sigma_b(\xi)}\frac{\nu(s)}{s-z}ds-\ln(z-\xi_j)\nu(\xi_j).
\end{align}
\end{enumerate}
\end{prop}
\begin{proof}
Simple calculations show that the claims ($i$)-($vi$) are hold. Next, we will prove the claim ($vii$). Observing that the solution to \eqref{del} can be written as
\begin{align}\label{del-gx}
\delta(z)=\exp\left(i\beta_j(z,\xi)+\nu(\xi_j)\ln(z-\xi_j)\right).
\end{align}
It also follows from \cite{AIHP} that
$$\left|(z-\xi_j)^{i\nu(\xi_j)}\right|\leq e^{-\pi\nu(\xi_j)}.$$
Combining the function $T(z)$ with \eqref{del-gx} yields that
\begin{align}\label{T-gx}
T(z)=\prod_{n\in\Delta}\frac{z-\eta_n}{\overline{\eta}_n^{-1}z-1}(z-\xi_j)^{i\nu(\xi_j)}\exp(i\beta_j(z,\xi_j)).
\end{align}
Then one has
\begin{align}\label{beta-est}
\left|\beta_j(z,\xi)-\beta_j(z,\xi_j)\right|\leq\left|\nu(\xi_j)\ln(z-\xi_j)\right|+\left|
\int_{\Sigma_b(\xi)}\frac{\nu(s)}{s-z}\ ds-\int_{\Sigma_b(\xi)}\frac{\nu(s)}{s-\xi_j}\ ds\right|.
\end{align}
For the first term of \eqref{beta-est},
\begin{align}
\nu(\xi_j)\ln(z-\xi_j)=\nu(\xi_j)(z-\xi_j)+O((z-\xi_j)^2).
\end{align}
For the second term of \eqref{beta-est}, applying H\"{o}lder inequality obtains
\begin{align}
&\left|
\int_{\Sigma_b(\xi)}\frac{\nu(s)}{s-z}\ ds-\int_{\Sigma_b(\xi)}\frac{\nu(s)}{s-\xi_j}\ ds\right|\leq
2\pi\left|C\nu(z)-C\nu(\xi_j)\right|=2\pi\left|\int_{\xi_j}^{z}C\nu'(s)\ ds\right|\nonumber\\
&\leq||C\nu'||_{L^{2}}|z-\xi_j|^{1/2}\lesssim ||r(z)||_{H^{1}(\D{R})}|z-\xi_j|^{1/2},
\end{align}
where $C$ denotes the Cauchy operator, which implies that the proof of the proposition is completed.
\end{proof}

In order to transform the residue conditions at the poles  into jump conditions on sufficient small disks, a positive constant $\epsilon$ is introduced
\begin{equation}
	\varrho=\frac{1}{2}\min\left\lbrace \min_{j\in N}\left\lbrace |\Im\eta_j|, \right\rbrace ,\min_{j\in N\setminus\Lambda,\Im\theta(z)=0}|\eta_j-z|, \min_{i\neq j\in N}|\eta_i-\eta_j|, \min_{  j\in N}|\eta_j-i|\right\rbrace ,
\end{equation}
from which we define the disks $\D{P}_n=\D{P}(\eta_n,\epsilon)$ and $\overline{\D{P}}_n=\D{P}(\overline{\eta}_n,\epsilon)$.
Observe that for any $n\in N$, the disks $\D{P}_n$ are pairwise disjoint, and also disjoint with $\{z\in\D{C}|\Im\theta(z)=0\}$ and the real axis $\D{R}$. The interpolation matrix $F(z)$ is further introduced as follows
\begin{equation}\label{inte-F}
F(z)=\left\{ \begin{array}{ll}
		\left(\begin{array}{cc}
			1 & 0\\
			-C_n(z-\eta_n)^{-1}e^{2it\theta_n} & 1
		\end{array}\right),   &\text{as } z\in\mathbb{P}_n,n\in\nabla\setminus\Lambda,\\[12pt]
		\left(\begin{array}{cc}
			1 & -C_n^{-1}(z-\eta_n)e^{-2it\theta_n}\\
			0 & 1
		\end{array}\right),   &\text{as } z\in\mathbb{P}_n,n\in\Delta\setminus\Lambda,\\
		\left(\begin{array}{cc}
		1 & \overline{C}_n(z-\overline{\eta}_n)^{-1}e^{-2it\overline{\theta}_n}\\
		0 & 1
		\end{array}\right),   &\text{as } 	z\in\overline{\mathbb{P}}_n,n\in\nabla\setminus\Lambda,\\
		\left(\begin{array}{cc}
		1 & 0	\\
		\overline{C}_n^{-1}(z-\overline{\eta}_n)e^{2it\bar{\theta}_n} & 1
		\end{array}\right),   &\text{as } 	z\in\overline{\mathbb{P}}_n,n\in\Delta\setminus\Lambda.\\
	I &\text{as } 	z \text{ in elsewhere};
	\end{array}\right.
\end{equation}

Applying interpolation matrix $F(z)$ \eqref{inte-F} to define a new matrix $M^{(1)}(z):=M^{(1)}(y,t,z)$ such that
\begin{align}\label{M1}
M^{(1)}(z)=M(z)F(z)T(z)^{\sigma_3},
\end{align}
which solves the following problem.
\begin{RHP}\label{RHP-M1}
Find a matrix-valued function $M^{(1)}(z)$ satisfy the following conditions:
\begin{enumerate}[(I)]
\item The function $M^{(1)}(z)$ is meromorphic in $\D{C}\setminus\Sigma^{(1)}$, where the contour $$\Sigma^{(1)}
=\D{R}\cup\left(\cup_{n\in N\setminus\Lambda}(\D{P}_{n}\cup\overline{\D{P}}_{n})\right).$$
\item The symmetries: $$M^{(1)}(z)=\overline{M^{(1)}(\bar{z}^{-1})}=
    \sigma_1\overline{M^{(1)}(\overline{z})}\sigma_1=
\sigma_3\overline{M^{(1)}(-\bar{z})}\sigma_3.$$
\item Jump condition: The non-tangential limits $M^{(1)}_{\pm}(z)=\lim\limits_{\varepsilon\to0^+}M^{(1)}(z+\pm i\varepsilon)$ for $z\in\D{R}\setminus\{\pm1\}$
\begin{align}\label{JM1}
M^{(1)}_{+}(z)=M^{(1)}_{-}(z)V^{(1)}(z),~~~~~z\in\D{R}\setminus\Sigma^{(1)},
\end{align}
where the jump matrix $V^{(1)}(z)$ is defined by
\begin{align}\label{Jump-M1}
V^{(1)}(z)=\left\{\begin{aligned}
&\left(
  \begin{array}{cc}
    1 & r(z)T^2(z)e^{-2it\theta(z)} \\
    0 & 1 \\
  \end{array}
\right)\left(
         \begin{array}{cc}
           1 & 0 \\
           -\overline{r}(z)T^{-2}(z)e^{-2it\theta(z)} & 1 \\
         \end{array}
       \right), && z\in\Sigma_a(\xi),   \\
&\left(
   \begin{array}{cc}
     1 & 0 \\
     -\frac{\overline{r}(z)}{1-|r(z)|^2}T_{-}^{-2}(z)e^{2it\theta(z)} & 1 \\
   \end{array}
 \right)\left(
                 \begin{array}{cc}
                   1 & \frac{r(z)}{1-|r(z)|^2}T_+^2(z)e^{-2it\theta(z)} \\
                   0 & 1 \\
                 \end{array}
               \right),&& z\in\Sigma_b(\xi),\\
&\left(\begin{array}{cc}
			1 & 0\\
			-C_n(z-\eta_n)^{-1}T^{-2}(z)e^{2it\theta_n} & 1
		\end{array}\right),   \qquad z\in\mathbb{P}_n,n\in\nabla\setminus\Lambda,\\[12pt]
&\left(\begin{array}{cc}
			1 & -C_n^{-1}(z-\eta_n)T^{2}(z)e^{-2it\theta_n}\\
			0 & 1
		\end{array}\right),   ~~\qquad z\in\mathbb{P}_n,n\in\Delta\setminus\Lambda,\\
&\left(\begin{array}{cc}
		1 & \overline{C}_n(z-\overline{\eta}_n)^{-1}T^{2}(z)e^{-2it\overline{\theta}_n}\\
		0 & 1
		\end{array}\right), ~~~~\qquad 	z\in\overline{\mathbb{P}}_n,n\in\nabla\setminus\Lambda,\\
&\left(\begin{array}{cc}
		1 & 0	\\
		\overline{C}_n^{-1}(z-\overline{\eta}_n)T^{-2}(z)e^{2it\bar{\theta}_n} & 1
		\end{array}\right),   \qquad\quad	z\in\overline{\mathbb{P}}_n,n\in\Delta\setminus\Lambda.
\end{aligned}\right.
\end{align}
\item Asymptotic behavior:
\begin{align}\label{Asy-M1}
&M^{(1)}(z)=I+O(z^{-1}), z\to\infty~ (\text{and}~ z\to0 ~\text{by symmetry}),\\
&M^{(1)}(z)=\left(\di(a_1(y,t),a_1^{-1}(y,t))+\left(
        \begin{array}{cc}
          0 & a_2(y,t) \\
          a_3(y,t) & 0 \\
        \end{array}
      \right)(z-i)\right)
T(i)^{\sigma_3}\left(I-T_0(\xi)(z-i)\right)
+O((z-i)^2),
\end{align}
where $a_j(y,t)$, $j=1,2,3$ are real-valued functions from the expansion of $M^{(1)}(z)$ at $z=i$.
\item Residue conditions: The matrix $M^{(1)}(z)$ has simple poles at $\eta_n$ and $\overline{\eta}_n$ with
\begin{subequations}\label{res-M1}
\begin{align}
&\mathop{\res}\limits_{\eta_n} M^{(1)}(z)=\lim_{z\to\eta_n} M^{(1)} (z)\left(
                                                  \begin{array}{cc}
                                                    0 & 0 \\
                                                    c_nT^2(z)e^{2it\theta(\eta_n)} & 0 \\
                                                  \end{array}
                                                \right), \label{resM1} \\
&\mathop{\res}\limits_{\overline{\eta}_n} M^{(1)}(z)=\lim_{z\to\overline{\eta}_n} M^{(1)} (z)\left(
                                                  \begin{array}{cc}
                                                    0 & \overline{c}_nT^{-2}(z)e^{-2it\overline{\theta}(\eta_n)} \\
                                                    0 & 0 \\
                                                  \end{array}\right).
\end{align}
\end{subequations}
\end{enumerate}
\end{RHP}

\subsection{Contour deformation}\label{subsec4.1}
In this section, our attention will turn to removing jump conditions of $M^{(1)}(z)$ on the real axis $\D{R}$
 taking  into account the fact that the growth and decay properties of  the exponential functions $e^{\pm2it\theta}$.

To that end, A fixed angle $\psi$ is sufficiently small to satisfy the following conditions
\begin{enumerate}[(i)]
\item For $\xi>2$, then $\cos2\psi>\frac{4}{\xi}-1$.
\item For $\xi<-\frac{1}{4}$, then $\cos2\psi>-\frac{1}{\xi}-1$.
\end{enumerate}
For the regions $\xi<-\frac{1}{4}$ and $\xi>2$, we define the $\Omega=\cup_{1}^4\Omega_k$, where
\begin{align*}
	&\Omega_{2n+1}=\left\lbrace z\in\mathbb{C}|n\pi \leq\arg z \leq n\pi+\psi \right\rbrace ,\\
	&\Omega_{2n+2}=\left\lbrace z\in\mathbb{C}|(n+1)\pi -\psi\leq\arg z \leq (n+1)\pi \right\rbrace,
\end{align*}
with $n=0,1$, which such that the boundary values
\begin{align}
&\Sigma_k=e^{(k-1)i\pi/2+i\psi}\D{R}^+,\hspace{0.5cm}k=1,3,\\
&\Sigma_k=e^{ki\pi/2-i\psi}\D{R}^+,\hspace{1cm}k=2,4,
\end{align}
shown in Fig. \ref{Fig2}. Let
\begin{align}
\Sigma^{(2)}(\xi)=\underset{n\in N\setminus\Lambda}{\cup}\left( \partial\overline{\mathbb{P}}_n\cup\partial\mathbb{P}_n\right).
\end{align}

\begin{figure}[H]
\centering
\begin{tikzpicture}[node distance=2cm]
\draw[fill=pink] (0,0)--(3,0)--(3,0.5)--(0,0)--(-3,0)--(-3,0.5)--(0,0);
\draw[fill=yellow](0,0)--(-3,0)--(-3,-0.5)--(0,0)--(3,0)--(3,-0.5)--(0,0);
\draw(0,0)--(3,0.5)node[above right]{$\Sigma_1$};
\draw(0,0)--(-3,0.5)node[above left]{$\Sigma_2$};
\draw(0,0)--(-3,-0.5)node[left]{$\Sigma_3$};
\draw(0,0)--(3,-0.5)node[right]{$\Sigma_4$};
\draw[line width=0.8](-4,0)--(4,0)node[right]{ $\Re z$};
\draw[line width=0.8](0,-3)--(0,3)node[above]{ $\Im z$};
\draw[-latex](4,0)--(4.1,0);
\draw[-latex](0,3)--(0,3.1);
\draw[-latex](0,0)--(-1.5,-0.25);
\draw[-latex](0,0)--(-1.5,0.25);
\draw[-latex](0,0)--(1.5,0.25);
\draw[-latex](0,0)--(1.5,-0.25);
\coordinate (C) at (-0.2,2.2);
\coordinate (D) at (2.2,0.2);
\fill (D) circle (0pt) node[right] {\footnotesize $\Omega_1$};
\coordinate (J) at (-2.2,-0.2);
\fill (J) circle (0pt) node[left] {\footnotesize $\Omega_3$};
\coordinate (k) at (-2.2,0.2);
\fill (k) circle (0pt) node[left] {\footnotesize $\Omega_2$};
\coordinate (k) at (2.2,-0.2);
\fill (k) circle (0pt) node[right] {\footnotesize $\Omega_4$};
\coordinate (I) at (0.2,0);
\fill (I) circle (0pt) node[below] {$0$};
\draw[red] (2,0) arc (0:360:2);
\draw[blue] (2,3) circle (0.12);
\coordinate (A) at (2,3);
\coordinate (B) at (2,-3);
\coordinate (C) at (-0.5546996232,0.8320505887);
\coordinate (D) at (-0.5546996232,-0.8320505887);
\coordinate (E) at (0.5546996232,0.8320505887);
\coordinate (F) at (0.5546996232,-0.8320505887);
\coordinate (G) at (-2,3);
\coordinate (H) at (-2,-3);
\coordinate (I) at (2,0);
\draw[blue] (2,-3) circle (0.12);
\draw[blue] (-0.55469962326,0.8320505887) circle (0.12);
\draw[blue] (0.5546996232,0.8320505887) circle (0.12);
\draw[blue] (-0.5546996232,-0.8320505887) circle (0.12);
\draw[blue] (0.5546996232,-0.8320505887) circle (0.12);
\draw[blue] (-2,3) circle (0.12);
\draw[blue] (-2,-3) circle (0.12);
\coordinate (J) at (1.7320508075688774,1);
\coordinate (K) at (1.7320508075688774,-1);
\coordinate (L) at (-1.7320508075688774,1);
\coordinate (M) at (-1.7320508075688774,-1);
\fill (A) circle (1pt) node[right] {$z_n$};
\fill (B) circle (1pt) node[right] {$\overline{z}_n$};
\fill (C) circle (1pt) node[left] {$-\frac{1}{z_n}$};
\fill (D) circle (1pt) node[left] {$-\frac{1}{\overline{z}_n}$};
\fill (E) circle (1pt) node[right] {$\frac{1}{\overline{z}_n}$};
\fill (F) circle (1pt) node[right] {$\frac{1}{z_n}$};
\fill (G) circle (1pt) node[left] {$-\overline{z}_n$};
\fill (H) circle (1pt) node[left] {$-z_n$};
\fill (I) circle (1pt) node[above] {$1$};
\fill (J) circle (1pt) node[right] {$w_m$};
\fill (K) circle (1pt) node[right] {$\overline{w}_m$};
\fill (L) circle (1pt) node[left] {$-\overline{w}_m$};
\fill (M) circle (1pt) node[left] {$-w_m$};
\end{tikzpicture}
\caption{The unknown matrix-valued function $M^{(2)}(z)$ has nonzero $\overline{\partial}$-derivatives in $\Omega_k$, and has jump conditions on $\D{P}_n\cup\overline{\D{P}}_n$. $M^{(2)}(z)$ is continuous on  the contour $\Sigma_k$, where  $\Sigma_k$ denote the boundary values of $\Omega_k$. }\label{Fig2}
\end{figure}
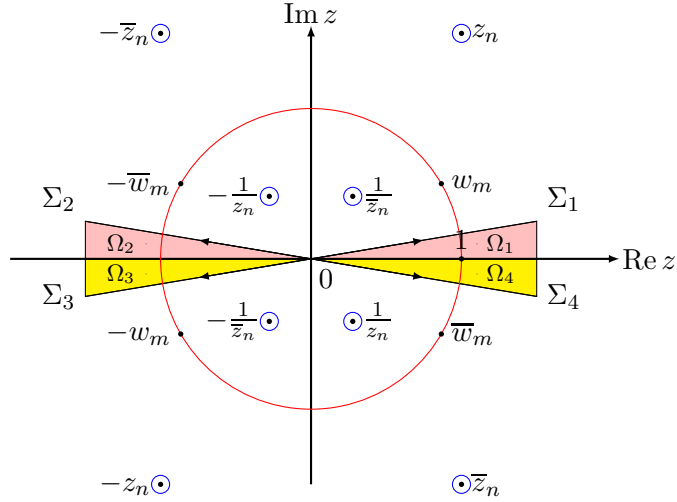
\begin{prop}\label{PP-theta}
Let $F(s)=s+\frac{1}{s}$ be real function for $s\in\D{R}$, then the imaginary part of phase function $\theta(z)$ satisfies the following inequalities:

For the case $\xi\in(2,\infty)$,
\begin{subequations}
\begin{align}
&\Im\theta(z)\geq|\sin\psi|F(|z|)\left(\frac{\xi}{4}-\frac{1}{\cos2\psi+1}\right), \quad\qquad z\in\Omega_1\cup\Omega_2, \label{theta-est1}\\
&\Im\theta(z)\leq-|\sin\psi|F(|z|)\left(\frac{\xi}{4}-\frac{1}{\cos2\psi+1}\right), \qquad z\in\Omega_3\cup\Omega_4,\label{theta-est2}
\end{align}
\end{subequations}
For the case $\xi\in(-\infty,-1/4)$,
\begin{subequations}
\begin{align}
&\Im\theta(z)\leq-|\sin\psi|F(|z|)\left(-\frac{\xi}{4}-\frac{1}{4(\cos2\psi+1)}\right), \qquad z\in\Omega_1\cup\Omega_2, \label{theta-est3}\\
&\Im\theta(z)\geq|\sin\psi|F(|z|)\left(-\frac{\xi}{4}-\frac{1}{4(\cos2\psi+1)}\right), \quad\qquad z\in\Omega_3\cup\Omega_4. \label{theta-est4}
\end{align}
\end{subequations}
\end{prop}

\begin{proof}
Taking the region $\Omega_1$ as an example, it can be proved that other regions are similar with the same method.
Let $z=|z|e^{i\phi}$ and $F(|z|)=|z|+\frac{1}{|z|}$,  then the phase function $\theta(z)$ \eqref{phe} can be written as
\begin{align}
\Im\theta(z)=-F(|z|)\sin\phi\left(-\frac{\xi}{4}+\frac{4\cos2\phi-2F(|z|)^2+12}{(F(|z|)^2+2\cos2\phi-2)^2}\right).
\end{align}
As in \cite{YF-AM}, considering the maximum and minimum values of $$H(|z|,\phi):=\frac{4\cos2\phi-2F(|z|)^2+12}{(F(|z|)^2+2\cos2\phi-2)^2}$$
yields $H(|z|,\phi)\in(-\frac{1}{4(\cos2\phi+1)},\frac{1}{\cos2\phi+1})$, which completes the proof of Proposition \ref{PP-theta}.
\end{proof}
\begin{cor}\label{cor-im}
Let $\xi\in(-\infty,-1/4)\cup(2,+\infty)$. For $z=|z|e^{i\phi}=u+iv$, there exists a constant $\tau(\xi)$ such that:

For the case $\xi\in(2,\infty)$,
\begin{subequations}
\begin{align}
&\Im\theta(z)\geq\tau(\xi)v,\qquad\quad z\in\Omega_1\cup\Omega_2,\\
&\Im\theta(z)\leq-\tau(\xi)v,\qquad z\in\Omega_3\cup\Omega_4.
\end{align}
\end{subequations}
For the case $\xi\in(-\infty,-1/4)$,
\begin{subequations}
\begin{align}
&\Im\theta(z)\leq-\tau(\xi)v,\qquad z\in\Omega_1\cup\Omega_2,\\
&\Im\theta(z)\geq\tau(\xi)v,\qquad\quad z\in\Omega_3\cup\Omega_4.
\end{align}
\end{subequations}
\end{cor}

In what follows, we will turn our attention to the continuous continuation of the jump matrix $V^{(1)}(z)$. For this purpose, we first define
\begin{align}
&r_{1}(z,\xi)=r_{2}(z,\xi)=\left\{\begin{aligned}
&\overline{r}(z),&&\xi>2,\\
&-\frac{r(z)}{1-|r(z)|^2},&&\xi<-1/4.
\end{aligned}\right.\\
&r_{3}(z,\xi)=r_{4}(z,\xi)=\left\{\begin{aligned}
&r(z),&&\xi>2,\\
&-\frac{\overline{r}(z)}{1-|r(z)|^2},&&\xi<-1/4.
\end{aligned}\right.
\end{align}
In order to eliminate the jump on the real axis $\D{R}$ and generate two steepest descent lines, it is necessary to introduce the matrix $R^{(2)}(z)$ with good $\overline{\partial}$-derivatives.\\
For $\xi\in(2,+\infty)$:
\begin{equation}
	R^{(2)}(z,\xi)=\left\{\begin{array}{lll}
		\left(\begin{array}{cc}
			1 & 0\\
			R_j(z,\xi)e^{2it\theta} & 1
		\end{array}\right), & z\in \Omega_j,j=1,2,\\
		\\
		\left(\begin{array}{cc}
			1 & R_j(z,\xi)e^{-2it\theta}\\
			0 & 1
		\end{array}\right),  &z\in \Omega_j,j=3,4,\\
		\\
		I,  &elsewhere;\\
	\end{array}\right.\label{R(2)+}
\end{equation}

For $\xi\in(-\infty,-1/4)$:
\begin{equation}
R^{(2)}(z,\xi)=\left\{\begin{array}{lll}
\left(\begin{array}{cc}
1 & R_j(z,\xi)e^{-2it\theta}\\
0 & 1
\end{array}\right), & z\in \Omega_j,j=1,2,\\
\\
\left(\begin{array}{cc}
1 & 0\\
R_j(z,\xi)e^{2it\theta} & 1
\end{array}\right),  &z\in \Omega_j,j=3,4,\\
\\
I,  &elsewhere,\\
\end{array}\right.\label{R(2)-}
\end{equation}
where the functions $R_{j}(z,\xi)$, $j=1,2,3,4$ are defined in the following proposition.
\begin{prop}\label{PP-R2}
The functions $R_j(z,\xi):\overline{\Omega}\rightarrow\D{C}$ have the boundary values:\\

For $\xi\in(2,+\infty)$:
\begin{subequations}
\begin{align}
&R_1(z,\xi)=\left\{\begin{aligned}
&r_1(z,\xi)T(z)^{-2}, & z\in \mathbb{R}^+,\\
&0,  &z\in \Sigma_1,\\
\end{aligned}\right.  \\
&R_2(z,\xi)=\left\{\begin{aligned}
&r_2(z,\xi)T(z)^{-2}, &z\in  \mathbb{R}^-,\\
			&0,  &z\in \Sigma_2,\\
\end{aligned}\right. \\
&R_3(z,\xi)=\left\{\begin{aligned}
			&r_3(z,\xi)T(z)^{2}, &z\in \mathbb{R}^-, \\
			&0, &z\in \Sigma_3,\\
\end{aligned}\right.\\
&R_4(z,\xi)=\left\{\begin{aligned}
&r_4(z,\xi)T(z)^{2}, &z\in  \mathbb{R}^+,\\
			&0,  &z\in \Sigma_4,\\		
\end{aligned}\right.
\end{align}	
\end{subequations}

for the case $\xi\in(-\infty,-1/4)$,
\begin{subequations}
\begin{align}
&R_1(z,\xi)=\left\{\begin{aligned}
&r_1(z,\xi)T_+(z)^{2}, & z\in \mathbb{R}^+,\\
&0,  &z\in \Sigma_1,\\
\end{aligned}\right. \\
&R_2(z,\xi)=\left\{\begin{aligned}
&r_2(z,\xi)T_+(z)^{2}, &z\in  \mathbb{R}^-,\\
			&0,  &z\in \Sigma_2,\\
\end{aligned}\right. \\
&R_3(z,\xi)=\left\{\begin{aligned}
			&r_3(z,\xi)T_-(z)^{-2}, &z\in \mathbb{R}^-, \\
			&0, &z\in \Sigma_3,\\
\end{aligned}\right.\\
&R_4(z,\xi)=\left\{\begin{aligned}
&r_4(z,\xi)T_-(z)^{-2}, &z\in  \mathbb{R}^+,\\
			&0,  &z\in \Sigma_4,\\		
\end{aligned}\right.
\end{align}	
\end{subequations}
where the matrix-valued functions $R_{k}(z,\xi)$ such that for $k=1,2,3,4$
\begin{align}\label{R2-D}
\left|\overline{\partial}R_{k}(z,\xi)\right|\lesssim \left|r'_k(|z|)\right|+|z|^{-1/2}, \qquad  z\in\Omega_k,
\end{align}
furthermore
\begin{align}
&\left|\overline{\partial}R_{k}(z,\xi)\right|\lesssim \left|r'_k(|z|)\right|+|z|^{-1}, &  z\in\Omega_k,\label{R2-D1}\\
&\left|\overline{\partial}R_{k}(z,\xi)\right|\equiv0,&z\in\text{elsewhere}.
\end{align}
\end{prop}
\begin{proof}
Taking $R_{1}(z,\xi)$ in the region $\xi\in(2,+\infty)$  as example, it can be proved that other situations are similar. We now construct a new function
\begin{align}\label{con-f}
f_1(z)=r_1(|z|)T^{-2}(z)\cos(z_0\arg z), \qquad z_0=\frac{\pi}{z\phi}.
\end{align}
Let $z=\rho e^{i\phi}$. Observing that the function \eqref{con-f} has the same boundary values on the contour $\Sigma_1$ and $\D{R}$, which means that the function we construct is appropriate. Applying the fact that $\overline{\partial}=\frac{e^{i\phi}}{2}\left(\partial_\rho +i\rho \partial_\phi\right)$ to the function $f_1(z)$ yields that
\begin{align}
\overline{\partial}f_1(z)=\frac{e^{i\phi}}{2}T^{-2}(z)\left(r'_1(\rho)\cos(z_0\phi)-\frac{iz_0}{\rho}r_1(\rho)
\sin(z_0\phi)\right).
\end{align}
It follows from the boundedness of functions $T(z)$ and $T(z)^{-1}$ that
\begin{align}\label{f-ds}
\left|\overline{\partial}f_1(z)\right|\leq|r'_1(|z|)|+\left|\frac{r_1(|z|)}{|z|}\right|.
\end{align}
On the other hand,
\begin{align}\label{ine}
|r_1(|z|)|&=|r_1(|z|)-r_1(0)|=\left|\int_{0}^{|z|}r'_1(s)\ ds\right|\nonumber\\
&\leq||r'_1(s)||_{L^{2}(\D{R})}|z|^{1/2}\lesssim |z|^{1/2}.
\end{align}
Combining the inequality \eqref{ine} with \eqref{f-ds} yields \eqref{R2-D}. Additionally, Observing the fact $r_1(|z|)\in L^{\infty}$, one has  \eqref{R2-D1}, which completes the proof of the proposition.
\end{proof}

Based on the above analysis, a new unknown function $M^{(2)}(z):=M^{(2)}(y,t,z)$ can be further defined by $R^{(2)}(z)$ as follows
\begin{align}\label{M2}
M^{(2)}(z)=M^{(1)}(z)R^{(2)}(z),
\end{align}
which solves a mixed RH problem.
\begin{RHP}\label{RHP-M2}
Find a matrix-valued function $M^{(2)}(z)$ satisfy the following conditions:
\begin{enumerate}[(I)]
\item The function $M^{(2)}(z)$ is continuous in $\D{C}\setminus\Sigma^{(2)}$ and takes continuous boundary values $M^{(2)}_{\pm}(z)$ on $\Sigma^{(2)}$ from the left (respectively right).
\item $M^{(2)}(z)=\sigma_1\overline{M^{(2)}(\overline{z})}\sigma_1=
    \sigma_2M^{(2)}(-z)\sigma_2=\sigma_1M^{(2)}(z^{-1})\sigma_1$.
\item Jump condition: The non-tangential limits $M^{(2)}_{\pm}(z)$ for $z\in\D{R}\setminus\{\pm1\}$
\begin{align}\label{JM2}
M^{(2)}_{+}(z)=M^{(2)}_{-}(z)V^{(2)}(z),~~~~~z\in\Sigma^{(2)},
\end{align}
where the jump matrix $V^{(2)}(z)$ is defined by
\begin{align}\label{Jump-M2}
V^{(2)}(z)=\left\{\begin{aligned}
&\left(\begin{array}{cc}
			1 & 0\\
			-C_n(z-\eta_n)^{-1}T^{-2}(z)e^{2it\theta_n} & 1
		\end{array}\right),   \qquad z\in\mathbb{P}_n,n\in\nabla\setminus\Lambda,\\[12pt]
&\left(\begin{array}{cc}
			1 & -C_n^{-1}(z-\eta_n)T^{2}(z)e^{-2it\theta_n}\\
			0 & 1
		\end{array}\right),   ~~\qquad z\in\mathbb{P}_n,n\in\Delta\setminus\Lambda,\\
&\left(\begin{array}{cc}
		1 & \overline{C}_n(z-\overline{\eta}_n)^{-1}T^{2}(z)e^{-2it\overline{\theta}_n}\\
		0 & 1
		\end{array}\right), ~~~~\qquad 	z\in\overline{\mathbb{P}}_n,n\in\nabla\setminus\Lambda,\\
&\left(\begin{array}{cc}
		1 & 0	\\
		\overline{C}_n^{-1}(z-\overline{\eta}_n)T^{-2}(z)e^{2it\bar{\theta}_n} & 1
		\end{array}\right),   \qquad\quad	z\in\overline{\mathbb{P}}_n,n\in\Delta\setminus\Lambda.
\end{aligned}\right.
\end{align}
\item Asymptotic behavior:
\begin{align}\label{Asy-M2}
&M^{(2)}(z)=I+O(z^{-1}), z\to\infty~ (\text{and}~ z\to0 ~\text{by symmetry}),\\
&M^{(2)}(z)=\left(
        \begin{array}{cc}
          a_1(y,t) & 0 \\
          0 & a_1^{-1}(y,t) \\
        \end{array}
      \right)+\left(
        \begin{array}{cc}
          0 & a_2(y,t) \\
          a_3(y,t) & 0 \\
        \end{array}
      \right)(z-i)+O((z-i)^2),
\end{align}
where $a_j(y,t)$, $j=1,2,3$ are real-valued functions from the expansion of $M^{(2)}(z)$ at $z=i$.
\item Residue conditions: The matrix $M^{(2)}(z)$ has simple poles at $\eta_n$ and $\overline{\eta}_n$ with
\begin{subequations}\label{res-M2}
\begin{align}
&\mathop{\res}\limits_{\eta_n} M^{(2)}(z)=\lim_{z\to\eta_n} M^{(2)} (z)\left(
                                                  \begin{array}{cc}
                                                    0 & 0 \\
                                                    c_nT^2(z)e^{2it\theta(\eta_n)} & 0 \\
                                                  \end{array}
                                                \right), \label{resM2} \\
&\mathop{\res}\limits_{\overline{\eta}_n} M^{(2)}(z)=\lim_{z\to\overline{\eta}_n} M^{(2)} (z)\left(
                                                  \begin{array}{cc}
                                                    0 & \overline{c}_nT^{-2}(z)e^{-2it\overline{\theta}(\eta_n)} \\
                                                    0 & 0 \\
                                                  \end{array}\right).
\end{align}
\end{subequations}
\item The $\overline{\partial}$-derivative: For $z\in\D{C}$
\begin{align}\label{R2-Dbar}
\overline{\partial}M^{(2)}(z)=M^{(2)}(z)\overline{\partial}R^{(2)}(z),
\end{align}
where for $\xi\in(2,+\infty)$
\begin{align*}
\overline{\partial}R^{(2)}(z)=\left\{\begin{aligned}
&\left(
  \begin{array}{cc}
    0 & 0 \\
    \overline{\partial}R_j(z,\xi)e^{2it\theta} & 0 \\
  \end{array}
\right), &z\in\Omega_j,~~j=1,2,\\
&\left(
  \begin{array}{cc}
    0 & \overline{\partial}R_j(z,\xi)e^{-2it\theta} \\
    0 & 0 \\
  \end{array}
\right), &z\in\Omega_j,~~j=3,4,\\
&\textbf{0},&z\in\text{elsewhere},
\end{aligned}\right.
\end{align*}
for $\xi\in(-\infty,-1/4)$
\begin{align*}
\overline{\partial}R^{(2)}(z)=\left\{\begin{aligned}
&\left(
  \begin{array}{cc}
    0 & \overline{\partial}R_j(z,\xi)e^{-2it\theta} \\
    0 & 0 \\
  \end{array}
\right), &z\in\Omega_j,~~j=1,2,\\
&\left(
  \begin{array}{cc}
    0 & \overline{\partial}R_j(z,\xi)e^{2it\theta} \\
    0 & 0 \\
  \end{array}
\right), &z\in\Omega_j,~~j=3,4,\\
&\textbf{0},&z\in\text{elsewhere}.
\end{aligned}\right.
\end{align*}
\end{enumerate}
\end{RHP}

\subsection{Decomposition of mixed RH problem}\label{subsec4.2}
In this section, RH problem \ref{RHP-M2} will be solved according to the classification of $\overline{\partial}$-derivatives $\overline{\partial}R^{(2)}(z)$. One is that $\overline{\partial}R^{(2)}(z)\equiv0$ corresponds to the pure soliton solution $M^{(rhp)}(z)$, and the other is that $\overline{\partial}R^{(2)}(z)\neq0$ corresponds to the pure $\overline{\partial}$-problem $M^{(3)}(z)$, which can be solved by double integral. First, consider the pure RH problem, that is, the corresponding $\overline{\partial}$-derivative $\overline{\partial}R^{(2)}(z)$ is constant to zero.
\begin{RHP}\label{RHP-Mrhp}
A matrix-valued function $M^{(rhp)}(z):=M^{(rhp)}(y,t,z)$ with the same symmetry, analytic property and asymptotic property as $M^{(2)}(z)$ satisfies the following conditions
\begin{enumerate}[(1)]
\item Jump condition: The non-tangential limits $M^{(rhp)}_{\pm}(z)$ for $z\in\D{R}\setminus\{\pm1\}$
\begin{align}\label{JMrhp}
M^{(rhp)}_{+}(z)=M^{(rhp)}_{-}(z)V^{(rhp)}(z),~~~~~z\in\Sigma^{(2)},
\end{align}
where the jump matrix $V^{(rhp)}(z)$ is defined by \eqref{Jump-M2}.
\item The $\overline{\partial}$-derivative: For $z\in\D{C}$, $\overline{\partial}R^{(2)}(z)\equiv0$.
\item Residue conditions: The matrix $M^{(rhp)}(z)$ has simple poles at $\eta_n$ and $\overline{\eta}_n$ with
\begin{subequations}\label{res-Mrhp}
\begin{align}
&\mathop{\res}\limits_{\eta_n} M^{(rhp)}(z)=\lim_{z\to\eta_n} M^{(rhp)} (z)\left(
                                                  \begin{array}{cc}
                                                    0 & 0 \\
                                                    c_nT^2(z)e^{2it\theta(\eta_n)} & 0 \\
                                                  \end{array}
                                                \right), \label{resMr} \\
&\mathop{\res}\limits_{\overline{\eta}_n} M^{(rhp)}(z)=\lim_{z\to\overline{\eta}_n} M^{(rhp)} (z)\left(
                                                  \begin{array}{cc}
                                                    0 & \overline{c}_nT^{-2}(z)e^{-2it\overline{\theta}(\eta_n)} \\
                                                    0 & 0 \\
                                                  \end{array}\right).
\end{align}
\end{subequations}
\end{enumerate}
\end{RHP}

The pure $\overline{\partial}$-problem can be obtained by removing the soliton component $M^{(rhp)}(z)$
\begin{align}\label{M3}
M^{(3)}(z)=M^{(2)}(z)M^{(rhp)}(z)^{-1},
\end{align}
which solves the following pure $\overline{\partial}$-problem.
\begin{DRHP}\label{RHP-Dbar}
Find a matrix-valued function $M^{(3)}(z)$ such that
\begin{enumerate}[$($$I$$)$]
\item $M^{(3)}(z)$ is continuous in $\D{C}$, and analytic in $\D{C}\setminus\overline{\Omega}$.
\item $M^{(3)}(z)\to I$ as $z\to \infty$.
\item $\overline{\partial}$-derivative: For $z\in\D{C}$, one has
\begin{align}\label{W3}
\overline{\partial}M^{(3)}(z)=M^{(3)}(z)W^{(3)}(z),
\end{align}
where $W^{(3)}(z)=M^{(rhp)}(z)\overline{\partial}R^{(2)}(z)M^{(rhp)}(z)^{-1}$.
\end{enumerate}
\end{DRHP}
\begin{proof}
The analytic properties of $M^{(3)}(z)$ can be obtained by using the properties of $M^{(2)}(z)$ and $M^{(rhp)}(z)$  in RH problems \ref{RHP-M2} and  \ref{RHP-Mrhp}, respectively. Observing that $M^{(2)}(z)$ and $M^{(rhp)}(z)$ have the same jump conditions $V^{(2)}(z)$ and $V^{(rhp)}(z)$ ($V^{(2)}(z)\equiv V^{(rhp)}(z)$), then
\begin{align*}
M_-^{(3)}(z)^{-1}M_+^{(3)}(z)&=M_-^{(2)}(z)^{-1}M_-^{(rhp)}(z)M_+^{(rhp)}(z)^{-1}M_+^{(2)}(z)\\
&=M_-^{(2)}(z)^{-1}V^{(2)}(z)^{-1}M_+^{(2)}(z)=I,
\end{align*}
which means that $M^{(3)}(z)$ is no jump condition and is continuous. Meanwhile, $M^{(3)}(z)$ has no poles for $\xi\in\{\eta_n,\overline{\eta}_n\} $ with $n\in\Lambda$. Let $A_\xi$ be the nilpotent matrix appearing in the residue conditions of RH problems \ref{RHP-M2} and \ref{RHP-Mrhp}, and then we have a local expansion in the neighborhood of $\xi$
\begin{align}
&M^{(2)}(z)=\alpha(\xi) \left[ \dfrac{A_\xi}{z-\xi}+I\right] +\mathcal{O}(z-\xi),\nonumber\\
&	M^{(rhp)}(z)=\beta(\xi) \left[ \dfrac{A_\xi}{z-\xi}+I\right] +\mathcal{O}(z-\xi),\nonumber
\end{align}
where $\alpha(\xi)$ and $\beta(\xi)$ are constant coefficients of the corresponding expansion respectively.
Taking the product yields that $M^{(3)}(z)=\mathcal{O}(1)$, which indicates that the poles of $M^{(3)}(z)$  are removable and $M^{(3)}(z)$ is boundary locally. The proof of $\overline{\partial}$-derivative for $M^{(3)}(z)$ follows the following process
\begin{align*}
\overline{\partial}M^{(3)}(z)&=\overline{\partial}(M^{(2)}(z)M^{(rhp)}(z)^{-1})
=M^{(2)}(z)\overline{\partial}R^{(2)}(z)M^{(rhp)}(z)^{-1}\\&=M^{(2)}(z)M^{(rhp)}
(z)^{-1}M^{(rhp)}(z)\overline{\partial}R^{(2)}(z)M^{(rhp)}(z)^{-1}\\
&=M^{(3)}(z)M^{(rhp)}(z)\overline{\partial}R^{(2)}(z)M^{(rhp)}(z)^{-1}\triangleq M^{(3)}(z)W^{(3)}(z).
\end{align*}
\end{proof}

\subsection{Asymptotic of $N$-soliton solution}\label{subsec4.3}
In this section, the solution of the original RH problem \ref{RHP-M} without reflection will be established to approximate the solution of RH problem \ref{RHP-Mrhp}. For $\overline{\partial}R^{(2)}(z)=0$, RH problem \ref{RHP-M2} is reduced to RH problem \ref{RHP-Mrhp}. The existence and uniqueness of solutions of RH problem \ref{RHP-Mrhp} can be guaranteed by connecting the solution of RH problem \ref{RHP-Mrhp} with the original RH problem  \ref{RHP-M}.

\begin{prop}\label{pp-Mrhp}
Let $M^{(rhp)}(z)$ be the solution to RH problem \ref{RHP-Mrhp} with $\overline{\partial}R^{(2)}(z)\equiv0$. For the scattering data $\{r(z),(c_n,\eta_n)_{n\in N}\}$ in  RH problem \ref{RHP-Mrhp}, the solution to RH problem  \ref{RHP-Mrhp} exists and is equivalent to the solution to RH problem \ref{RHP-M} with reflection-less scattering data $\{0,(\widetilde{c}_n,\eta_n)_{n\in N}\}$, where the modified connection coefficients $\widetilde{c}_n:=\widetilde{c}_n(x,t)$ are defined by
\begin{align}
\widetilde{c}_n=c_n\exp\left(-\frac{1}{i\pi}\int_{\Sigma_{b}(\xi)}\frac{\ln(1-|r(s)|^2)}{s-z}\ ds\right),
\end{align}
where $r(z)$ is reflection coefficient defined by Proposition \ref{pp-S}.
\end{prop}
\begin{proof}
With $\overline{\partial}R^{(2)}(z)\equiv0$, then the RH problem \ref{RHP-Mrhp} for $M^{(rhp)}(z)$ degenerates into a piecewise meromorphic function with discontinuous jumps on the contour $\Sigma^{(2)}$. Observing the solutions   $M^{(rhp)}(z)$ and $M(z)$, we need restore the poles at $\eta_n$ and $\overline{\eta}_n$ for $M^{(rhp)}(z)$. Reversing the  triangularity effected by \eqref{T} and \eqref{M1} yields that
\begin{align}\label{MJ}
M^{(\natural)}(z)=M^{(rhp)}(z)G(z)\left(\prod_{n\in\Delta}\frac{z-\eta_n}{
\overline{\eta}_n^{-1}z-1}\right)^{-\sigma_3},
\end{align}
where
\begin{equation*}
G(z)=\left\{ \begin{array}{ll}
\left(\begin{array}{cc}
			1 & 0\\
			C_n(z-\eta_n)^{-1}e^{2it\theta_n} & 1
\end{array}\right),   &z\in\mathbb{P}_n,n\in\nabla\setminus\Lambda,\\[12pt]
\left(\begin{array}{cc}
			1 & C_n^{-1}(z-\eta_n)e^{-2it\theta_n}\\
			0 & 1
\end{array}\right),   &z\in\mathbb{P}_n,n\in\Delta\setminus\Lambda,\\
\left(\begin{array}{cc}
		1 & -\overline{C}_n(z-\overline{\eta}_n)^{-1}e^{-2it\overline{\theta}_n}\\
		0 & 1
\end{array}\right),   &z\in\overline{\mathbb{P}}_n,n\in\nabla\setminus\Lambda,\\
\left(\begin{array}{cc}
		1 & 0	\\
		-\overline{C}_n^{-1}(z-\overline{\eta}_n)e^{2it\bar{\theta}_n} & 1
\end{array}\right),   &z\in\overline{\mathbb{P}}_n,n\in\Delta\setminus\Lambda.
\end{array}\right.
\end{equation*}

It follows from the transformation \eqref{MJ} that for the origin and infinity, $M^{(\natural)}(z)$ preserves the normalization conditions. Comparing with \eqref{M1}, the transformation \eqref{MJ} recovers the jump conditions on the circles $\D{P}_n$ and $\overline{\D{P}}_n$ with $n\notin\Lambda$.  For $n\in\Lambda$, it can be verified that $M^{(\natural)}(z)$ satisfies residue formula \eqref{resMr} with modified connection coefficient. Taking $\eta_n$ as example,
\begin{align}
\mathop{\res}\limits_{\eta_n}M^{(\natural)}(z)&=\mathop{\res}\limits_{\eta_n} M^{(rhp)}(z)G(z)\left(\prod_{n\in\Delta}\frac{z-\eta_n}{
\overline{\eta}_n^{-1}z-1}\right)^{-\sigma_3},\nonumber\\
&=\lim_{z\to\eta_n}M^{(rhp)}(z)\left(
                                                  \begin{array}{cc}
                                                    0 & 0 \\
                                                    c_nT^2(z)e^{2it\theta(\eta_n)} & 0 \\
                                                  \end{array}
                                                \right)G(z)\left(\prod_{n\in\Delta}\frac{z-\eta_n}{
\overline{\eta}_n^{-1}z-1}\right)^{-\sigma_3},\nonumber\\
&=\lim_{z\to\eta_n}M^{(\natural)}(z)\left(
                                      \begin{array}{cc}
                                        0 & 0 \\
                                        c_ne^{2it\theta}T^{2}(z)
                                        \left(\mathop{\prod}\limits_{n\in\Delta}\frac{z-\eta_n}{
\overline{\eta}_n^{-1}z-1}\right)^2 & 0 \\
                                      \end{array}
                                    \right),\nonumber\\
&=\lim_{z\to\eta_n}M^{(\natural)}(z)\left(
                                      \begin{array}{cc}
                                        0 & 0 \\
                                        \widetilde{c}_ne^{2it\theta} & 0 \\
                                      \end{array}
                                    \right).
\end{align}
The analyticity and symmetry of $M^{(\natural)}(z)$ can be derived by the properties of $M^{(rhp)}(z)$, $T(z)$, and $G(z)$. On the other hand, although $M^{(\natural)}(z)$ has no normalization condition at singularity $i$, singularity $i$ is not the pole of $M^{(\natural)}(z)$, which means that the singularity does not affect $M^{(rhp)}(z)$. In conclusion, we know that $M^{(\natural)}(z)$ is the solution of RH problem \ref{RHP-M} under the condition of no reflection. Its uniqueness comes from Liouville's theorem, which leads to the uniqueness and existences of $M^{(\natural)}(z)$  from transformation \eqref{MJ}.
\end{proof}

Proposition \ref{pp-Mrhp} shows the existence and uniqueness of the solution of $M^{(rhp)}(z)$, but since the discrete spectra are distributed in the whole complex plane, the contribution of the discrete spectra to the solution needs to be further considered. As in \cite{Cu-CMP}, the jump condition $V^{(2)}(z)$ such that
\begin{align}\label{est-V2}
\left|\left|V^{(2)}(z)-I\right|\right|_{L^{p}(\Sigma^{(2)})}\lesssim K_pe^{-2\varepsilon_0t}, 1\leq p\leq\infty
\end{align}
where $\varepsilon_0=\min_{n\in N\setminus\Lambda}\{|\Im\theta|\}$ and  $K_p\geq0$ are constants independent of $(y,t)$, which implies that  the jump matrix $V^{(2)}(z)$ on the disks $\D{P}_n$ uniformly converges to the identity matrix, and the contribution  of these discrete spectra to the solution can be ignored. The proof of inequality \eqref{est-V2} can be briefly described as follows. Taking $z\in\D{P}_n$ with $n\in\nabla$ and $p=\infty$ as example, it follows from \eqref{JM2} that
\begin{align}
\left|\left|V^{(2)}(z)-I\right|\right|_{L^{\infty}(\Sigma^{(2)})}&=
\left|c_n(z-\eta_n)^{-1}T(z)^2e^{2it\theta}\right|\nonumber \\
&\lesssim e^{\Re(2it\theta)}=e^{-2t\Im\theta}\lesssim e^{-2\varepsilon_0t}.
\end{align}
In these inequalities, use has been made of the fact that for $n\in\nabla$, then $\Im\theta>0$. At the same time, the estimation of jump matrix $V^{(2)}(z)$ \eqref{est-V2} further encourages us to decompose matrix $M^{(rhp)}(z)$ into the following form
\begin{align}\label{Ez}
M^{(rhp)}(z)=M^{(err)}(z)M^{(sol)}(z),
\end{align}
where $E(z)$ is error function solved by small-norm RH problem as \cite{KMM-2003,Deift-1994,DZ-CPAM} and RH problem \ref{RHP-Mrhp} for $M^{(rhp)}(z)$ reduces to a new RH problem for $M^{(sol)}(z)$ as $V^{(2)}(z)=0$.

\begin{RHP}\label{RHP-Msol}
Find a matrix-valued function $M^{(sol)}(z)$ such that
\begin{enumerate}[(I)]
\item The function $M^{(sol)}(z)$ is analytic in $\D{C}\setminus\{\eta_n,\overline{\eta}_n\}_{n\in\Lambda}$.
\item $M^{(sol)}(z)=I+O(z^{-1})$ as $z\to\infty$.
\item For $n\in\Lambda$, the function $M^{(sol)}(z)$ is of simple poles at $\eta_n$ and $\overline{\eta}_n$ with
\begin{subequations}\label{res-Msol}
\begin{align}
&\mathop{\res}\limits_{\eta_n} M^{(sol)}(z)=\lim_{z\to\eta_n} M^{(sol)} (z)\left(
                                                  \begin{array}{cc}
                                                    0 & 0 \\
                                                    c_nT^2(z)e^{2it\theta(\eta_n)} & 0 \\
                                                  \end{array}
                                                \right), \label{resMs} \\
&\mathop{\res}\limits_{\overline{\eta}_n} M^{(sol)}(z)=\lim_{z\to\overline{\eta}_n} M^{(sol)} (z)\left(
                                                  \begin{array}{cc}
                                                    0 & \overline{c}_nT^{-2}(z)e^{-2it\overline{\theta}(\eta_n)} \\
                                                    0 & 0 \\
                                                  \end{array}\right).
\end{align}
\end{subequations}
\end{enumerate}
\end{RHP}

The unique solution of RH problem \ref{RHP-Msol} is embodied in Proposition \ref{prop-sol-Msol} and as $z\to i$, one has
\begin{align}\label{exp-Msol-i}
M^{(sol)}(z)=M^{(sol)}(i)+M^{(sol)}_1(z-i)+O((z-i)^{2}).
\end{align}

\begin{prop}\label{prop-sol-Msol}
The solution $M^{(sol)}(z)$ to RH problem \ref{RHP-Msol} is  equivalent to the original RH problem \ref{RHP-M} with reflection-less and modified scattering data $\left\{0,\{\eta_n,C_nT^2(\eta_n)\}_{n\in\Lambda}\right\}$.
\begin{enumerate}[(1)]
\item As $\Lambda=\emptyset$, then
\begin{align}
M^{(sol)}(z)=I+\frac{i\widehat{\alpha}_+}{2(z-1)}\left(
                                                   \begin{array}{cc}
                                                     -1 & 1 \\
                                                     -1 & 1 \\
                                                   \end{array}
                                                 \right)-
\frac{i\widehat{\alpha}_+}{2(z+1)}\left(
                                                   \begin{array}{cc}
                                                     1 & 1 \\
                                                     -1 & -1 \\
                                                   \end{array}
                                                 \right).
\end{align}
\item As $\Lambda\neq\emptyset$, and $\Lambda=\{\eta_{j_k}\}_{k=1}^{N}$, then
\begin{align}\label{sol-Msol}
M^{(sol)}(z)=I&+\frac{i\widehat{\alpha}_+}{2(z-1)}\left(
                                                   \begin{array}{cc}
                                                     -1 & 1 \\
                                                     -1 & 1 \\
                                                   \end{array}
                                                 \right)-
\frac{i\widehat{\alpha}_+}{2(z+1)}\left(
                                                   \begin{array}{cc}
                                                     1 & 1 \\
                                                     -1 & -1 \\
                                                   \end{array}
                                                 \right) \nonumber \\
&+\sum_{k=1}^{N}\left(
                  \begin{array}{cc}
                    \frac{\alpha_k}{z-\eta_{j_k}} & \frac{\overline{\beta}_k}{z-\overline{\eta}_{j_k}} \\
                    \frac{\beta_k}{z-\eta_{j_k}} & \frac{\overline{\alpha}_k}{z-\overline{\eta}_{j_k}}  \\
                  \end{array}
                \right),
\end{align}
where the coefficients $\alpha_k$ and $\beta_k$ satisfy that for the discrete spectrum $\{z_n,-\overline{z}_n\}_{n=1}^{N_1}$, $\{z_{n}^{-1},-\overline{z}_{n}^{-1}\}_{n=N_1+1}^{2N_1}$ and  $\{w_n,-\overline{w}_n\}_{n=2N_1+1}^{2N_1+N_2}$, then $\alpha_k=-\overline{\alpha}_{k+n}$ and $\beta_k=\overline{\beta}_{k+n}$, which are determined by the system
\begin{subequations}\label{system}
\begin{align}
&\alpha_kc_{j_k}^{-1}T(\eta_{j_k})^{-2}e^{-2it\theta}=\frac{i\widehat{\alpha}_+}{2(z-1)}+
\frac{i\widehat{\alpha}_+}{2(z-1)}+\sum_{h=1}^{N}\frac{\bar{\eta}_{h}}{\eta_{j_k}-\eta_{j_h}},\\
&\beta_kc_{j_k}^{-1}T(\eta_{j_k})^{-2}e^{-2it\theta}=1+\frac{i\widehat{\alpha}_+}{2(z-1)}-
\frac{i\widehat{\alpha}_+}{2(z-1)}+\sum_{h=1}^{N}\frac{\bar{\alpha}_{h}}{\eta_{j_k}-\eta_{j_h}}.
\end{align}
\end{subequations}
\end{enumerate}
\end{prop}
\begin{proof}
For the case $\Lambda=\emptyset$, the conclusion is obvious, so next we turn our attention mainly to proving case $\Lambda\neq\emptyset$. For convenience, only the case of out of circle discrete spectrum is considered, thus $M^{(sol)}(z)$ has the following expansion form
\begin{align}\label{exp-Msol}
M^{(sol)}(z)=I&+\frac{i\widehat{\alpha}_+}{2(z-1)}\left(
                                                   \begin{array}{cc}
                                                     -c & 1 \\
                                                     -c & 1 \\
                                                   \end{array}
                                                 \right)-
\frac{i\widehat{\alpha}_+}{2(z+1)}\left(
                                                   \begin{array}{cc}
                                                     1 & 1 \\
                                                     -c & -c \\
                                                   \end{array}
                                                 \right) \nonumber \\
&+\sum_{k=1}^{N}\left(
                  \begin{array}{cc}
                    \frac{c_k^1}{z-z_{k}}+\frac{c_k^2}{z+\bar{z}_{k}}+
                    \frac{c_k^3}{z+z_k^{-1}}+\frac{c_k^3}{z-\bar{z}_k^{-1}} & \frac{\widetilde{c}_k^1}{z-\bar{z}_{k}}+\frac{\tilde{c}_k^2}{z+z_{k}}+
                    \frac{\tilde{c}_k^3}{z+\bar{z}_k^{-1}}+\frac{\tilde{c}_k^3}{z-z_k^{-1}} \\
                    \frac{c_k^5}{z-z_{k}}+\frac{c_k^6}{z+\bar{z}_{k}}+
                    \frac{c_k^7}{z+z_k^{-1}}+\frac{c_k^8}{z-\bar{z}_k^{-1}} & \frac{\widetilde{c}_k^5}{z-\bar{z}_{k}}+\frac{\tilde{c}_k^6}{z+z_{k}}+
                    \frac{\tilde{c}_k^7}{z+\bar{z}_k^{-1}}+\frac{\tilde{c}_k^8}{z-z_k^{-1}}  \\
                  \end{array}
                \right).
\end{align}
It follows from the symmetry $M^{(sol)}_1(-z)=\sigma_3\sigma_1M^{(sol)}_2(z)$ that $c=1$  and
\begin{align}\label{sym-coe1}
\begin{split}
&\tilde{c}_k^6=-c_k^1,\ \ \tilde{c}_k^5=-c_k^2,\ \
\tilde{c}_k^8=-c_k^3,\ \ \tilde{c}_k^7=-c_k^4,\\
&\tilde{c}_k^1=c_k^6,\ \ \tilde{c}_k^2=c_k^5,\ \
\tilde{c}_k^3=c_k^8,\ \ \tilde{c}_k^4=c_k^7.
\end{split}
\end{align}
On the other hand, applying the symmetry $M^{(sol)}_1(-\bar{z})=\sigma_3M^{(sol)}_1(z)$ obtains
\begin{align}\label{sym-coe2}
c_k^2=-\bar{c}_k^1,\ \ c_k^4=-\bar{c}_k^3,\ \ c_k^6=\bar{c}_k^5,\ \
c_k^8=\bar{c}_k^7.
\end{align}
Thus, combining the relation \eqref{sym-coe1} with  \eqref{sym-coe2} yields \eqref{sol-Msol}. Also the system \eqref{system} is obtained by substituting the expansion \eqref{exp-Msol} into the residue condition \eqref{resMs} at the discrete spectrums $\eta_{j_k}$.
\end{proof}

\begin{cor}\label{cor-sol}
Under the condition of no scattering, that is, $r(z)=0$, the corresponding scattering matrix $S(z)$ is the identity matrix. Let $q_{sol}(x,t,\Lambda)$ denote the $N(\Lambda)$-soliton solutions with the modified scattering data $\sigma_d^{\Lambda}=\left\{0,\{\eta_n,c_nT^{2}(\eta_n)\}_{n\in\Lambda}\right\}$. It follows from the reconstruction formula \eqref{q-sol} for mCH equation that the solution $q_{sol}(x,t,\Lambda)$ is expressed by
\begin{align}
\begin{split}
q_{sol}(x,t,\Lambda)=&-\partial_zM_{12}^{(sol)}(z)|_{z=i}M_{11}^{(sol)}(i)
-\partial_zM_{21}^{(sol)}(z)|_{z=i}M_{11}^{(sol)}(i)^{-1},\\
&x(y,t)=y+\ln M_{11}^{(sol)}(i).
\end{split}
\end{align}
\begin{enumerate}[(I)]
\item For the case $\Lambda=\emptyset$, then
\begin{align}
q_{sol}(x,t,\Lambda)=&-\partial_zM_{12}^{(sol)}(z)|_{z=i}M_{11}^{(sol)}(i)
-\partial_zM_{21}^{(sol)}(z)|_{z=i}M_{11}^{(sol)}(i)^{-1}\nonumber \\
&=i\left(1-\widehat{\alpha}_+/2\right),\\
&x(y,t,\sigma_d^{\Lambda})=y+\ln\left[1-\frac{i\widehat{\alpha}_+}{2(i-1)}
-\frac{i\widehat{\alpha}_+}{2(i+1)}\right]\triangleq y+c_+.
\end{align}
\item For the case $\Lambda=\emptyset$, then
\begin{align}
q_{sol}(x,t,\Lambda)=&-\partial_zM_{12}^{(sol)}(z)|_{z=i}M_{11}^{(sol)}(i)
-\partial_zM_{21}^{(sol)}(z)|_{z=i}M_{11}^{(sol)}(i)^{-1}\nonumber \\
=&-\left(\frac{\widehat{\alpha}_+}{2}
-\sum_{k=1}^{N}\frac{\overline{\beta}_k}{(i-\overline{\eta}_{j_k})^2}\right)
\left(1-\frac{\widehat{\alpha}_+}{4}+\sum_{k=1}^{N}\frac{\alpha_k}
{(i-\eta_{j_k})}\right)\nonumber \\
&+\left(\frac{\widehat{\alpha}_+}{2}
-\sum_{k=1}^{N}\frac{\beta_k}{(i-\eta_{j_k})^2}\right)\left(1-\frac{\widehat{\alpha}_+}{4}
+\sum_{k=1}^{N}\frac{\alpha_k}
{(i-\eta_{j_k})}\right)^{-1}
\end{align}
and
\begin{align}
x(y,t,\sigma_d^{\Lambda})=y+\ln\left[1-\frac{\widehat{\alpha}_+}{4}
+\sum_{k=1}^{N}\frac{\alpha_k}
{(i-\eta_{j_k})}\right]\triangleq y+c_+'.
\end{align}
\end{enumerate}

\end{cor}
\subsection{The small-norm RH problem for error function}\label{subsec4.4}
In \eqref{Ez}, the jump condition $V^{(2)}(z)$ on path $\Sigma^{(2)}$ is ignored, and then error function $M^{(err)}(z)$ is generated. In this section, the attention is mainly focused on the problem of error function by a small-norm RH problem \cite{KMM-2003,Deift-1994,DZ-CPAM}.
\begin{RHP}\label{RHP-Merr}
Find a matrix valued function $M^{(err)}(z)$ satisfies
\begin{enumerate}[(1)]
\item The error function $M^{(err)}(z)$  is analytic in $\mathbb{C}\setminus\Sigma^{(2)}$.
\item Asymptotic behavior:
\begin{align*}
M^{(err)}(z)=I+\mathcal {O}(z^{-1}), ~~~z\rightarrow\infty.
\end{align*}
\item The jump condition:
\begin{align}
M^{(err)}_{+}(z)=M^{(err)}_{-}(z)V^{(err)}(z),
\end{align}
where $V^{(err)}(z)=M^{(sol)}(z)V^{(2)}(z)M^{(sol)}(z)^{-1}$, and for $z\in\Sigma^{(2)}$ the error function $M^{(err)}(z)$ is of continuous boundary values $M^{(err)}_{\pm}(z)$.
\end{enumerate}
\end{RHP}

The error function $M^{(err)}(z)$ has no poles based on the fact that $M^{(rhp)}(z)$ and $M^{(sol)}(z)$ have the same jump conditions and  poles. Furthermore, the jump matrix $V^{(err)}(z)$ follows the same estimate as $V^{(2)}(z)$
 from \eqref{est-V2} for $1\leq p\leq+\infty$
\begin{align}\label{est-Verr}
\left|\left|V^{(err)}(z)-I\right|\right|_{L^{p}(\Sigma^{(2)})}&=
\left|\left|M^{(sol)}(z)V^{(2)}(z)M^{(sol)}(z)^{-1}-I\right|\right|_{L^{p}(\Sigma^{(2)})}\nonumber\\
&\leq\left|\left|V^{(2)}(z)-I\right|\right|_{L^{p}(\Sigma^{(2)})}=
O(e^{-2\varepsilon_0t}),
\end{align}
which indicates that the jump matrix $V^{(err)}(z)$ uniformly converges to the identity matrix.
Then, it follows from the small norm theory of RH problem \cite{KMM-2003,Deift-1994,DZ-CPAM} that the solution of RH problem \ref{RHP-Merr} can be uniquely expressed as
\begin{align}\label{Ez-sol}
M^{(err)}(z)=I+\frac{1}{2\pi i}\int_{\Sigma^{(2)}}\dfrac{\left( I+\rho(s)\right) (V^{(err)}(s)-I)}{s-z}ds,
\end{align}
where $\rho\in L^{2}(\Sigma^{(2)})$ is the unique solution of the operator equation
\begin{align}\label{Ez-oper}
(1-C_{V^{(err)}})\rho=C_{V^{(err)}}I,
\end{align}
where the operator $C_{V^{(err)}}:$ $L^2(\Sigma^{(2)})\to L^2(\Sigma^{(2)})$ denotes the Cauchy projection operator
\begin{align*}
C_{V^{(err)}}[f](z)=C_{-}[f(V^{(err)}-I)]=\lim_{k\rightarrow z}\int_{\Sigma^{(2)}}\frac{f(s)(V^{(err)}(s)-I)}{s-k}ds.
\end{align*}
The definition of the limit $C_-$ here is  that  the limit is obtained from the right side of the oriented contour $\Sigma^{(2)}$ with a non-tangent. The existence and uniqueness of the solution $\rho$ of the equation \eqref{Ez-oper} originates from the boundedness of Cauchy projection operator $C_-$, which immediately leads to
\begin{align}
\left|\left|C_{V^{(err)}}\right|\right|_{L^2(\Sigma^{(2)})\to L^2(\Sigma^{(2)})}=O(e^{-2\varepsilon_0t}).
\end{align}
Furthermore,
\begin{align}\label{est-rho}
||\rho||_{L^2(\Sigma^{(2)})}\lesssim \frac{||C_{V^{(err)}}||}{1-||C_{V^{(err)}}||}\lesssim O(e^{-2\varepsilon_0t}).
\end{align}

In conclusion, the existence, uniqueness and boundedness of the solution to RH problem \ref{RHP-Merr} can be obtained. Further, in order to recover the solution $q(x,t)$ of mCH equation, we need to consider not only the asymptotic behavior at singular point $z=\infty$, but also the asymptotic behavior at singular point $z=i$  shown in Proposition \ref{Merr-i}.
\begin{prop}\label{Merr-i}
As $z\to i$, the solution \eqref{Ez-sol} to RH problem \ref{RHP-Merr} is of
 the following expansion at $z=i$
\begin{align}\label{expansion}
E(z)=E(i)+E_1(z-i)+O((z-i)^2),
\end{align}
where
\begin{align}
&E(i)=I+\frac{1}{2\pi i}\int_{\Sigma^{(2)}}\frac{(I+\rho(s))(V^{(err)}-I)}{s-i}\ ds,\\
&E_1=\frac{1}{2\pi i}\int_{\Sigma^{(2)}}\frac{(I+\rho(s))(V^{(err)}-I)}{(s-i)^2}\ ds.
\end{align}
The coefficients $E(i)$ and $E_1$ admit the following estimates
\begin{align}
|E(i)-I|\lesssim O(e^{-2\varepsilon_0t}),\qquad |E_1|\lesssim O(e^{-2\varepsilon_0t}).
\end{align}
\end{prop}

\begin{proof}
By directly calculating the expansion of the solution at $z=i$, the equation \eqref{expansion} can be obtained.  We now turn our attention to the estimation of coefficients $E(i)$ and $E_1$. Combining the estimate \eqref{est-Verr} and \eqref{est-rho} yields
\begin{align}
|M^{(err)}-I|\leq|(1-C_{V^{(err)}})(\rho)|+|C_{V^{(err)}}(\rho)|\lesssim O(e^{-2\varepsilon_0t}).
\end{align}
Additionally, observing the fact $|s-i|^{-2}$ is bounded on the contour $\Sigma^{(2)}$, one then has
\begin{align}
|E_1|\leq ||V^{(err)}-I||_{L^1}+||\rho||_{L^2}||V^{(err)}-I||_{L^2}\lesssim O(e^{-2\varepsilon_0t}),
\end{align}
which completes the proof of the proposition.
\end{proof}

\subsection{Analysis on pure $\overline{\partial}$-problem}\label{subsec4.5}
This section will study the pure $\overline{\partial}$-problem \ref{RHP-Dbar} for $M^{(3)}(z)$ defined in \eqref{M3}. As \cite{AIHP,Cu-CMP}, the solution to $\overline{\partial}$-problem \ref{RHP-Dbar} can be expressed by the integral equation
\begin{equation}\label{M3-sol}
M^{(3)}(z)=I+\frac{1}{\pi}\iint_\mathbb{C}\dfrac{M^{(3)}(s)W^{(3)} (s)}{s-z}dA(s),
\end{equation}
where $A(s)$ denotes the Lebesgue measure on the complex plane $\mathbb{C}$. With Cauchy-Green integral operator, the solution reduces to
\begin{align}\label{M3-oper}
M^{(3)}(z)=I\cdot(I-\mathbb{P}_z)^{-1},
\end{align}
where
\begin{align*}
f\mathbb{P}_z(z)=\frac{1}{\pi}\int\int_{\D{C}}\frac{f(s)W^{(3)}(s)}{s-z}\ dA(s).
\end{align*}

The following proposition shows that the operator $\mathbb{P}_z$ is small-norm for a sufficiently large $t$, which means that the operator equation \eqref{M3-oper} is meaningful and $(1-\mathbb{P}_z)^{-1}$ exists.
\begin{prop}\label{pp-oper-P}
As $t\to\infty$, the Cauchy-Green integral operator $\mathbb{P}_z$ satisfies the following inequality
\begin{align}
||\mathbb{P}_z||_{L^{\infty}\to L^{\infty}}\lesssim t^{-1/2}.
\end{align}
\end{prop}
\begin{proof}
With the definition of Cauchy-Green integral operator $\mathbb{P}_z$
\begin{align}\label{CG-1}
\left|\left|f\mathbb{P}_z\right|\right|_{L^{\infty}}\lesssim ||f||_{L^{\infty}}\int\int_{\D{C}}\frac{|W^{(3)}(s)|}{
|s-z|}\ dA(s),
\end{align}
where $f\in L^{\infty}$ and $W^{(3)}(z)=M^{(rhp)}(z)\overline{\partial}R^{(2)}(z)M^{(rhp)}(z)^{-1}$, which implies that  we only need to estimate the second factor of the above equation. In what follows, we will prove the case $\xi>2$ for $z\in\Omega_1$ in detail, and other areas can be similarly proved. Observing that $W^{(3)}(z)\equiv0$ as $z\notin\overline{\Omega}$ and Proposition \ref{pp-Mrhp} indicates the solution $M^{(rhp)}(z)$ and $M^{(rhp)}(z)^{-1}$ are bounded for $z\in\overline{\Omega}$.  Thus \eqref{CG-1} is equivalent to
\begin{align}\label{CG-2}
\int\int_{\Omega_1}\frac{|W^{(3)}(s)|}{
|s-z|}\ dA(s)\leq \int\int_{\Omega_1}\frac{|\overline{\partial}R_1(s)\cdot e^{2it\theta}|}{
|s-z|}\ dA(s).
\end{align}

It follows from Corollary \ref{cor-im} and the $\overline{\partial}$-derivative \eqref{R2-D} that the inequality \eqref{CG-2}
can be decompose into two  integral equations
\begin{align}
\int\int_{\Omega_1}\frac{|\overline{\partial}R_1(s)| e^{-2t\Im\theta}}{
|s-z|}\ dA(s)\lesssim J_1+J_2,
\end{align}
where
\begin{subequations}
\begin{align}
&J_1=\int\int_{\Omega_1}\frac{|r'_1(s)| e^{-2t\Im\theta}}{
|s-z|}\ dA(s),\label{J1}\\
&J_2=\int\int_{\Omega_1}\frac{|s|^{-1/2} e^{-2t\Im\theta}}{
|s-z|}\ dA(s).\label{J2}
\end{align}
\end{subequations}

Let $z=\alpha+i\beta$ and $s=u+iv$. Observing that for $1\leq q\leq+\infty$ with $\frac{1}{p}+\frac{1}{q}=1$
\begin{align}
\left|\left|\frac{1}{|s-z|}\right|\right|&=\left(\int_{v}^{+\infty}\left(\frac{1}{(u-\alpha)^2+(v-\beta)^2}\right)^{q/2}
\right)^{1/q}\nonumber\\
&=\left(\int_{v}^{+\infty}\left(1+\left(\frac{u-\alpha}{v-\beta}\right)\right)^{-q/2}\ d\left(
\frac{u-\alpha}{v-\beta}\right)\cdot(v-\beta)^{-q+1}\right)^{-1/q}\nonumber\\
&\lesssim |v-\beta|^{1/q-1}.
\end{align}

On the other hand, for $z\in\Omega_1$ in the region $\xi>2$, one has $e^{-2t\Im\theta}\leq e^{-\tau(\xi)v}$. Thus, applying H\"{o}lder inequality yields
\begin{align}
J_1&\leq\int_0^\infty\int_{v}^\infty\frac{|r'_1(s)| e^{-2t\Im\theta}}{
|s-z|}\ dudv\nonumber\\
&\leq\int_0^\infty
\left|\left|\frac{1}{|s-z|}\right|\right|_{L^2(\D{R}^+)}\left|\left|\frac{1}{|r'_1(s)|}\right|\right|_{L^2(\D{R}^+)}
e^{-\tau(\xi)v}\ dudv\nonumber\\
&\leq \int_0^\infty|v-\beta|^{-1/2}e^{-\tau(\xi)v}\ dv\nonumber\\
&=\int_0^\beta|v-\beta|^{-1/2}e^{-\tau(\xi)v}\ dv+\int_\beta^\infty|v-\beta|^{-1/2}e^{-\tau(\xi)v}\ dv\triangleq
J_1^{(1)}+J_1^{(2)}.\label{CG-3}
\end{align}

For $J_1^{(1)}$,
\begin{align*}
J_1^{(1)}&=\int_0^\beta\frac{1}{\sqrt{\beta-v}}e^{-\tau(\xi)v}\ dv=\int_0^1\frac{\sqrt{\beta}e^{-t\tau(\xi)\beta \varpi}}{\sqrt{1-\varpi}}\ d\varpi\nonumber\\
&\leq t^{-1/2}\int_0^1\frac{1}{\sqrt{\varpi(1-\varpi)}}\ d\varpi\lesssim t^{-1/2},
\end{align*}
where the variable $\varpi=v/\beta$. Similarly, we can proof $J_1^{(2)}\lesssim t^{-1/2}$, which completes the proof of \eqref{J1}. Now let's turn our attention to the second integral equation \eqref{J2}. Also, it's easy to check that for $1\leq q\leq+\infty$ with $\frac{1}{p}+\frac{1}{q}=1$
\begin{align*}
\left|\left|\frac{1}{|s|}\right|\right|_{L^{p}(v,+\infty)}
=\left(\int_v^\infty\left(\frac{1}{\sqrt{u^2+v^2}}\right)^{p}\ du\right)^{1/p}\lesssim v^{1/p-1}.
\end{align*}
Then combined with Cauchy-Schwarz inequality, the following results are obtained
\begin{align}
J_2&\leq\int_0^\infty\left|\left|\frac{1}{|s|}\right|\right|_{L^{p}(\D{R}^+)}
\left|\left|\frac{1}{|s-z|}\right|\right|_{L^{q}(\D{R}^+)}e^{-\tau(\xi)v}\ dv\nonumber\\
&\leq\int_0^\infty v^{1/p-1/2}|v-\beta|^{1/q-1}e^{-\tau(\xi)v}\ dv \nonumber\\
&\leq\int_0^\beta v^{1/p-1/2}|v-\beta|^{1/q-1}e^{-\tau(\xi)v}\ dv
+\int_\beta^\infty v^{1/p-1/2}|v-\beta|^{1/q-1}e^{-\tau(\xi)v}\ dv.
\end{align}
Similar to \eqref{CG-3}, one has $J_2\lesssim t^{-1/2}$. To sum up, we have completed the proof of proposition \ref{pp-oper-P}.
\end{proof}

As $z=i$, integral equation \eqref{M3-sol} becomes
\begin{align}\label{M3i}
M^{(3)}(i)=I+\frac{1}{\pi}\int\int_{\D{C}}\frac{M^{(3)}(s)W^{(3)}(s)}{s-i}\ dA(s).
\end{align}
The asymptotic behavior of $M^{(3)}(z)$ at $z=i$ is required when constructing the solution of the mCH equation, which is reflected in the following proposition.
\begin{prop}\label{pp-M3i}
The integral equation \eqref{M3-sol} admits
\begin{align}\label{est-M3i}
\left|\left|M^{(3)}(z)-I\right|\right|=
\left|\left|\frac{1}{\pi}\int\int_{\D{C}}\frac{M^{(3)}(s)W^{(3)}(s)}{s-i}\ dA(s)\right|\right|\lesssim
O(t^{-1/2}).
\end{align}
As $z\to i$, $M^{(3)}(z)$ has the following expansion
\begin{align}
M^{(3)}(z)=M^{(3)}(i)+M^{(3)}_{1}(x,t)(z-i)+O((z-i)^2),
\end{align}
where
\begin{align}
M^{(3)}_{1}(x,t)=
\frac{1}{\pi}\int\int_{\D{C}}\frac{M^{(3)}(s)W^{(3)}(s)}{(s-i)^2}\ dA(s),
\end{align}
which such that as $t\to\infty$
\begin{align}\label{est-M3i1}
\left|M^{(3)}_{1}(x,t)\right|\lesssim O(t^{-1/2}).
\end{align}
\end{prop}
\begin{proof}
It is similar to the proof process of Proposition \ref{pp-oper-P}, and it follows from \eqref{M3-oper} and  \eqref{pp-oper-P} that $||M^{(3)}||_{L^{\infty}}\lesssim 1$.
Take the region $\xi\in(2,+\infty)$ as an example, the other situations are similar. Let $s=u+iv\in\Omega_{1}$, then
\begin{align*}
M^{(3)}(z)'=\frac{1}{\pi}\int\int_{\D{C}}\frac{M^{(3)}(s)W^{(3)}(s)}{(s-i)^2}\ dA(s)
\end{align*}
and
\begin{align}
\left|\left|M^{(3)}(z)-I\right|\right|&=\frac{1}{\pi}
\left|\left|\frac{1}{\pi}\int\int_{\D{C}}\frac{M^{(3)}(s)W^{(3)}(s)}{s-i}\ dA(s)\right|\right|\nonumber \\
&\lesssim\int\int_{\D{C}}\frac{|\overline{\partial}R_1(s)|e^{-\tau(\xi)tv}}{|s-i|}\ dudv
\triangleq J_3+J_4,
\end{align}
where
\begin{subequations}
\begin{align}
&J_3=\int\int_{\Omega_1}\frac{|r'_1(s)|e^{-\tau(\xi)tv}}{|s-i|}\ dudv,\label{J3} \\
&J_4=\int\int_{\Omega_1}\frac{|s|^{-1}e^{-\tau(\xi)tv}}{|s-i|}\ dudv.\label{J4}
\end{align}
\end{subequations}

Now turn our attention to proving $J_3$. Observing that $r\in H^{1,1}(\D{R})$, then $r'\in L^{1}(\D{R})$. Applying H\"{o}lder inequality yields
\begin{align*}
J_3&\leq\int_{0}^{+\infty}\int_{v}^{+\infty}\frac{|r'_1(s)|e^{-\tau(\xi)tv}}{|s-i|}\ dudv
\nonumber \\
&\leq\int_{0}^{+\infty}\sqrt{2}||r'_1(s)||_{L^1(\D{R})}e^{-\tau(\xi)tv}\ dv\nonumber \\
&\lesssim \int_{0}^{+\infty}e^{-\tau(\xi)tv}\ dv\lesssim O(t^{-1}).
\end{align*}
For $J_4$, it follows from the fact $\left|\left||s|^{-1}\right|\right|_{L^{p}(v,+\infty)}
\lesssim v^{1/p-1}$ for $1\leq q\leq+\infty$ with $\frac{1}{p}+\frac{1}{q}=1$ and $\left|\left||s-i|^{-1}\right|\right|_{L^{q}(v,+\infty)}\lesssim|v-1|^{1/q-1}$. Then
\begin{align*}
J_4&=\int\int_{\Omega_1}\frac{|s|^{-1}e^{-\tau(\xi)tv}}{|s-i|}\ dudv\nonumber \\
&\leq \int_{0}^{+\infty}e^{-\tau(\xi)tv}\left|\left||s-i|^{-1}\right|\right|_{L^{q}}
\left|\left||s|^{-1}\right|\right|_{L^{p}}\ dv\nonumber \\
&\lesssim\int_{0}^{+\infty}e^{-\tau(\xi)tv}|v-1|^{1/q-1}|v|^{1/p-1}\ dv\nonumber \\
&=\int_{0}^{1}e^{-\tau(\xi)tv}(1-v)^{1/q-1}v^{1/p-1}\ dv+
\int_{1}^{+\infty}e^{-\tau(\xi)tv}(v-1)^{1/q-1}v^{1/p-1}\ dv
\nonumber \\
&\triangleq J_4^{(1)}+J_4^{(2)}.
\end{align*}
As in \cite{AIHP}, one can obtain  $J_4\lesssim O(t^{-1/2})$. Similarly, inequality \eqref{est-M3i1} can be proved by the same method.
\end{proof}

\section{Deformation of RH problem for $\xi\in(-1/4,0)$and $\xi\in(0,2)$}\label{sec5}
Starting from this section, we will analyze the long-time asymptotic behavior of solutions in region $\xi\in(-1/4,0)$ and region $\xi\in(0,2)$, corresponding to the existence of eight phase points $\{\xi_i\}_{i=1}^{8}$ and four phase points $\{\xi_i\}_{i=1}^{4}$ respectively, which is significantly different from that in regions $\xi\in(-\infty,-1/4)$ and $\xi\in(2,+\infty)$.  In the analysis process, some matrix-valued functions here are equivalent to the matrix-valued functions in the above two areas, which will not be repeated, including the function $T(z):=T(z,\xi)$ defined in \eqref{T} and the matrix $M^{(1)}(z)$ defined by \eqref{M1} satisfied RH problem \ref{RHP-M1}. On the other hand, due to the existence of steady-state phase points, we need to partition the real axis $\D{R}$. This is because in different intervals, the jump matrix has different forms of triangular decomposition as in Section \ref{sec3}, which will affect the continuous extansion of the next jump matrix. Thus, we define the intervals for $\xi\in(-1/4,0)$ and $\xi\in(0,2)$ as
\begin{align}
I_{j1}=I_{j2}=\left\{ \begin{array}{ll}
\left(\xi_j ,\frac{\xi_j+\xi_{j-1}}{2}\right),\    & j=1,3,5,7,\\[10pt]
\left( \frac{\xi_j+\xi_{j+1}}{2},\xi_j\right) ,   &j=2,4,6,8,
\end{array}\right.\\
I_{j3}=I_{j4}=\left\{ \begin{array}{ll}
\left( \frac{\xi_j+\xi_{j+1}}{2},\xi_j\right) ,\    & j=1,3,5,7,\\[10pt]
\left(\xi_j ,\frac{\xi_j+\xi_{j-1}}{2}\right),   &j=2,4,6,8.
\end{array}\right.\label{In2}
\end{align}
and
\begin{align}
I_{j1}=I_{j2}=\left\{ \begin{array}{ll}
\left( \frac{\xi_j+\xi_{j+1}}{2},\xi_j\right) ,\    & j=1,3,\\[10pt]
\left(\xi_j ,\frac{\xi_j+\xi_{j-1}}{2}\right),   &j=2,4,
\end{array}\right.\label{In1}\\
I_{j3}=I_{j4}=\left\{ \begin{array}{ll}
\left(\xi_j ,\frac{\xi_j+\xi_{j-1}}{2}\right),\    & j=1,3,\\[10pt]
\left( \frac{\xi_j+\xi_{j+1}}{2},\xi_j\right) ,   &j=2,4,
\end{array}\right.
\end{align}
shown in Fig. \ref{phase}.

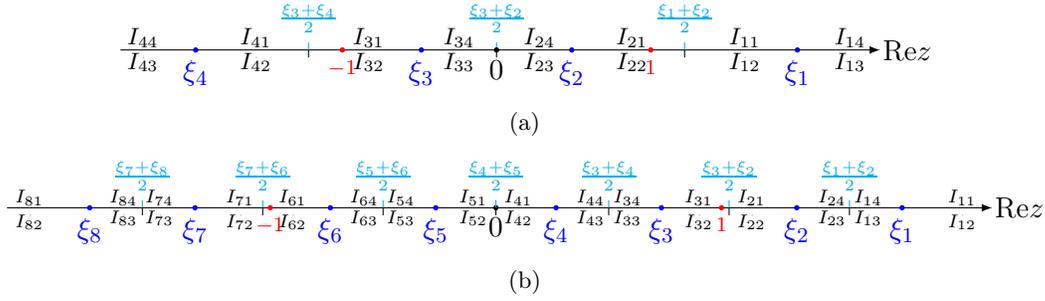
\begin{figure}[h]
	\subfigure[]{
		\begin{tikzpicture}
			\draw[-](-5,0)--(5,0)node[right]{ Re$z$};
\draw[-latex](5,0)--(5.1,0);
			\draw[cyan](2.5,0)--(2.5,0.1)node[above]{\scriptsize$\frac{\xi_1+\xi_2}{2}$};
			\draw(2.5,0)--(2.5,-0.1);
			\draw[cyan](-2.5,0)--(-2.5,0.1)node[above]{\scriptsize$\frac{\xi_3+\xi_4}{2}$};
			\draw(-2.5,0)--(-2.5,-0.1);
			\draw[cyan](0,0)--(0,0.1)node[above]{\scriptsize$\frac{\xi_3+\xi_2}{2}$};
			\draw(0,0)--(0,-0.1);
			\coordinate (I) at (0,0);
			\fill (I) circle (1pt) node[below] {$0$};
			\coordinate (A) at (-4,0);
			\fill[blue] (A) circle (1pt) node[below] {$\xi_4$};
			\coordinate (b) at (-1,0);
			\fill[blue] (b) circle (1pt) node[below] {$\xi_3$};
			\coordinate (e) at (4,0);
			\fill[blue] (e) circle (1pt) node[below] {$\xi_1$};
			\coordinate (f) at (1,0);
			\fill[blue] (f) circle (1pt) node[below] {$\xi_2$};
			\coordinate (ke) at (4.7,0.1);
			\fill (ke) circle (0pt) node[below] {\footnotesize$I_{13}$};
			\coordinate (k1e) at (4.7,-0.1);
			\fill (k1e) circle (0pt) node[above] {\footnotesize$I_{14}$};
			\coordinate (le) at (3.3,0.1);
			\fill (le) circle (0pt) node[below] {\footnotesize$I_{12}$};
			\coordinate (l1e) at (3.3,-0.1);
			\fill (l1e) circle (0pt) node[above] {\footnotesize$I_{11}$};
			\coordinate (n2) at (0.57,0.1);
			\fill (n2) circle (0pt) node[below] {\footnotesize$I_{23}$};
			\coordinate (n12) at (0.57,-0.1);
			\fill (n12) circle (0pt) node[above] {\footnotesize$I_{24}$};
			\coordinate (m2) at (1.8,0.1);
			\fill (m2) circle (0pt) node[below] {\footnotesize$I_{22}$};
			\coordinate (m12) at (1.8,-0.1);
			\fill (m12) circle (0pt) node[above] {\footnotesize$I_{21}$};
			\coordinate (k) at (-4.7,0.1);
			\fill (k) circle (0pt) node[below] {\footnotesize$I_{43}$};
			\coordinate (k1) at (-4.7,-0.1);
			\fill (k1) circle (0pt) node[above] {\footnotesize$I_{44}$};
			\coordinate (l) at (-3.2,0.1);
			\fill (l) circle (0pt) node[below] {\footnotesize$I_{42}$};
			\coordinate (l1) at (-3.2,-0.1);
			\fill (l1) circle (0pt) node[above] {\footnotesize$I_{41}$};
			\coordinate (n) at (-0.5,0.1);
			\fill (n) circle (0pt) node[below] {\footnotesize$I_{33}$};
			\coordinate (n1) at (-0.5,-0.1);
			\fill (n1) circle (0pt) node[above] {\footnotesize$I_{34}$};
			\coordinate (m) at (-1.7,0.1);
			\fill (m) circle (0pt) node[below] {\footnotesize$I_{32}$};
			\coordinate (m1) at (-1.7,-0.1);
			\fill (m1) circle (0pt) node[above] {\footnotesize$I_{31}$};
			\coordinate (c) at (-2.05,0);
			\fill[red] (c) circle (1pt) node[below] {\scriptsize$-1$};
			\coordinate (d) at (2.05,0);
			\fill[red] (d) circle (1pt) node[below] {\scriptsize$1$};
			\end{tikzpicture}
			\label{pcase1}}
		\subfigure[]{
		\begin{tikzpicture}
		\draw[-](-6.5,0)--(6.5,0)node[right]{ Re$z$};
\draw[-latex](6.5,0)--(6.6,0);
		\coordinate (I) at (0,0);
		\fill (I) circle (1pt) node[below] {$0$};
		\coordinate (c) at (-3,0);
		\fill[red] (c) circle (1pt) node[below] {\scriptsize$-1$};
		\coordinate (D) at (3,0);
		\fill[red] (D) circle (1pt) node[below] {\scriptsize$1$};
		\draw[cyan](-1.5,0)--(-1.5,0.1)node[above]{\scriptsize$\frac{\xi_5+\xi_6}{2}$};
		\draw(-1.5,0)--(-1.5,-0.1);
		\draw[cyan](-4.7,0)--(-4.7,0.1)node[above]{\scriptsize$\frac{\xi_7+\xi_8}{2}$};
		\draw(-4.7,0)--(-4.7,-0.1);
		\draw[cyan](-3.1,0)--(-3.1,0.1)node[above]{\scriptsize$\frac{\xi_7+\xi_6}{2}$};
		\draw(-3.1,0)--(-3.1,-0.1);
		\draw[cyan](1.5,0)--(1.5,0.1)node[above]{\scriptsize$\frac{\xi_3+\xi_4}{2}$};
		\draw(1.5,0)--(1.5,-0.1);
		\draw[cyan](4.7,0)--(4.7,0.1)node[above]{\scriptsize$\frac{\xi_1+\xi_2}{2}$};
		\draw(4.7,0)--(4.7,-0.1);
		\draw[cyan](3.1,0)--(3.1,0.1)node[above]{\scriptsize$\frac{\xi_3+\xi_2}{2}$};
		\draw(3.1,0)--(3.1,-0.1);
		\draw[cyan](0,0)--(0,0.1)node[above]{\scriptsize$\frac{\xi_4+\xi_5}{2}$};
		\draw(0,0)--(0,-0.1);
		\coordinate (A) at (-5.4,0);
		\fill[blue] (A) circle (1pt) node[below] {$\xi_8$};
		\coordinate (b) at (-4,0);
		\fill[blue] (b) circle (1pt) node[below] {$\xi_7$};
		\coordinate (C) at (-0.8,0);
		\fill[blue] (C) circle (1pt) node[below] {$\xi_5$};
		\coordinate (d) at (-2.2,0);
		\fill[blue] (d) circle (1pt) node[below] {$\xi_6$};
		\coordinate (E) at (5.4,0);
		\fill[blue] (E) circle (1pt) node[below] {$\xi_1$};
		\coordinate (R) at (4,0);
		\fill[blue] (R) circle (1pt) node[below] {$\xi_2$};
		\coordinate (T) at (0.8,0);
		\fill[blue] (T) circle (1pt) node[below] {$\xi_4$};
		\coordinate (Y) at (2.2,0);
		\fill[blue] (Y) circle (1pt) node[below] {$\xi_3$};
		\coordinate (q) at (6.2,-0.1);
		\fill (q) circle (0pt) node[above] {\tiny$I_{11}$};
		\coordinate (q1) at (6.2,0.05);
		\fill (q1) circle (0pt) node[below] {\tiny$I_{12}$};
		\coordinate (w) at (4.95,-0.1);
		\fill (w) circle (0pt) node[above] {\tiny$I_{14}$};
		\coordinate (w1) at (4.95,0.1);
		\fill (w1) circle (0pt) node[below] {\tiny$I_{13}$};
		\coordinate (e) at (4.47,-0.1);
		\fill (e) circle (0pt) node[above] {\tiny$I_{24}$};
		\coordinate (e1) at (4.47,0.1);
		\fill (e1) circle (0pt) node[below] {\tiny$I_{23}$};
		\coordinate (r) at (3.4,-0.1);
		\fill (r) circle (0pt) node[above] {\tiny$I_{21}$};
		\coordinate (r1) at (3.4,0.05);
		\fill (r1) circle (0pt) node[below] {\tiny$I_{22}$};
		\coordinate (t) at (2.7,-0.1);
		\fill (t) circle (0pt) node[above] {\tiny$I_{31}$};
		\coordinate (t1) at (2.7,0.05);
		\fill (t1) circle (0pt) node[below] {\tiny$I_{32}$};
		\coordinate (y) at (1.75,-0.1);
		\fill (y) circle (0pt) node[above] {\tiny$I_{34}$};
		\coordinate (y1) at (1.75,0.1);
		\fill (y1) circle (0pt) node[below] {\tiny$I_{33}$};
		\coordinate (l) at (1.26,-0.1);
		\fill (l) circle (0pt) node[above] {\tiny$I_{44}$};
		\coordinate (l1) at (1.26,0.1);
		\fill (l1) circle (0pt) node[below] {\tiny$I_{43}$};
		\coordinate (k) at (0.3,-0.1);
		\fill (k) circle (0pt) node[above] {\tiny$I_{41}$};
		\coordinate (k1) at (0.3,0.1);
		\fill (k1) circle (0pt) node[below] {\tiny$I_{42}$};
		\coordinate (q8) at (-6.2,-0.1);
		\fill (q8) circle (0pt) node[above] {\tiny$I_{81}$};
		\coordinate (q18) at (-6.2,0.05);
		\fill (q18) circle (0pt) node[below] {\tiny$I_{82}$};
		\coordinate (w8) at (-4.95,-0.1);
		\fill (w8) circle (0pt) node[above] {\tiny$I_{84}$};
		\coordinate (w18) at (-4.95,0.1);
		\fill (w18) circle (0pt) node[below] {\tiny$I_{83}$};
		\coordinate (e7) at (-4.47,-0.1);
		\fill (e7) circle (0pt) node[above] {\tiny$I_{74}$};
		\coordinate (e17) at (-4.47,0.1);
		\fill (e17) circle (0pt) node[below] {\tiny$I_{73}$};
		\coordinate (7r) at (-3.4,-0.1);
		\fill (7r) circle (0pt) node[above] {\tiny$I_{71}$};
		\coordinate (r17) at (-3.4,0.05);
		\fill (r17) circle (0pt) node[below] {\tiny$I_{72}$};
		\coordinate (t6) at (-2.7,-0.1);
		\fill (t6) circle (0pt) node[above] {\tiny$I_{61}$};
		\coordinate (t16) at (-2.7,0.05);
		\fill (t16) circle (0pt) node[below] {\tiny$I_{62}$};
		\coordinate (y6) at (-1.75,-0.1);
		\fill (y6) circle (0pt) node[above] {\tiny$I_{64}$};
		\coordinate (y16) at (-1.75,0.1);
		\fill (y16) circle (0pt) node[below] {\tiny$I_{63}$};
		\coordinate (l5) at (-1.26,-0.1);
		\fill (l5) circle (0pt) node[above] {\tiny$I_{54}$};
		\coordinate (l15) at (-1.26,0.1);
		\fill (l15) circle (0pt) node[below] {\tiny$I_{53}$};
		\coordinate (k5) at (-0.3,-0.1);
		\fill (k5) circle (0pt) node[above] {\tiny$I_{51}$};
		\coordinate (k15) at (-0.3,0.1);
		\fill (k15) circle (0pt) node[below] {\tiny$I_{52}$};
		\end{tikzpicture}
		\label{phase2}}
	\caption{ (a)  corresponds to the region $0\leq\xi<2$. (b) corresponds to the  $-\frac{1}{4}<\xi<0$.}
	\label{phase}
\end{figure}

\subsection{Contour deformation}
In order to get a new RH problem, which has no jump on the real axis $\D{R}$, the function  $M^{(1)}(z)$ needs to be continuously extended. In essence, the oscillatory functions $e^{\pm2it\theta}$ in the jump matrix is treated according to its attenuation property. Before that, some contours need to be introduced to better describe the deformation of the jump matrix.

For $\ell\in\left(0,\frac{|\xi_{j+(-1)^j}-\xi_j|}{2\cos\phi}\right)$ with $\xi\in(-1/4,0)$, define
\begin{align*}
\Sigma_{jk}=\left\{\begin{aligned}
&\xi_j+e^{i\left((k/2+1/2+j)\pi+(-1)^{j+1}\phi\right)}\ell,&k=1,3,\\
&\xi_j+e^{i\left((k/2+j)\pi+(-1)^{j}\phi\right)}\ell,&k=2,4.
\end{aligned}\right.
\end{align*}

For $\xi\in(0,2)$, define
\begin{align*}
\Sigma_{jk}=\left\{\begin{aligned}
&\xi_j+e^{i\left((k/2+j)\pi+(-1)^{j}\phi\right)}\ell,&k=1,3,\\
&\xi_j+e^{i\left((k/2+1/2+j)\pi+(-1)^{j+1}\phi\right)}\ell,&k=2,4.
\end{aligned}\right.
\end{align*}

As $j=1,8$ with $\xi\in(-1/4,0)$
\begin{align*}
\left\{\begin{aligned}
&\Sigma_{j1}(\xi)=\xi_j+e^{(1+j)\pi i+(-1)^{j+1}i\phi}\D{R}^+,\\
&\Sigma_{j2}(\xi)=\xi_j+e^{(1+j)\pi i+(-1)^{j}i\phi}\D{R}^+,\\
&\Sigma_{j3}(\xi)=\xi_j+e^{j\pi i+(-1)^{j+1}i\phi}\D{R}^+,\\
&\Sigma_{j4}(\xi)=\xi_j+e^{j\pi i+(-1)^{j}i\phi}\D{R}^+,\\
&\Sigma_p^{\pm}=\frac{\xi_{p+1}+\xi}{2}+e^{i\pi}\ell,\quad p=1,2,\ldots,7.
\end{aligned}\right.
\end{align*}

As $j=1,4$ with $\xi\in(0,2)$
\begin{align*}
\left\{\begin{aligned}
&\Sigma_{j1}(\xi)=\xi_j+e^{j\pi i+(-1)^{j}i\phi}\D{R}^+,\\
&\Sigma_{j2}(\xi)=\xi_j+e^{j\pi i+(-1)^{j+1}i\phi}\D{R}^+,\\
&\Sigma_{j3}(\xi)=\xi_j+e^{(1+j)\pi i+(-1)^{j}i\phi}\D{R}^+,\\
&\Sigma_{j4}(\xi)=\xi_j+e^{(1+j)\pi i+(-1)^{j+1}i\phi}\D{R}^+,\\
&\Sigma_p^{\pm}=\frac{\xi_{p+1}+\xi}{2}+e^{i\pi}\ell,\quad p=2,3,4.
\end{aligned}\right.
\end{align*}

Applying the above contours defines the following regions and new contours
\begin{align}
	&\Omega(\xi)=\underset{j=1,..,n(\xi)}{\underset{k=1,...,4,}{\cup}}\Omega_{jk},\hspace{0.5cm} \Omega_\pm(\xi)=\mathbb{C}\setminus\Omega,\\
	&\Sigma^\sharp(\xi)=\left( \underset{j=1,..,n(\xi)}{\underset{k=1,...,4,}{\cup}}\Sigma_{jk}\right) \cup\left(\underset{j=1,...,n(\xi)}{\cup}\Sigma_{j}^{\pm} \right) ,\hspace{0.5cm}\\
	&\Sigma^{(2)}(\xi)=\Sigma^\sharp(\xi)\underset{n\in N\setminus\Lambda}{\cup}\left( \partial\overline{\mathbb{P}}_n\cup\partial\mathbb{P}_n\right).
\end{align}

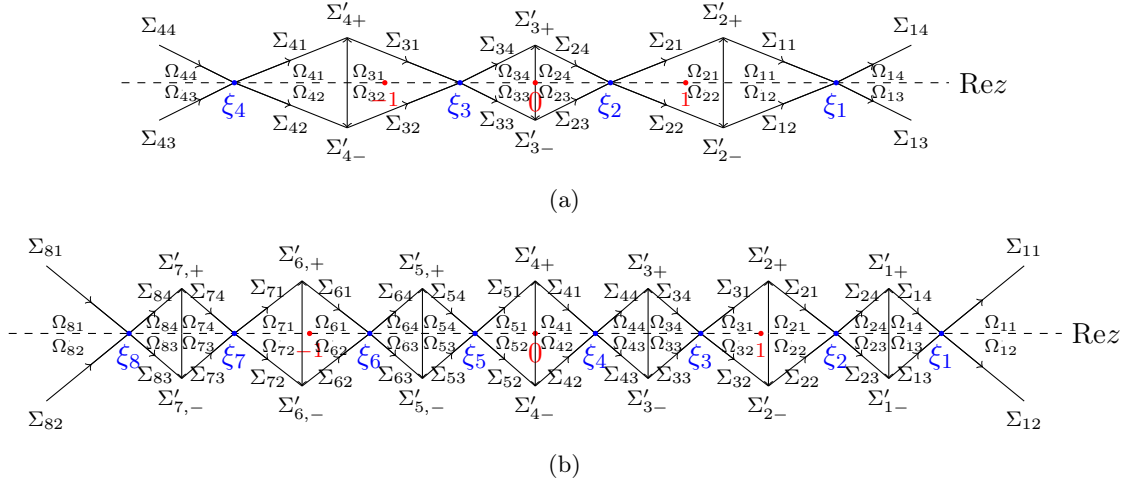
\begin{figure}[H]
\subfigure[]{
\begin{tikzpicture}
\draw(-4,0)--(-5,0.5)node[above]{\scriptsize$\Sigma_{44}$};
\draw[-<](-4,0)--(-4.5,0.25);
\draw(-4,0)--(-2.5,0.6);
\draw[->](-4,0)--(-3.25,-0.3)node[below]{\scriptsize$\Sigma_{42}$};
\draw(-4,0)--(-5,-0.5)node[below]{\scriptsize$\Sigma_{43}$};
\draw[->](-4,0)--(-3.25,0.3)node[above]{\scriptsize$\Sigma_{41}$};
\draw(-4,0)--(-2.5,-0.6);
\draw[-<](-4,0)--(-4.5,-0.25);
\draw(-1,0)--(0,0.5);
\draw[->](-1,0)--(-0.5,0.25)node[above]{\scriptsize$\Sigma_{34}$};
\draw(-1,0)--(-2.5,0.6);
\draw[-<](-1,0)--(-1.75,-0.3)node[below]{\scriptsize$\Sigma_{32}$};
\draw(-1,0)--(0,-0.5);
\draw[-<](-1,0)--(-1.75,0.3)node[above]{\scriptsize$\Sigma_{31}$};
\draw(-1,0)--(-2.5,-0.6);
\draw[->](-1,0)--(-0.5,-0.25)node[below]{\scriptsize$\Sigma_{33}$};
\draw[dashed](-5.5,0)--(5.5,0)node[right]{ Re$z$};
\draw(1,0)--(0,0.5);
\draw[-<](1,0)--(0.5,0.25)node[above]{\scriptsize$\Sigma_{24}$};
\draw(1,0)--(2.5,0.6);
\draw[->](1,0)--(1.75,-0.3)node[below]{\scriptsize$\Sigma_{22}$};
\draw(1,0)--(0,-0.5);
\draw[->](1,0)--(1.75,0.3)node[above]{\scriptsize$\Sigma_{21}$};
\draw(1,0)--(2.5,-0.6);
\draw[-<](1,0)--(0.5,-0.25)node[below]{\scriptsize$\Sigma_{23}$};
\draw(4,0)--(5,0.5)node[above]{\scriptsize$\Sigma_{14}$};
\draw[->](4,0)--(4.5,0.25);
\draw(4,0)--(2.5,0.6);
\draw[-<](4,0)--(3.25,-0.3)node[below]{\scriptsize$\Sigma_{12}$};
\draw(4,0)--(5,-0.5)node[below]{\scriptsize$\Sigma_{13}$};
\draw[-<](4,0)--(3.25,0.3)node[above]{\scriptsize$\Sigma_{11}$};
\draw(4,0)--(2.5,-0.6);
\draw[->](4,0)--(4.5,-0.25);
\draw[->](2.5,0)--(2.5,0.6)node[above]{\scriptsize$\Sigma_{2+}'$};
\draw[->](2.5,0)--(2.5,-0.6)node[below]{\scriptsize$\Sigma_{2-}'$};
\draw[->](-2.5,0)--(-2.5,0.6)node[above]{\scriptsize$\Sigma_{4+}'$};
\draw[->](-2.5,0)--(-2.5,-0.6)node[below]{\scriptsize$\Sigma_{4-}'$};
\draw[->](0,0)--(0,0.5)node[above]{\scriptsize$\Sigma_{3+}'$};
\draw[->](0,0)--(0,-0.5)node[below]{\scriptsize$\Sigma_{3-}'$};
\coordinate (I) at (0,0);
\fill[red] (I) circle (1pt) node[below] {$0$};
\coordinate (A) at (-4,0);
\fill[blue] (A) circle (1pt) node[below] {$\xi_4$};
\coordinate (b) at (-1,0);
\fill[blue] (b) circle (1pt) node[below] {$\xi_3$};
\coordinate (e) at (4,0);
\fill[blue] (e) circle (1pt) node[below] {$\xi_1$};
\coordinate (f) at (1,0);
\fill[blue] (f) circle (1pt) node[below] {$\xi_2$};
\coordinate (ke) at (4.7,0.1);
\fill (ke) circle (0pt) node[below] {\tiny$\Omega_{13}$};
\coordinate (k1e) at (4.7,-0.1);
\fill (k1e) circle (0pt) node[above] {\tiny$\Omega_{14}$};
\coordinate (le) at (3,0.1);
\fill (le) circle (0pt) node[below] {\tiny$\Omega_{12}$};
\coordinate (l1e) at (3,-0.1);
\fill (l1e) circle (0pt) node[above] {\tiny$\Omega_{11}$};
\coordinate (n2) at (0.27,0.1);
\fill (n2) circle (0pt) node[below] {\tiny$\Omega_{23}$};
\coordinate (n12) at (0.27,-0.1);
\fill (n12) circle (0pt) node[above] {\tiny$\Omega_{24}$};
\coordinate (m2) at (2.25,0.1);
\fill (m2) circle (0pt) node[below] {\tiny$\Omega_{22}$};
\coordinate (m12) at (2.25,-0.1);
\fill (m12) circle (0pt) node[above] {\tiny$\Omega_{21}$};
\coordinate (k) at (-4.7,0.1);
\fill (k) circle (0pt) node[below] {\tiny$\Omega_{43}$};
\coordinate (k1) at (-4.7,-0.1);
\fill (k1) circle (0pt) node[above] {\tiny$\Omega_{44}$};
\coordinate (l) at (-3,0.1);
\fill (l) circle (0pt) node[below] {\tiny$\Omega_{42}$};
\coordinate (l1) at (-3,-0.1);
\fill (l1) circle (0pt) node[above] {\tiny$\Omega_{41}$};
\coordinate (n) at (-0.27,0.1);
\fill (n) circle (0pt) node[below] {\tiny$\Omega_{33}$};
\coordinate (n1) at (-0.27,-0.1);
\fill (n1) circle (0pt) node[above] {\tiny$\Omega_{34}$};
\coordinate (m) at (-2.2,0.1);
\fill (m) circle (0pt) node[below] {\tiny$\Omega_{32}$};
\coordinate (m1) at (-2.2,-0.1);
\fill (m1) circle (0pt) node[above] {\tiny$\Omega_{31}$};
\coordinate (c) at (-2,0);
\fill[red] (c) circle (1pt) node[below] {\scriptsize$-1$};
\coordinate (d) at (2,0);
\fill[red] (d) circle (1pt) node[below] {\scriptsize$1$};
\end{tikzpicture}
\label{case1}}
\subfigure[]{
\begin{tikzpicture}
\draw[dashed](-7,0)--(7,0)node[right]{ Re$z$};
\coordinate (I) at (0,0);
\fill[red] (I) circle (1pt) node[below] {$0$};
\coordinate (c) at (-3,0);
\fill[red] (c) circle (1pt) node[below] {\scriptsize$-1$};
\coordinate (D) at (3,0);
\fill[red] (D) circle (1pt) node[below] {\scriptsize$1$};
\draw(-0.8,0)--(-0,0.7);
\draw[->](-0.8,0)--(-0.4,0.35)node[above]{\scriptsize$\Sigma_{51}$};
\draw(-0.8,0)--(-1.5,0.6);
\draw[-<](-0.8,0)--(-1.15,-0.3)node[below]{\scriptsize$\Sigma_{53}$};
\draw(-0.8,0)--(-0,-0.7);
\draw[-<](-0.8,0)--(-1.15,0.3)node[above]{\scriptsize$\Sigma_{54}$};
\draw(-0.8,0)--(-1.5,-0.6);
\draw[->](-0.8,0)--(-0.4,-0.35)node[below]{\scriptsize$\Sigma_{52}$};
\draw(-2.2,0)--(-1.5,0.6);
\draw[-<](-2.2,0)--(-2.65,0.35)node[above]{\scriptsize$\Sigma_{61}$};
\draw(-2.2,0)--(-1.5,-0.6);
\draw[->](-2.2,0)--(-1.85,-0.3)node[below]{\scriptsize$\Sigma_{63}$};
\draw(-2.2,0)--(-3.1,0.7);
\draw[->](-2.2,0)--(-1.85,0.3)node[above]{\scriptsize$\Sigma_{64}$};
\draw(-2.2,0)--(-3.1,-0.7);
\draw[-<](-2.2,0)--(-2.65,-0.35)node[below]{\scriptsize$\Sigma_{62}$};
\draw(-5.4,0)--(-6.5,0.9)node[above]{\scriptsize$\Sigma_{81}$};
\draw[-<](-5.4,0)--(-5.95,0.45);
\draw(-5.4,0)--(-4.7,0.6);
\draw[->](-5.4,0)--(-5.05,-0.3)node[below]{\scriptsize$\Sigma_{83}$};
\draw(-5.4,0)--(-6.5,-0.9)node[below]{\scriptsize$\Sigma_{82}$};
\draw[->](-5.4,0)--(-5.05,0.3)node[above]{\scriptsize$\Sigma_{84}$};
\draw(-5.4,0)--(-4.7,-0.6);
\draw[-<](-5.4,0)--(-5.95,-0.45);
\draw(-4,0)--(-3.1,0.7);
\draw[->](-4,0)--(-3.55,0.35)node[above]{\scriptsize$\Sigma_{71}$};
\draw(-4,0)--(-4.7,0.6);
\draw[-<](-4,0)--(-4.35,-0.3)node[below]{\scriptsize$\Sigma_{73}$};
\draw(-4,0)--(-3.1,-0.7);
\draw[-<](-4,0)--(-4.35,0.3)node[above]{\scriptsize$\Sigma_{74}$};
\draw(-4,0)--(-4.7,-0.6);
\draw[->](-4,0)--(-3.55,-0.35)node[below]{\scriptsize$\Sigma_{72}$};
\draw[->](-1.5,0)--(-1.5,0.6)node[above]{\scriptsize$\Sigma_{5,+}'$};
\draw[->](-1.5,0)--(-1.5,-0.6)node[below]{\scriptsize$\Sigma_{5,-}'$};
\draw[->](-4.7,0)--(-4.7,0.6)node[above]{\scriptsize$\Sigma_{7,+}'$};
\draw[->](-4.7,0)--(-4.7,-0.6)node[below]{\scriptsize$\Sigma_{7,-}'$};
\draw[->](-3.1,0)--(-3.1,0.7)node[above]{\scriptsize$\Sigma_{6,+}'$};
\draw[->](-3.1,0)--(-3.1,-0.7)node[below]{\scriptsize$\Sigma_{6,-}'$};
\draw(0.8,0)--(0,0.7);
\draw[-<](0.8,0)--(0.4,0.35)node[above]{\scriptsize$\Sigma_{41}$};
\draw(0.8,0)--(1.5,0.6);
\draw[->](0.8,0)--(1.15,-0.3)node[below]{\scriptsize$\Sigma_{43}$};
\draw(0.8,0)--(0,-0.7);
\draw[->](0.8,0)--(1.15,0.3)node[above]{\scriptsize$\Sigma_{44}$};
\draw(0.8,0)--(1.5,-0.6);
\draw[-<](0.8,0)--(0.4,-0.35)node[below]{\scriptsize$\Sigma_{42}$};
\draw(2.2,0)--(1.5,0.6);
\draw[->](2.2,0)--(2.65,0.35)node[above]{\scriptsize$\Sigma_{31}$};
\draw(2.2,0)--(1.5,-0.6);
\draw[-<](2.2,0)--(1.85,-0.3)node[below]{\scriptsize$\Sigma_{33}$};
\draw(2.2,0)--(3.1,0.7);
\draw[-<](2.2,0)--(1.85,0.3)node[above]{\scriptsize$\Sigma_{34}$};
\draw(2.2,0)--(3.1,-0.7);
\draw[->](2.2,0)--(2.65,-0.35)node[below]{\scriptsize$\Sigma_{32}$};
\draw(5.4,0)--(6.5,0.9)node[above]{\scriptsize$\Sigma_{11}$};
\draw[->](5.4,0)--(5.95,0.45);
\draw(5.4,0)--(4.7,0.6);
\draw[-<](5.4,0)--(5.05,-0.3)node[below]{\scriptsize$\Sigma_{13}$};
\draw(5.4,0)--(6.5,-0.9)node[below]{\scriptsize$\Sigma_{12}$};
\draw[-<](5.4,0)--(5.05,0.3)node[above]{\scriptsize$\Sigma_{14}$};
\draw(5.4,0)--(4.7,-0.6);
\draw[->](5.4,0)--(5.95,-0.45);
\draw(4,0)--(3.1,0.7);
\draw[-<](4,0)--(3.55,0.35)node[above]{\scriptsize$\Sigma_{21}$};
\draw(4,0)--(4.7,0.6);
\draw[->](4,0)--(4.35,-0.3)node[below]{\scriptsize$\Sigma_{23}$};
\draw(4,0)--(3.1,-0.7);
\draw[->](4,0)--(4.35,0.3)node[above]{\scriptsize$\Sigma_{24}$};
\draw(4,0)--(4.7,-0.6);
\draw[-<](4,0)--(3.55,-0.35)node[below]{\scriptsize$\Sigma_{22}$};
\draw[->](1.5,0)--(1.5,0.6)node[above]{\scriptsize$\Sigma_{3+}'$};
\draw[->](1.5,0)--(1.5,-0.6)node[below]{\scriptsize$\Sigma_{3-}'$};
\draw[->](4.7,0)--(4.7,0.6)node[above]{\scriptsize$\Sigma_{1+}'$};
\draw[->](4.7,0)--(4.7,-0.6)node[below]{\scriptsize$\Sigma_{1-}'$};
\draw[->](3.1,0)--(3.1,0.7)node[above]{\scriptsize$\Sigma_{2+}'$};
\draw[->](3.1,0)--(3.1,-0.7)node[below]{\scriptsize$\Sigma_{2-}'$};
\draw[->](0,0)--(0,0.7)node[above]{\scriptsize$\Sigma_{4+}'$};
\draw[->](0,0)--(0,-0.7)node[below]{\scriptsize$\Sigma_{4-}'$};
\coordinate (A) at (-5.4,0);
\fill[blue] (A) circle (1pt) node[below] {$\xi_8$};
\coordinate (b) at (-4,0);
\fill[blue] (b) circle (1pt) node[below] {$\xi_7$};
\coordinate (C) at (-0.8,0);
\fill[blue] (C) circle (1pt) node[below] {$\xi_5$};
\coordinate (d) at (-2.2,0);
\fill[blue] (d) circle (1pt) node[below] {$\xi_6$};
\coordinate (E) at (5.4,0);
\fill[blue] (E) circle (1pt) node[below] {$\xi_1$};
\coordinate (R) at (4,0);
\fill[blue] (R) circle (1pt) node[below] {$\xi_2$};
\coordinate (T) at (0.8,0);
\fill[blue] (T) circle (1pt) node[below] {$\xi_4$};
\coordinate (Y) at (2.2,0);
\fill[blue] (Y) circle (1pt) node[below] {$\xi_3$};
\coordinate (q) at (6.2,-0.1);
\fill (q) circle (0pt) node[above] {\tiny$\Omega_{11}$};
\coordinate (q1) at (6.2,0.05);
\fill (q1) circle (0pt) node[below] {\tiny$\Omega_{12}$};
\coordinate (w) at (4.95,-0.1);
\fill (w) circle (0pt) node[above] {\tiny$\Omega_{14}$};
\coordinate (w1) at (4.95,0.1);
\fill (w1) circle (0pt) node[below] {\tiny$\Omega_{13}$};
\coordinate (e) at (4.47,-0.1);
\fill (e) circle (0pt) node[above] {\tiny$\Omega_{24}$};
\coordinate (e1) at (4.47,0.1);
\fill (e1) circle (0pt) node[below] {\tiny$\Omega_{23}$};
\coordinate (r) at (3.4,-0.1);
\fill (r) circle (0pt) node[above] {\tiny$\Omega_{21}$};
\coordinate (r1) at (3.4,0.05);
\fill (r1) circle (0pt) node[below] {\tiny$\Omega_{22}$};
\coordinate (t) at (2.7,-0.1);
\fill (t) circle (0pt) node[above] {\tiny$\Omega_{31}$};
\coordinate (t1) at (2.7,0.05);
\fill (t1) circle (0pt) node[below] {\tiny$\Omega_{32}$};
\coordinate (y) at (1.75,-0.1);
\fill (y) circle (0pt) node[above] {\tiny$\Omega_{34}$};
\coordinate (y1) at (1.75,0.1);
\fill (y1) circle (0pt) node[below] {\tiny$\Omega_{33}$};
\coordinate (l) at (1.26,-0.1);
\fill (l) circle (0pt) node[above] {\tiny$\Omega_{44}$};
\coordinate (l1) at (1.26,0.1);
\fill (l1) circle (0pt) node[below] {\tiny$\Omega_{43}$};
\coordinate (k) at (0.3,-0.1);
\fill (k) circle (0pt) node[above] {\tiny$\Omega_{41}$};
\coordinate (k1) at (0.3,0.1);
\fill (k1) circle (0pt) node[below] {\tiny$\Omega_{42}$};
\coordinate (q8) at (-6.2,-0.1);
\fill (q8) circle (0pt) node[above] {\tiny$\Omega_{81}$};
\coordinate (q18) at (-6.2,0.05);
\fill (q18) circle (0pt) node[below] {\tiny$\Omega_{82}$};
\coordinate (w8) at (-4.95,-0.1);
\fill (w8) circle (0pt) node[above] {\tiny$\Omega_{84}$};
\coordinate (w18) at (-4.95,0.1);
\fill (w18) circle (0pt) node[below] {\tiny$\Omega_{83}$};
\coordinate (e7) at (-4.47,-0.1);
\fill (e7) circle (0pt) node[above] {\tiny$\Omega_{74}$};
\coordinate (e17) at (-4.47,0.1);
\fill (e17) circle (0pt) node[below] {\tiny$\Omega_{73}$};
\coordinate (7r) at (-3.4,-0.1);
\fill (7r) circle (0pt) node[above] {\tiny$\Omega_{71}$};
\coordinate (r17) at (-3.4,0.05);
\fill (r17) circle (0pt) node[below] {\tiny$\Omega_{72}$};
\coordinate (t6) at (-2.7,-0.1);
\fill (t6) circle (0pt) node[above] {\tiny$\Omega_{61}$};
\coordinate (t16) at (-2.7,0.05);
\fill (t16) circle (0pt) node[below] {\tiny$\Omega_{62}$};
\coordinate (y6) at (-1.75,-0.1);
\fill (y6) circle (0pt) node[above] {\tiny$\Omega_{64}$};
\coordinate (y16) at (-1.75,0.1);
\fill (y16) circle (0pt) node[below] {\tiny$\Omega_{63}$};
\coordinate (l5) at (-1.26,-0.1);
\fill (l5) circle (0pt) node[above] {\tiny$\Omega_{54}$};
\coordinate (l15) at (-1.26,0.1);
\fill (l15) circle (0pt) node[below] {\tiny$\Omega_{53}$};
\coordinate (k5) at (-0.3,-0.1);
\fill (k5) circle (0pt) node[above] {\tiny$\Omega_{51}$};
\coordinate (k15) at (-0.3,0.1);
\fill (k15) circle (0pt) node[below] {\tiny$\Omega_{52}$};
\end{tikzpicture}
\label{case2}}
\caption{Figure (a) and (b) are corresponding to the  $0<\xi<2$ and  $-\frac{1}{4}<\xi<0$ respectively. The regions
$\Omega_{ij}$ are of the boundaries $\Sigma_{ij}$.}
	\label{FigOmig}
\end{figure}
Additionally, taking $\Sigma_{8}^\pm=\emptyset $ for $\xi(-1/4,0)$ and $\Sigma_{1}^\pm=\emptyset$ for $\xi(0,2)$.
Fix a sufficiently small constant angle $\phi$ such that the following conditions hold
\begin{enumerate}[(1)]
\item Each $\Omega_j$ are disjoint with $\D{P}_N$, $\overline{\D{P}}_N$ and $\{z\in\D{C}|\Im\theta(z)=0\}$.
\item $2\tan\phi>\xi_{n(\xi)/2}-\xi_{n(\xi)/2+1}$.
\end{enumerate}

Since for each regions $\Omega_j$, the oscillation terms $e^{\pm2it\theta}$ have different attenuation properties, which will affect the contour deformation, so we first give the following proposition.
\begin{prop}
For $\xi\in(-1/4,0)$ and $\xi\in(0,2)$, there exists a constant $\tau(\xi)$ such that
\begin{align}
&\Im\theta\leq-\tau(\xi)\Im z\cdot\frac{|z|^2-\xi_j^2}{4+|z|^2}, \ z\in\Omega_{j1},\Omega_{j3},\label{est-theta1}\\
&\Im\theta\geq\tau(\xi)\Im z\cdot\frac{|z|^2-\xi_j^2}{4+|z|^2},\ z\in\Omega_{j2},\Omega_{j4}.\label{est-theta2}
\end{align}
\end{prop}
\begin{proof}
The proof of this proposition is similar to \cite{YF-AM}.
\end{proof}

In what follows, we perform contour deformation for RH problem \ref{RHP-M1}, following the standard procedure in  \cite{DZ-AM1993} and \cite{AIHP} in the presence of discrete spectrum.
Now we introduce a new matrix-valued function $M^{(2)}(z)=M^{(2)}(y,t,z)$ defined by
\begin{align}\label{m2}
M^{(2)}(z)=M^{(1)}(z)\mathcal {R}^{(2)}(z),
\end{align}
where the choice of  $\mathcal {R}^{(2)}(z)$ needs to meet the following conditions
\begin{enumerate}[(1)]
\item To remove jump condition for $M^{(1)}(z)$ on real axis $\D{R}$.
\item Does not affect the residue condition at the poles $\eta_n$ and $\overline{\eta}_n$ for $n\in\Lambda$.
\item The new jump matrix for $M^{(2)}(z)$ along path $\Sigma^{(2)}$ is exponentially decaying.
\item $\mathcal {R}^{(2)}_-(z)^{-1}V^{(1)}(z)\mathcal {R}^{(2)}_+(z)\equiv I$.
\end{enumerate}
Simple calculations show that
\begin{align*}
M^{(2)}_+(z)&=M^{(1)}_+(z)\mathcal {R}^{(2)}_+(z)=M^{(1)}_-(z)V^{(1)}(z)
\\&=M^{(2)}_-(z)
\mathcal {R}^{(2)}_-(z)^{-1}V^{(1)}(z)\mathcal {R}^{(2)}_+(z),
\end{align*}
where $\mathcal {R}^{(2)}_\pm(z)$ are the boundaries of $\mathcal {R}^{(2)}(z)$ as $\Im(z)\to0$. To this end, for $k=1,2,\ldots,n(\xi)$, we define
\begin{align}\label{r2}
\mathcal {R}^{(2)}(z)=\left\{\begin{aligned}
&\left(
  \begin{array}{cc}
    1 & R_{kj}(z)e^{-2it\theta} \\
    0 & 1 \\
  \end{array}
\right),&z\in\Omega_{kj},\quad j=1,3,\\
&\left(
  \begin{array}{cc}
    1 & 0 \\
    R_{kj}(z)e^{2it\theta} & 1 \\
  \end{array}
\right),&z\in\Omega_{kj},\quad j=2,4,\\
&I, &z\in\text{elsewhere},
\end{aligned}\right.
\end{align}
where the functions $R_{kj}(z)$ are defined in Proposition \ref{pp-r2}.
\begin{prop}\label{pp-r2}
For $\xi\in(-1/4,0)$ and $\xi\in(0,2)$, the function $R_{kj}(z):\overline{\Omega}_{kj}\to\D{C}$, $j=1,2,3,4$ and $k=1,2,\ldots,n(\xi)$ are of the boundary values
\begin{subequations}\label{Rkj}
\begin{align}
&R_{k1}(z,\xi)=\Bigg\{\begin{array}{ll}
p_{k1}(z,\xi)T_+(z)^{2}, & z\in I_{k1},\\
p_{k1}(\xi_k,\xi)T_k(\xi)^{2}(z-\xi_k)^{2i\nu(\xi_k)},  &z\in \Sigma_{k1},\\
\end{array} \\
&R_{k2}(z,\xi)=\Bigg\{\begin{array}{ll}
p_{k2}(z,\xi)T_-(z)^{-2}, &z\in  I_{k2},\\
p_{k2}(\xi_k,\xi)T_k(\xi)^{-2}(z-\xi_k)^{-2i\nu(\xi_k)},  &z\in \Sigma_{k2},\\
\end{array} \\
&R_{k3}(z,\xi)=\Bigg\{\begin{array}{ll}
p_{k3}(z,\xi)T(z)^{2}, &z\in I_{k3}, \\
p_{k3}(\xi_k,\xi)T_k(\xi)^{2}(z-\xi_k)^{2i\nu(\xi_k)}, &z\in \Sigma_{k3},\\
\end{array} \\
&R_{k4}(z,\xi)=\Bigg\{\begin{array}{ll}
p_{k4}(z,\xi)T(z)^{-2}, &z\in I_{k4},\\
p_{k4}(\xi_k,\xi)T_k(\xi)^{-2}(z-\xi_k)^{-2i\nu(\xi_k)}, &z\in \Sigma_{k4},\\
\end{array}
\end{align}	
\end{subequations}
where the entries $p_{kj}$ appearing in \eqref{Rkj} are defined by
\begin{align}
&p_{k1}(z,\xi)=-\frac{r(z)}{1-|r(z)|^2}, \qquad p_{k3}(z,\xi)=r(z),\\
&p_{k2}(z,\xi)=-\frac{\overline{r}(z)}{1-|r(z)|^2}, \qquad p_{k4}(z,\xi)=\overline{r}(z).
\end{align}
The function $R_{kj}(z)$ has the following estimates
\begin{align}
&\left|R_{kj}(z)\right|\lesssim \sin^2(z_0\arg(z-\xi_k))+(1+\Re z)^{-1/2},&z\in\Omega_{kj},\\
&\left|\overline{\partial}R_{kj}(z)\right|\lesssim \left|p'_{kj}(\Re z)\right|+|z-\xi_k|^{-1/2},&z\in\Omega_{kj},\\
&\overline{\partial}R_{kj}(z)\equiv0,&z\in\text{elsewhere}.
\end{align}
\end{prop}

Based on the above analysis, our new matrix $M^{(2)}(z)$ satisfies the following mixed RH problem.
\begin{RHP}\label{RHP-m2}
Find a matrix-valued function $M^{(2)}(z)$ defined by \eqref{m2} such that
\begin{enumerate}[(i)]
\item The function $M^{(2)}(z)$ is continuous in $\D{C}$ with continuous first partial derivatives in $\D{C}\setminus(\Sigma^{(2)}\cup\{\eta_n,\overline{\eta}_n\}_{n\in\Lambda})$.
\item The jump relation
\begin{align}\label{Jump-m2}
M^{(2)}_+(z)=M^{(2)}_-(z)V^{(2)}(z),\ z\in\Sigma^{(2)},
\end{align}
where the jump matrix is defined by
\begin{align} \label{jumpv21}
V^{(2)}(z)=\left\{
\begin{array}{ll}
R^{(2)}(z)^{-1}|_{\Omega_{k1}\cup\Omega_{k4}},& z\in\Sigma_{k1}\cup\Sigma_{k4},\\[12pt]
R^{(2)}(z)|_{\Omega_{k2}\cup\Omega_{k3}},& z\in\Sigma_{k2}\cup\Sigma_{k3},\\[12pt]
R^{(2)}(z)^{-1}|_{\Omega_{(k-1)\ (3\mp1)/2}}R^{(2)}(z)|_{\Omega_{k\ (3\mp1)/2}},& z\in\Sigma_{k}^{\pm},\ k \text{ is even},\\[12pt]
R^{(2)}(z)^{-1}|_{\Omega_{k\ (7\pm1)/2}}R^{(2)}(z)|_{\Omega_{(k+1)\ (7\pm1)/2}},& z\in\Sigma_{k}^{\pm},\ k \text{ is odd},\\[12pt]
\left(\begin{array}{cc}
1 & 0\\
-C_n(z-\eta_n)^{-1}T(z)^{-2}e^{2it\theta_n} & 1
\end{array}\right),   & z\in\partial\mathbb{P}_n,\ n\in\nabla,\\[12pt]
\left(\begin{array}{cc}
1 & -C_n^{-1}(z-\eta_n)T^{2}(z)e^{-2it\theta_n}\\
0 & 1
\end{array}\right),   & z\in\partial\mathbb{P}_n,\ n\in\Delta,\\
\left(\begin{array}{cc}
1 & \overline{C}_n(z-\overline{\eta}_n)^{-1}T^{2}(z)e^{-2it\overline{\theta}_n}\\
0 & 1
\end{array}\right),   & z\in\partial\overline{\mathbb{P}}_n,\ n\in\nabla,\\
\left(\begin{array}{cc}
1 & 0	\\
\overline{C}_n^{-1}(z-\overline{\eta}_n)e^{2it\overline{\theta}_n}T(z)^{-2} & 1
\end{array}\right),   & z\in\partial\overline{\mathbb{P}}_n,\ n\in\Delta.\\
\end{array}\right.
\end{align}
\item Asymptotic behavior:
\begin{align}\label{Asy-m2}
&M^{(2)}(z)=I+O(z^{-1}), z\to\infty~ (\text{and}~ z\to0 ~\text{by symmetry}),\\
&M^{(2)}(z)=\left(\di(a_1(y,t),a_1^{-1}(y,t))+\left(
        \begin{array}{cc}
          0 & a_2(y,t) \\
          a_3(y,t) & 0 \\
        \end{array}
      \right)(z-i)\right)
T(i)^{\sigma_3}\left(I-T_0(\xi)(z-i)\right)
+O((z-i)^2),
\end{align}
where $a_j(y,t)$, $j=1,2,3$ are real-valued functions from the expansion of $M^{(2)}(z)$ at $z=i$.
\item Residue conditions: The matrix $M^{(2)}(z)$ has simple poles at $\eta_n$ and $\overline{\eta}_n$ with
\begin{subequations}\label{res-m2}
\begin{align}
&\mathop{\res}\limits_{\eta_n} M^{(2)}(z)=\lim_{z\to\eta_n} M^{(2)} (z)\left(
                                                  \begin{array}{cc}
                                                    0 & 0 \\
                                                    c_nT^2(z)e^{2it\theta(\eta_n)} & 0 \\
                                                  \end{array}
                                                \right), \label{remm21} \\
&\mathop{\res}\limits_{\overline{\eta}_n} M^{(2)}(z)=\lim_{z\to\overline{\eta}_n} M^{(2)} (z)\left(
                                                  \begin{array}{cc}
                                                    0 & \overline{c}_nT^{-2}(z)e^{-2it\overline{\theta}(\eta_n)} \\
                                                    0 & 0 \\
                                                  \end{array}\right).
\end{align}
\end{subequations}
\item The $\overline{\partial}$-derivative: for $z\in\D{C}$, one has
\begin{align}
\overline{\partial}M^{(2)}(z)=M^{(2)}(z)W(z),
\end{align}
where
\begin{align}
W(z)=\left\{\begin{aligned}
&\left(
   \begin{array}{cc}
     0 & \overline{\partial}R_{kj}(z,\xi)e^{-2it\theta} \\
     0 & 0 \\
   \end{array}
 \right),&&z\in\Omega_{kj},\ j=1,3,\\
&\left(
   \begin{array}{cc}
     0 & 0 \\
     \overline{\partial}R_{kj}(z,\xi)e^{2it\theta} & 0 \\
   \end{array}
 \right),&&z\in\Omega_{kj},\ j=2,4,\\
&\textbf{0},&&z\in\text{elsewhere.}
\end{aligned}\right.
\end{align}
\end{enumerate}
\end{RHP}

\subsection{Decomposition of the mixed RH problem}
To solve the mixed RH problem \ref{RHP-m2} for $M^{(2)}(z)$, the mixed RH problem is decomposed into pure RH problem for $M^{(2)}_{rhp}(z)$ corresponding to $W(z)\equiv0$ and pure $\overline{\partial}$-problem for $M^{(3)}(z)$ corresponding to $W(z)\neq0$.
\begin{RHP}\label{RHP-mrhp}
Find a matrix function $M^{(2)}_{rhp}(z)$ with the same asymptotic and symmetric properties as $M^{(2)}(z)$ to satisfy
\begin{enumerate}
\item $M^{(2)}_{rhp}(z)$ is analytic in $\D{C}\setminus\Sigma^{(2)}$.
\item Jump condition:
\begin{align}
M^{(2)}_{rhp,+}(z)=M^{(2)}_{rhp,-}(z)V^{(2)}(z),
\end{align}
where $V^{(2)}(z)$ is defined by \eqref{jumpv21}.
\item The $\overline{\partial}$-derivative: for $z\in\D{C}$, $W(z)=0$.
\item Residue conditions: The matrix $M^{(2)}_{rhp}(z)$ has simple poles at $\eta_n$ and $\overline{\eta}_n$ with
\begin{subequations}\label{res-mrhp}
\begin{align}
&\mathop{\res}\limits_{\eta_n} M^{(2)}_{rhp}(z)=\lim_{z\to\eta_n}M^{(2)}_{rhp}(z)\left(
                                                  \begin{array}{cc}
                                                    0 & 0 \\
                                                    c_nT^2(z)e^{2it\theta(\eta_n)} & 0 \\
                                                  \end{array}
                                                \right), \label{remmrhp} \\
&\mathop{\res}\limits_{\overline{\eta}_n} M^{(2)}_{rhp}(z)=\lim_{z\to\overline{\eta}_n}M^{(2)}_{rhp}(z)\left(
                                                  \begin{array}{cc}
                                                    0 & \overline{c}_nT^{-2}(z)e^{-2it\overline{\theta}(\eta_n)} \\
                                                    0 & 0 \\
                                                  \end{array}\right).
\end{align}
\end{subequations}
\end{enumerate}
\end{RHP}

Due to the existence of steady-state phase points $\xi_j$ ($j=1,2,\ldots,n(\xi)$), and the jump matrix $V^{(2)}$ does not uniformly converge to the identity matrix in the small neighborhood of phase points $\xi_j$, a small neighborhood $U(\xi)$ is introduced in each phase point
\begin{align}
U(\xi)=\bigcup_{j=1,2,\ldots,n(\xi)}U_{\xi_j},\
U_{\xi_j}=\{z:|z-\xi_j|\leq\min\{\rho,\frac{1}{3}\min_{m\neq n}|\eta_m-\eta_n|\}\}.
\end{align}
In what follows, we investigate the estimate for $V^{(2)}$.
\begin{prop}\label{pp-est-V2-phase}
For $1\leq p\leq+\infty$, the jump matrix $V^{(2)}$ such that
\begin{align}
&\left|\left|V^{(2)}-I\right|\right|_{L^{p}(\Sigma_{jk}\setminus U_{\xi_j})}=O(e^{-ct}),& t\to\infty,\\
&\left|\left|V^{(2)}-I\right|\right|_{L^{p}(\Sigma_{j}^{\pm})}=O(e^{-c't}),&t\to\infty,\label{est-V2-phase}
\end{align}
where $c$ ($c'$) are positive constants, $j=1,2,\ldots,n(\xi)$ and $k=1,2,3,4$.
\end{prop}
\begin{proof}
Taking $\xi\in(-1/4,0)$ as example, other situations are similar and can be proved. Let $z-\xi=\ell e^{i\psi}$ with $\ell\in(\rho,\infty)$ for $\Sigma_{11}\setminus U_{\xi_1}$
\begin{align}
\left|\left|V^{(2)}-I\right|\right|^p_{L^{p}(\Sigma_{11}\setminus U_{\xi_1})}&=
\left|\left|p_{11}(\xi_1)T_1(\xi)^2(z-\xi_1)^{2iv(\xi_1)}e^{-2it\theta}\right|\right|^p_{L^{p}(\Sigma_{11}\setminus U_{\xi_1})}\nonumber \\
&\lesssim \int_{\Sigma_{11}\setminus U_{\xi_1}}e^{-p\tau(\xi)t\Im z\frac{|z|^2-\xi_i^2}{4+|z|^2}}\ dz \nonumber \\
&\lesssim\int_{\rho}^{+\infty}e^{-p\tau(\xi)t\ell}\ d\ell\lesssim t^{-1}\exp(-p\tau(\xi)t\rho).
\end{align}
Using the same idea, inequality \eqref{est-V2-phase} can be proved, which completes the proof of the proposition.
\end{proof}

Proposition \ref{pp-est-V2-phase} shows that the jump matrix $V^{(2)}(z)$ uniformly converges to the identity matrix on $\Sigma^{\sharp}\setminus U(\xi)$, which encourages us to ignore the jump of $M^{(2)}_{rhp}(z)$ out $ U(\xi)$, and then we can do the following decomposition
\begin{align}\label{dec-M2}
M^{(2)}_{rhp}(z)=\left\{\begin{aligned}
&E(z,\xi)M^{(sol)}(z),&z\notin U(\xi),\\
&E(z,\xi)M^{(sol)}(z)M^{(mod)}(z),&z\in U(\xi).
\end{aligned}\right.
\end{align}
It follows from \eqref{dec-M2} that $M^{(2)}_{rhp}(z)$ is decomposed into two parts, one is $M^{(sol)}(z)$ corresponding to the pure RH problem and ignoring the jump condition out $ U(\xi)$, the other is the local model $M^{(mod)}(z)$ on the small neighborhood of the steady-state phase points $\xi_j$, which can be matched by using the parabolic cylinder model, where $E(z,\xi)$ is the error function, which is analyzed by the small-norm RH problem \cite{KMM-2003,small-1}.

\subsection{Analysis on pure RH problem}
At the beginning of this section, we will analyze the construction of the solution of the pure RH problem on region $\xi\in(-1/4,0)\cup(0,2)$. From \eqref{dec-M2}, it is clear that this includes two parts, one is the soliton solution provided by the discrete spectrum outside the small neighborhood of the phase point, and the other is the solution provided by the phase points $\xi_j$ inside the small neighborhood. To this end, first define some new contours
$$\Sigma^{(\natural)}= (\underset{j=1,..,n(\xi)}{\underset{k=1,...,4,}{\cup}}\Sigma_{jk} )\cap U(\xi),$$
which is shown in Fig. \ref{sigmaj}.
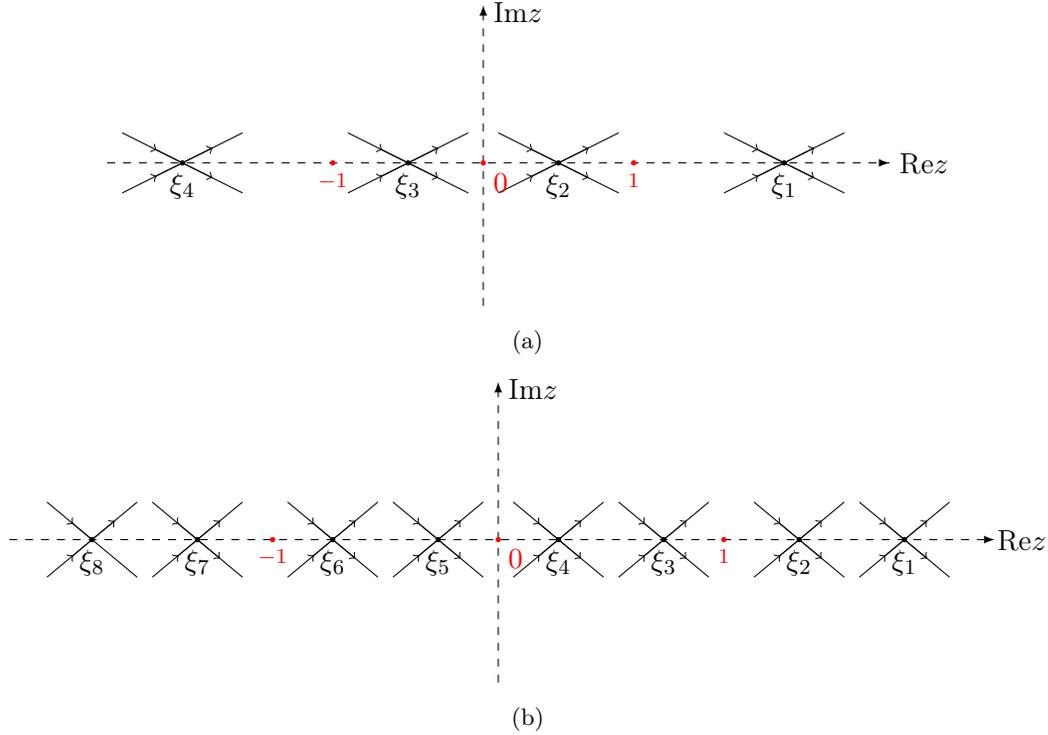
\begin{figure}[h]
\subfigure[]{
\begin{tikzpicture}
\draw(-4,0)--(-4.8,0.4);
\draw[dashed](0,2)node[right]{ Im$z$}--(0,-2);
\draw[-latex](0,2)--(0,2.1);
\draw[-latex](5.3,0)--(5.4,0);
\draw[-<](-4,0)--(-4.4,0.2);
\draw(-4,0)--(-3.2,0.4);
\draw[->](-4,0)--(-3.6,-0.2);
\draw(-4,0)--(-4.8,-0.4);
\draw[->](-4,0)--(-3.6,0.2);
\draw(-4,0)--(-3.2,-0.4);
\draw[-<](-4,0)--(-4.4,-0.2);
\draw(-1,0)--(-0.2,0.4);
\draw[->](-1,0)--(-0.6,0.2);
\draw(-1,0)--(-1.8,0.4);
\draw[-<](-1,0)--(-1.4,-0.2);
\draw(-1,0)--(-0.2,-0.4);
\draw[-<](-1,0)--(-1.4,0.2);
\draw(-1,0)--(-1.8,-0.4);
\draw[->](-1,0)--(-0.6,-0.2);
\draw[dashed](-5,0)--(5.4,0)node[right]{ Re$z$};
\draw(1,0)--(0.2,0.4);
\draw[-<](1,0)--(0.6,0.2);
\draw(1,0)--(0.2,-0.4);
\draw[->](1,0)--(1.4,-0.2);
\draw(1,0)--(1.8,0.4);
\draw[->](1,0)--(1.4,0.2);
\draw(1,0)--(1.8,-0.4);
\draw[-<](1,0)--(0.6,-0.2);
\draw(4,0)--(4.8,0.4);
\draw[->](4,0)--(4.4,0.2);
\draw(4,0)--(3.2,0.4);
\draw[-<](4,0)--(3.6,-0.2);
\draw(4,0)--(4.8,-0.4);
\draw[-<](4,0)--(3.6,0.2);
\draw(4,0)--(3.2,-0.4);
\draw[->](4,0)--(4.4,-0.2);
\coordinate (I) at (0,0);
\fill[red] (I) circle (1pt) node[below right] {$0$};
\coordinate (A) at (-4,0);
\fill (A) circle (1pt) node[below] {$\xi_4$};
\coordinate (b) at (-1,0);
\fill (b) circle (1pt) node[below] {$\xi_3$};
\coordinate (e) at (4,0);
\fill (e) circle (1pt) node[below] {$\xi_1$};
\coordinate (f) at (1,0);
\fill (f) circle (1pt) node[below] {$\xi_2$};
\coordinate (c) at (-2,0);
\fill[red] (c) circle (1pt) node[below] {\scriptsize$-1$};
\coordinate (d) at (2,0);
\fill[red] (d) circle (1pt) node[below] {\scriptsize$1$};
\end{tikzpicture}
\label{si1}}
\subfigure[]{
\begin{tikzpicture}
\draw[dashed](-6.5,0)--(6.5,0)node[right]{ Re$z$};
\draw[dashed](0,2)node[right]{ Im$z$}--(0,-2);
\draw[-latex](0,2)--(0,2.1);
\draw[-latex](6.5,0)--(6.6,0);
\coordinate (I) at (0,0);
\fill[red] (I) circle (1pt) node[below right] {$0$};
\coordinate (c) at (-3,0);
\fill[red] (c) circle (1pt) node[below] {\scriptsize$-1$};
\coordinate (D) at (3,0);
\fill[red] (D) circle (1pt) node[below] {\scriptsize$1$};
\draw(-0.8,0)--(-0.2,0.5);
\draw[->](-0.8,0)--(-0.5,0.25);
\draw(-0.8,0)--(-1.4,0.5);
\draw[-<](-0.8,0)--(-1.1,-0.25);
\draw(-0.8,0)--(-0.2,-0.5);
\draw[-<](-0.8,0)--(-1.1,0.25);
\draw(-0.8,0)--(-1.4,-0.5);
\draw[->](-0.8,0)--(-0.5,-0.25);
\draw(-2.2,0)--(-1.6,0.5);
\draw[-<](-2.2,0)--(-2.5,0.25);
\draw(-2.2,0)--(-1.6,-0.5);
\draw[->](-2.2,0)--(-1.9,-0.25);
\draw(-2.2,0)--(-2.8,0.5);
\draw[->](-2.2,0)--(-1.9,0.25);
\draw(-2.2,0)--(-2.8,-0.5);
\draw[-<](-2.2,0)--(-2.5,-0.25);
\draw(-5.4,0)--(-6,0.5);
\draw[-<](-5.4,0)--(-5.7,0.25);
\draw(-5.4,0)--(-4.8,0.5);
\draw(-5.4,0)--(-6,-0.5);
\draw[->](-5.4,0)--(-5.1,0.25);
\draw(-5.4,0)--(-4.8,-0.5);
\draw[-<](-5.4,0)--(-5.7,-0.25);
\draw(-4,0)--(-3.4,0.5);
\draw[->](-4,0)--(-3.7,0.25);
\draw(-4,0)--(-4.6,0.5);
\draw[-<](-4,0)--(-4.3,-0.25);
\draw(-4,0)--(-3.4,-0.5);
\draw[-<](-4,0)--(-4.3,0.25);
\draw(-4,0)--(-4.6,-0.5);
\draw[->](-4,0)--(-3.7,-0.25);
\draw(0.8,0)--(0.2,0.5);
\draw[-<](0.8,0)--(0.5,0.25);
\draw(0.8,0)--(1.4,0.5);
\draw[->](0.8,0)--(1.1,-0.25);
\draw(0.8,0)--(0.2,-0.5);
\draw[->](0.8,0)--(1.1,0.25);
\draw(0.8,0)--(1.4,-0.5);
\draw[-<](0.8,0)--(0.5,-0.25);
\draw(2.2,0)--(1.6,0.5);
\draw[->](2.2,0)--(2.5,0.25);
\draw(2.2,0)--(1.6,-0.5);
\draw[-<](2.2,0)--(1.9,-0.25);
\draw(2.2,0)--(2.8,0.5);
\draw[-<](2.2,0)--(1.9,0.25);
\draw(2.2,0)--(2.8,-0.5);
\draw[->](2.2,0)--(2.5,-0.25);
\draw(5.4,0)--(6,0.5);
\draw[->](5.4,0)--(5.7,0.25);
\draw(5.4,0)--(4.8,0.5);
\draw[-<](5.4,0)--(5.1,-0.25);
\draw(5.4,0)--(6,-0.5);
\draw[-<](5.4,0)--(5.1,0.25);
\draw(5.4,0)--(4.8,-0.5);
\draw[->](5.4,0)--(5.7,-0.25);
\draw(4,0)--(3.4,0.5);
\draw[-<](4,0)--(3.7,0.25);
\draw(4,0)--(4.6,0.5);
\draw[->](4,0)--(4.3,-0.25);
\draw(4,0)--(3.4,-0.5);
\draw[->](4,0)--(4.3,0.25);
\draw(4,0)--(4.6,-0.5);
\draw[-<](4,0)--(3.7,-0.25);
\coordinate (A) at (-5.4,0);
\fill (A) circle (1pt) node[below] {$\xi_8$};
\coordinate (b) at (-4,0);
\fill (b) circle (1pt) node[below] {$\xi_7$};
\coordinate (C) at (-0.8,0);
\fill (C) circle (1pt) node[below] {$\xi_5$};
\coordinate (d) at (-2.2,0);
\fill (d) circle (1pt) node[below] {$\xi_6$};
\coordinate (E) at (5.4,0);
\fill (E) circle (1pt) node[below] {$\xi_1$};
\coordinate (R) at (4,0);
\fill (R) circle (1pt) node[below] {$\xi_2$};
\coordinate (T) at (0.8,0);
\fill (T) circle (1pt) node[below] {$\xi_4$};
\coordinate (Y) at (2.2,0);
\fill (Y) circle (1pt) node[below] {$\xi_3$};
\end{tikzpicture}
\label{si2}}
\caption{Figures (a) and (b) denote the contour $\Sigma^{(\natural)}$ corresponding to the  $0<\xi<2$ and  $-\frac{1}{4}<\xi<0$, respectively.}
	\label{sigmaj}
\end{figure}
\subsubsection{The outer model: an $N$-soliton potential}
In this section, similar to region $\xi\in(-\infty,-1/4)\cup(2,+\infty)$, the external soliton solution is constructed from the original RH problem \ref{RHP-M}. The difference is that there are phase points $\xi_j$ at this time, which means that the combination of $U(\xi)$ is not empty.
\begin{prop}
The matrix $M^{(sol)}(z)$ solves RH problem \ref{RHP-Mrhp} for $M^{(rhp)}(z)$ and satisfies Proposition \ref{pp-Mrhp}.
\end{prop}

\subsubsection{Local model near the saddle points $\xi_j$}
This section will consider the local RH problem in the small neighborhood of the steady-state phase points $\xi_j$, based on the parabolic cylindrical function.
\begin{RHP}\label{RHP-Mmod}
Find a matrix-valued function $M^{(mod)}(z)$ such that
\begin{enumerate}
\item $M^{(mod)}(z)$ is analytic in $\D{C}\setminus\Sigma^{(\natural)}$.
\item For $z\in\Sigma^{(\natural)}$, $M^{(mod)}(z)$ has continuous boundary values $M^{(mod)}_{\pm}(z)$  with jump condition
\begin{align}\label{jump-mmod}
M^{(mod)}_{+}(z)=M^{(mod)}_{-}(z)V^{(2)}(z).
\end{align}
\item Asymptotic behavior:
\begin{align*}
M^{(mod)}(z)=I+\mathcal {O}(z^{-1}), ~~~~~z\rightarrow\infty.
\end{align*}
\end{enumerate}
\end{RHP}

The solution of this RH problem \ref{RHP-Mmod} depends on the Beal-Coifman operator theory \cite{BC-1984}. Now we construct the relationship between $M^{(mod)}(z)$ and $M^{(mod),\xi_k}(z)$, where $k=1,2,\ldots,n(\xi)$. Observing that RH problem \ref{RHP-Mmod} only has jump condition and no poles. The jump condition $V^{(2)}(z)$ adopts trivial decomposition based on
Beal-Coifman operator theory
\begin{align*}
V^{(2)}(z)=(I-\omega_{kj}^-)^{-1}(I+\omega_{kj}^+),
\end{align*}
where $\omega_{kj}^-=I-V^{(2)}(z)^{-1}=V^{(2)}(z)-I$ and $\omega_{kj}^+=0$ with
\begin{align*}
\omega_{kj}^-(z)=\left\{\begin{array}{lll}
\left(\begin{array}{cc}
0 & R_{kj}(z,\xi)e^{-2it\theta}\\
0 & 0
\end{array}\right), &z\in \Sigma_{kj},\ j=1,3,\\[10pt]
\left(\begin{array}{cc}
0 & 0\\
R_{kj}(z,\xi)e^{2it\theta} & 0
\end{array}\right),  &z\in \Sigma_{kj},\ j=2,4.
\end{array}\right.
\end{align*}

Applying the Cauchy projection operators $C_\pm$ on $\Sigma^{(\natural)}_{kj}$ yields for any $f\in L^{2}(\Sigma_{jk}^{(\natural)})$
\begin{align}\label{oper-cau}
C_{\pm}f(z)=\lim_{s\rightarrow z\in\Sigma_{jk}^{(\natural)}}\frac{1}{2\pi i}\int_{\Sigma_{jk}^{(\natural)}}
\frac{f(s)}{s-z}~ds,
\end{align}
where   $k=1,2,\ldots,n(\xi)$. Furthermore, we define the Beals-Coifman operator as
\begin{align}\label{oper-cau-}
C_{w_{jk}}(f)=C_{-}(fw_{jk}^+)+C_+(fw_{jk}^-).
\end{align}

we define
\begin{align*}
&\Sigma_{k}^{(\natural)}=\bigcup_{j=1,2,3,4}\Sigma_{kj}^{(\natural)},\ w_{k}=\sum_{j=1}^4w_{kj},\\
&\omega=\sum_{k=1}^{n(\xi)}w_{k},\ \Sigma^{(\natural)}=\bigcup_{k=1,2,\ldots,n(\xi)}\Sigma_{k}^{(\natural)},
\end{align*}
and then we have $C_{\omega}=\sum_{k=1}^{n(\xi)}C_{\omega_k}.$  The existence and uniqueness of the solution of RH problem \ref{RHP-Mmod} can be guaranteed by the following lemma corresponding to Proposition 2.11 in Deift and Zhou \cite{DZ-CPAM}.
\begin{lem}\label{lemm1}
Let $f\in I+L^{2}(\Sigma)$ be the solution of the  singular integral equation
\begin{align}
f=I+C_{w}(f).
\end{align}
Then,  the unique solution to RH problem \ref{RHP-Mmod} for $M^{(mod)}(z)$ has the form
\begin{align}
M^{(mod)}(z)=I+\frac{1}{2\pi i}\int_{ \Sigma^{(\natural)}}\frac{f(s)w(s)}{s-z}ds.
\end{align}
\end{lem}

Simple calculations show that $||\omega_{kj}^{-1}||_{L^2(\Sigma_{kj}^{(\natural)})}=O(t^{-1/2})$, which implies $1-C_{\omega}$, and $1-C_{\omega_k}$, and $1-C_{\omega_{kj}}$ exist and are reversible as $t\to+\infty$. Thus, combining
with Lemma \ref{lemm1}, the solution to RH problem \ref{RHP-Mmod} is expressed by
\begin{align}\label{sol-Mmod}
M^{(mod)}(z)=I+\frac{1}{2\pi i}\int_{\Sigma^{(\natural)}}\frac{(I-C_w)^{-1}Iw(s)}{s-z}ds.
\end{align}
In what follows, following the strategy of Deift-Zhou \cite{DZ-AM1993}, the contribution to the solution is calculated at each phase point. First, the following equation is introduced in Proposition \ref{PP-asy-phase}.
\begin{prop}\label{PP-asy-phase}
As $t\to+\infty$, the integral equation can be decomposed into the sum of the integral equation at phase points $\xi_k$
\begin{align}
\frac{1}{2\pi i}\int_{\Sigma^{(\natural)}}\frac{(I-C_w)^{-1}Iw(s)}{s-z}ds
=\frac{1}{2\pi i}\sum_{k=1}^{n(\xi)}\int_{\Sigma^{(\natural)}_k}\frac{(I-C_{w_k})^{-1}Iw_k(s)}
{s-z}ds+\mathcal {O}(t^{-1}).
\end{align}
\end{prop}
\begin{proof}
Following the decomposition \cite{Varzugin-JMP}
\begin{align*}
&1-\sum_{j\neq k}^{n(\xi)}C_{\omega_j}C_{\omega_k}\left(1-C_{\omega_k}\right)^{-1}=\left(1-C_{\omega}\right)\left(1+\sum_{j=1}^{n(\xi)}
C_{\omega_j}\left(1-C_{\omega_j}\right)^{-1}\right),\\
&1-\sum_{j\neq k}^{n(\xi)}\left(1-C_{\omega_k}\right)^{-1}C_{\omega_j}C_{\omega_k}=\left(1+\sum_{j=1}^{n(\xi)}C_{\omega_j}
\left(1-C_{w_j}\right)^{-1}\right)\left(1-C_{\omega}\right),
\end{align*}
we  observe that it is necessary to estimate the norm of $C_{\omega_j}C_{\omega_k}$ from $L^{\infty}$ to $L^2$ and
$L^{2}$ to $L^2$. Combining the result $||\omega_{kj}^{-1}||_{L^2(\Sigma_{kj}^{(\natural)})}=O(t^{-1/2})$ with Lemma 3.5 in Deift-Zhou \cite{DZ-AM1993} yields
\begin{align}
||C_{w_j}C_{w_k}||_{L^{2}(\Sigma^{(\natural)})}=\mathcal {O}(t^{-1/2}),\ \
||C_{w_j}C_{w_k}||_{L^{\infty}(\Sigma^{(\natural)})\rightarrow L^{2}(\Sigma^{(\natural)})}=\mathcal {O}(t^{-3/4}).
\end{align}
Additionally, the following  equality holds
\begin{align*}
(1-C_\omega)^{-1}I&=I+\sum_{j=1}^{n(\xi)}C_{\omega_j}(1-C_\omega)^{-1}I+\left[1+\sum_{j=1}^{n(\xi)}C_{\omega_j}(1-C_\omega)^{-1}\right]\\
&\times\left[1-\sum_{j\neq k}C_{\omega_j}C_{\omega_k}((1-C_\omega)^{-1})\right]^{-1}
\left(\sum_{j\neq k}C_{\omega_j}C_{\omega_k}(1-C_\omega)^{-1}\right)I
\\& \triangleq I+\sum_{j=1}^{n(\xi)}C_{\omega_j}(1-C_\omega)^{-1}I+J_1J_2J_3I,
\end{align*}
where
\begin{align*}
J_1&=1+\sum_{j=1}^{n(\xi)}C_{\omega_j}(1-C_\omega)^{-1},\\
J_2&=\left[1-\sum_{j\neq k}C_{\omega_j}C_{\omega_k}((1-C_\omega)^{-1})\right]^{-1},\\
J_3&=\sum_{j\neq k}C_{\omega_j}C_{\omega_k}(1-C_\omega)^{-1}.
\end{align*}
It follows from Cauchy-Schwartz inequality that
\begin{align}
\left|\int_{\Sigma^{(\natural)}}J_1J_2J_3Iw\right|&\leq||J_1||_{L^2(\Sigma^{(\natural)}_j)}||J_2||
_{L^2(\Sigma^{(\natural)}_j)}||J_3||_{{L^\infty(\Sigma^{(\natural)}_j)}\rightarrow{L^2(\Sigma^{(\natural)}_j)}}
||w||_{L^2(\Sigma^{(\natural)}_j)}\nonumber\\
&\lesssim t^{-3/4}t^{-1/4}=O(t^{-1}).
\end{align}
\end{proof}

It follows from Proposition \ref{PP-asy-phase} that RH problem \ref{RHP-Mmod} can be reduced to a model RH problem solved by parabolic cylinder functions at each points $\xi_j$. For the points $\xi_j$, the phase function $\theta(z)$ can be expressed by
\begin{align}
\theta(z)=\theta(\xi_j)+\frac{\theta''(\xi_j)}{2}(z-\xi_j)^2+O(|z-\xi_j|^3),\ z\to\xi_j.
\end{align}
As in Section 8.2 \cite{YDo} and Lemma 3.35 \cite{PIZbook}, the higher-order expansion term of the phase function $\theta(z)$ decays rapidly. Now considering the local RH problem for $M^{(mod),j}(z)$ at $\xi_j$, $j=1,2,\ldots, n(\xi).$
\begin{RHP}\label{RHP-local}
Find a  $2\times2$ matrix-valued function $M^{(mod),j}(z)$ such that
\begin{enumerate}[(I)]
\item $M^{(mod),j}(z)$ is analytic in $\mathbb{C}\setminus\Sigma_{j}^{mod}$.
\item Jump condition:
\begin{align}\label{jump-local}
M_{+}^{(mod),j}(z)=M_{-}^{(mod),j}(z)V^{(mod),j}(z),~~~~z\in\Sigma_{j}^{mod},
\end{align}
where
\begin{align}
V^{(mod),j}(z)=\left\{\begin{aligned}
                 &\left(
                   \begin{array}{cc}
                     1 & \frac{r(\xi_j)}{1-|r(\xi_j)|^2}T_j(\xi_j)^{2}(z-\xi_j)
                       ^{2i\eta v(\xi_j)}e^{-2it\theta} \\
                     0 & 1 \\
                   \end{array} \right),&&z\in\Sigma_{j1}^{mod},\\
                   &\left(
                     \begin{array}{cc}
                       1 & 0 \\
                       -\frac{\overline{r}(\xi_j)}{1-|r(\xi_j)|^2}T_j(\xi_j)^{-2}(z-\xi_j)^{-2i\eta v(\xi_j)}
                       e^{2it\theta} & 1 \\
                     \end{array}
                   \right),&&z\in\Sigma_{j2}^{mod},\\
                   &\left(
                     \begin{array}{cc}
                       1 & r(\xi_j)T_j(\xi_j)^{2}(z-\xi_j)^{2i\eta v(\xi_j)}e^{-2it\theta} \\
                       0 & 1 \\
                     \end{array}
                   \right),&&z\in\Sigma_{j3}^{mod},\\
                   &\left(
                     \begin{array}{cc}
                       1 &  0 \\
                       -\overline{r}(\xi_j)T_j(\xi_j)^{-2}(z-\xi_j)^{-2i\eta v(\xi_j)}e^{2it\theta} & 1 \\
                     \end{array}
                   \right),&&z\in\Sigma_{j4}^{mod}.
                   \end{aligned}
                   \right.
\end{align}
\item Asymptotic behavior: $M^{(mod),j}(z)=I+\mathcal {O}(z^{-1})$, as $z\rightarrow\infty$.
\end{enumerate}
\end{RHP}

We give a detailed proof at phase point $\xi_1$ below for $\xi\in(-1/4,0)\cup(0,2)$, and the rest are similar. The jump contours corresponding to $\xi\in(0,2)$ and $\xi\in(-1/4,0)$ are show in Fig.\ref{Fig-phase1} and Fig.\ref{Fig-phase2}, respectively.

\begin{figure}[H]
\centering
\begin{tikzpicture}[scale=0.7]
\draw[-][red, dashed](-6,0)--(6,0);
\draw[-][thick](-4,-4)--(4,4);
\draw[-][thick](-4,4)--(4,-4);
\draw[->][thick](2,2)--(3,3);
\draw[->][thick](-4,4)--(-3,3);
\draw[->][thick](-4,-4)--(-3,-3);
\draw[->][thick](2,-2)--(3,-3);
\draw[fill] (3.2,5)node[below][red]{$\Sigma_{24}^{mod}$};
\draw[fill] (3.2,-5)node[above][red]{$\Sigma_{23}^{mod}$};
\draw[fill] (-3.2,5)node[below][red]{$\Sigma_{21}^{mod}$};
\draw[fill] (-3.2,-5)node[above][red]{$\Sigma_{22}^{mod}$};
\draw[fill] (0,0)node[below]{$\xi_1$};
\draw[fill] [blue] (7.2,2.5)node[below]{$
\left(
  \begin{array}{cc}
    1 & 0 \\
    -\overline{r}(\xi_1)T_{2}(\xi_1)^2(z-\xi_1)^{-2iv(\xi_1)}e^{2it\theta} & 1 \\
  \end{array}
\right)
$};
\draw[fill] [blue] (7.5,-1)node[below]{$
\left(
  \begin{array}{cc}
    1 & r(\xi_1)T_{1}(\xi_1)^{2}(z-\xi_1)^{2iv(\xi_1)}e^{-2it\theta} \\
    0 & 1 \\
  \end{array}
\right)
$};
\draw[fill][blue] (-8,2.5)node[below]{$
\left(
  \begin{array}{cc}
    1 & \frac{r(\xi_1)}{1-|r(\xi_1)|^{2}}T_{1}(\xi_1)^{2}(z-\xi_1)^{2iv(\xi_1)}e^{-2it\theta} \\
    0 & 1 \\
  \end{array}
\right)
$};
\draw[fill] [blue] (-8.4,-1)node[below]{$
\left(
  \begin{array}{cc}
    1 & 0 \\
    -\frac{\overline{r}(\xi_1)}{1-|r(\xi_1)|^{2}}T_{1}(\xi_1)^{-2}(z-\xi_1)^{-2iv(\xi_1)}e^{2it\theta} & 1 \\
  \end{array}
\right)
$};
\end{tikzpicture}
\caption{The jump condition $\Sigma^{(\natural)}_{1}$ and jump matrix in the region $\xi\in(0,2)$ for the phase point $\xi_1$.}\label{Fig-phase1}
\end{figure}
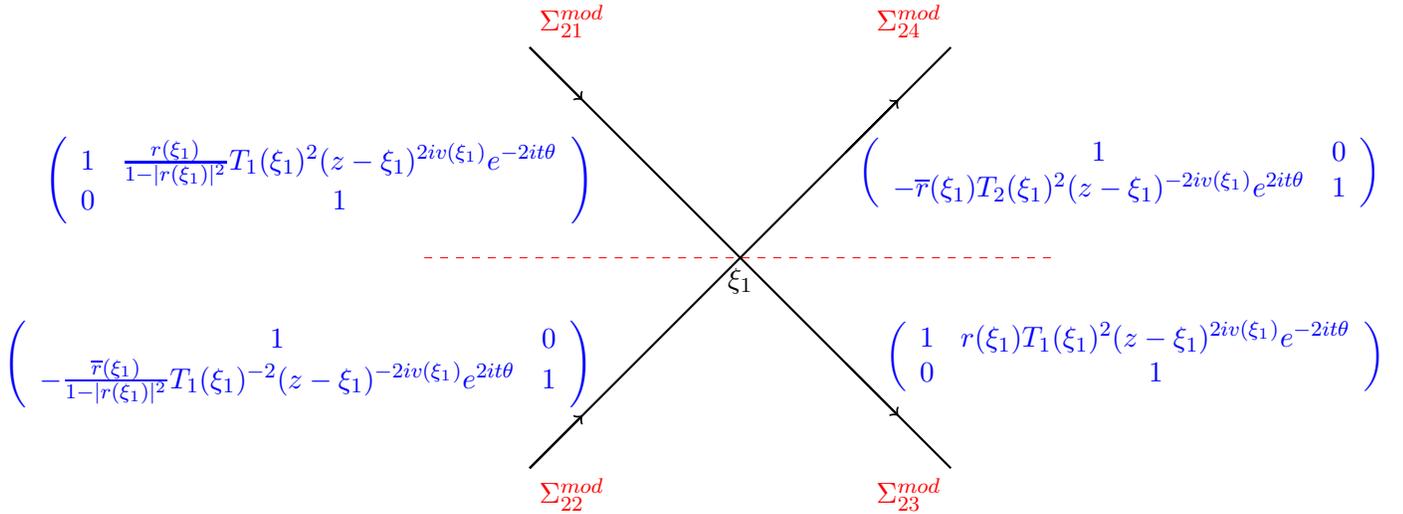

\begin{figure}[H]
\centering
\begin{tikzpicture}[scale=0.7]
\draw[-][red, dashed](-6,0)--(6,0);
\draw[-][thick](-4,-4)--(4,4);
\draw[-][thick](-4,4)--(4,-4);
\draw[->][thick](2,2)--(3,3);
\draw[->][thick](-4,4)--(-3,3);
\draw[->][thick](-4,-4)--(-3,-3);
\draw[->][thick](2,-2)--(3,-3);
\draw[fill] (3.2,5)node[below][red]{$\Sigma_{24}^{mod}$};
\draw[fill] (3.2,-5)node[above][red]{$\Sigma_{23}^{mod}$};
\draw[fill] (-3.2,5)node[below][red]{$\Sigma_{21}^{mod}$};
\draw[fill] (-3.2,-5)node[above][red]{$\Sigma_{22}^{mod}$};
\draw[fill] (0,0)node[below]{$\xi_1$};
\draw[fill] [blue] (8,2.5)node[below]{$
\left(
  \begin{array}{cc}
    1 & -\frac{r(\xi_1)}{1-|r(\xi_1)|^{2}}T_{1}(\xi_1)^{2}(z-\xi_1)^{-2iv(\xi_1)}e^{-2it\theta} \\
    0 & 1 \\
  \end{array}
\right)
$};
\draw[fill] [blue] (8,-1)node[below]{$
\left(
  \begin{array}{cc}
    1 & 0 \\
    \frac{\overline{r}(\xi_1)}{1-|r(\xi_1)|^{2}}T_{1}(\xi_1)^{-2}(z-\xi_1)^{2iv(\xi_1)}e^{2it\theta} & 1 \\
  \end{array}
\right)
$};
\draw[fill][blue] (-7,2.5)node[below]{$
\left(
  \begin{array}{cc}
    1 & 0 \\
    -\overline{r}(\xi_1)T_{1}(\xi_1)^{-2}(z-\xi_1)^{2iv(\xi_1)}e^{2it\theta} & 1 \\
  \end{array}
\right)
$};
\draw[fill] [blue] (-7,-1)node[below]{$
\left(
  \begin{array}{cc}
    1 & r(\xi_1)T_{1}(\xi_1)^2(z-\xi_1)^{-2iv(\xi_1)}e^{-2it\theta}  \\
    0 & 1 \\
  \end{array}
\right)
$};
\end{tikzpicture}
\caption{The jump condition $\Sigma^{(\natural)}_{2}$ and jump matrix in the region $\xi\in(-1/4,0)$ for the phase point $\xi_1$.}\label{Fig-phase2}
\end{figure}
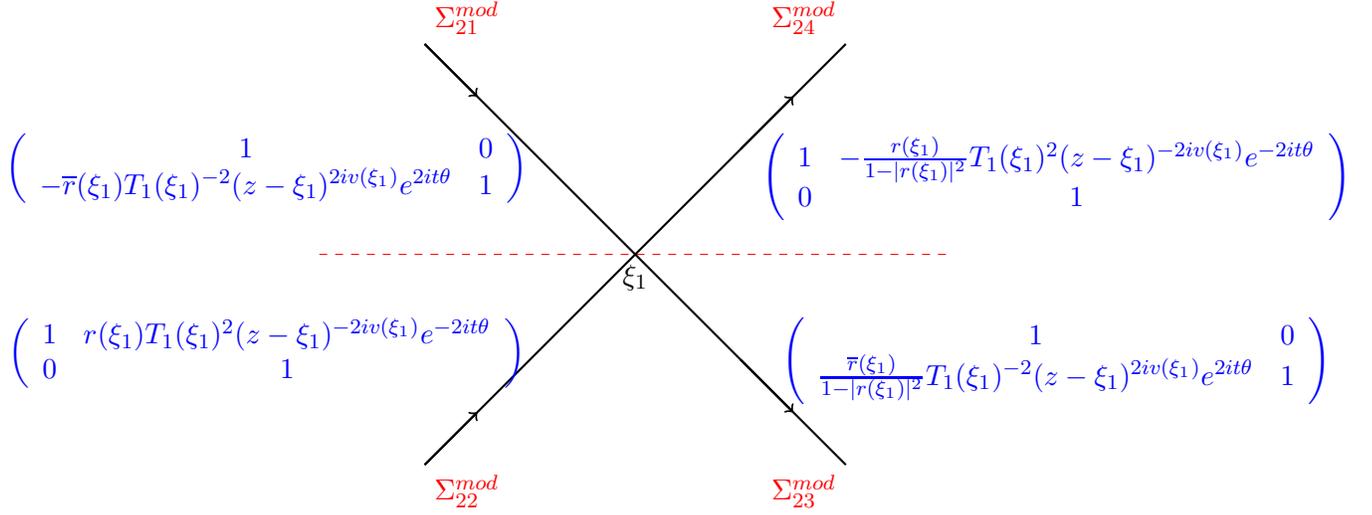

Introduce the variable
\begin{align}
\zeta(z)=(2t\epsilon_j\phi^{''}(\xi_j))^{1/2}(z-\xi_j),~~~j=1,2,\ldots,n(\xi),
\end{align}
and let
\begin{align}\label{e60}
r_{\xi_j}=r(\xi_j)T_{j}(\xi_j)^{2}e^{-2it\theta(\xi_j)}e^{-i\eta v(\xi_j)\ln(2t\eta\ \theta^{''}(\xi_j))},
\end{align}
where $|r(\xi_j)|^2=|r_{\xi_j}|^2$. Taking the branch of the logarithm with $0<\arg\zeta<2\pi$ for $\xi\in(0,2)$ and $-\pi<\arg\zeta<\pi$ for $\xi\in(-1/4,0)$. Following the idea \cite{Deift-1994,toda-2009}, one obtains the proposition as follows.
\begin{prop}\label{pp-phase-sol}
As $t\to\infty$,
\begin{align}
M^{(mod)}(z)=I+t^{-1/2}\sum_{k=1}^{n(\xi)}\frac{A_k(\xi)}{z-\xi_k}+O(t^{-1}),
\end{align}
where
\begin{align}
A_k(\xi)=\frac{1}{\sqrt{2\eta \theta''(\xi_k)}}\left(
                                                 \begin{array}{cc}
                                                   0 & -\beta_{12}^{\xi_k} \\
                                                   i\beta_{21}^{\xi_k} & 0 \\
                                                 \end{array}
                                               \right)
\end{align}
and as $\xi\in(0,2)$, $k=2,4$, as $\xi\in(-1/4,0)$, $k=1,3,5,7$
\begin{align}
&\beta_{12}^{\xi_k}=-\frac{\sqrt{2\pi}e^{-\frac{-\pi v(\xi_k)}{2}}e^{\frac{i\pi}{4}}}{\overline{r}_{\xi_k}
\Gamma(-iv(\xi_k))},\ \beta_{12}^{\xi_k}\beta_{21}^{\xi_k}=v(\xi_k),\\
&r_{\xi_k}=r(\xi_k)T_k(\xi)^2e^{-2it\theta(\xi_k)}\exp(iv(\xi_k)\ln(2t\theta''(\xi_k))),\\
&\arg(\beta_{12}^{\xi_k})=\frac{\pi}{4}+\arg(r_{\xi_k})-\arg(\Gamma(-iv(\xi_k))).
\end{align}
As $\xi\in(0,2)$, $k=1,3$, as $\xi\in(-1/4,0)$, $k=2,4,6,8$
\begin{align}
&\beta_{12}^{\xi_k}=-\frac{\sqrt{2\pi}e^{\frac{3\pi v(\xi_k)}{2}}e^{-\frac{5i\pi}{4}}}{\overline{r}_{\xi_k}
\Gamma(iv(\xi_k))},\ \beta_{12}^{\xi_k}\beta_{21}^{\xi_k}=v(\xi_k),\\
&r_{\xi_k}=r(\xi_k)T_k(\xi)^2e^{-2it\theta(\xi_k)}\exp(iv(\xi_k)\ln(-2t\theta''(\xi_k))),\\
&\arg(\beta_{12}^{\xi_k})=-\frac{5\pi}{4}+\arg(r_{\xi_k})-\arg(\Gamma(iv(\xi_k))).
\end{align}
\end{prop}
\begin{proof}
See  Appendix \ref{AppA}.
\end{proof}
\subsection{The small-norm RH problem}
In this section, we will consider the problem that the error function $E(z)$ defined by \eqref{dec-M2} by the small norm RH problem. In reality, the error function such that
\begin{RHP}\label{RHP-error}
Find a $2\times2$ matrix function $E(z)$ satisfies that
\begin{enumerate}[(I)]
\item $E(z)$ is analytic in $\mathbb{C}\setminus\Sigma^{err}$, where $$\Sigma^{err}
=\partial U(\xi)\cup\left(\Sigma^{(2)}\setminus U(\xi)\right)$$ shown in Fig.\ref{figerror}.
\item Jump condition:
\begin{align}
E_{+}(z)=E_{-}(z)V^{err}(z),~~~~~z\in\Sigma^{err},
\end{align}
where
\begin{align*}
V^{err}(z)=\left\{\begin{aligned}
&M^{(sol)}(z)V^{(2)}(z)M^{(sol)}(z)^{-1},&&z\in\Sigma^{(2)}\setminus U(\xi),\\
&M^{(sol)}(z)M^{(mod)}(z)M^{(sol)}(z)^{-1},&&z\in\partial U(\xi).
\end{aligned}\right.
\end{align*}
\item Asymptotic behavior
\begin{align}
E(z)=I+O(z^{-1}),~~~~~~z\rightarrow\infty.
\end{align}
\end{enumerate}
\end{RHP}

\begin{figure}[H]
	\centering
\subfigure[]{
	\begin{tikzpicture}
\draw(4.35,0.18)--(5,0.5);
\draw[->](4.35,0.18)--(4.5,0.25);
\draw(3.64,0.15)--(2.5,0.6);
\draw[-<](3.64,-0.15)--(3.25,-0.3);
\draw(4.35,-0.18)--(5,-0.5);
\draw[-<](3.64,0.15)--(3.25,0.3);
\draw(3.64,-0.15)--(2.5,-0.6);
\draw[->](4.35,-0.18)--(4.5,-0.25);
\draw(-4.35,0.18)--(-5,0.5);
\draw[-<](-4.35,0.18)--(-4.5,0.25);
\draw(-3.64,0.15)--(-2.5,0.6);
\draw[->](-3.64,-0.15)--(-3.25,-0.3);
\draw(-4.35,-0.18)--(-5,-0.5);
\draw[->](-3.64,0.15)--(-3.25,0.3);
\draw(-3.64,-0.15)--(-2.5,-0.6);
\draw[-<](-4.35,-0.18)--(-4.5,-0.25);
\draw(-0.64,0.15)--(0,0.5);
\draw[->](-0.64,0.15)--(-0.4,0.28);
\draw(-1.35,0.16)--(-2.5,0.6);
\draw[-<](-1.35,-0.16)--(-1.75,-0.31);
\draw(-0.64,-0.15)--(0,-0.5);
\draw[-<](-1.35,0.16)--(-1.75,0.31);
\draw(-1.35,-0.16)--(-2.5,-0.6);
\draw[->](-0.64,-0.15)--(-0.4,-0.28);
\draw[dashed](-5.5,0)--(5.5,0)node[right]{Re$z$};
\draw [-latex](5.5,0)--(5.6,0);
\draw(0.64,0.15)--(0,0.5);
\draw[-<](0.64,0.15)--(0.4,0.28);
\draw(1.35,0.16)--(2.5,0.6);
\draw[->](1.35,-0.16)--(1.75,-0.31);
\draw(0.64,-0.15)--(0,-0.5);
\draw[->](1.35,0.16)--(1.75,0.31);
\draw(1.35,-0.16)--(2.5,-0.6);
\draw[-<](0.64,-0.15)--(0.4,-0.28);
\draw[->](2.5,0)--(2.5,0.6);
\draw[->](2.5,0)--(2.5,-0.6);
\draw[->](-2.5,0)--(-2.5,0.6);
\draw[->](-2.5,0)--(-2.5,-0.6);
\draw[->](0,0)--(0,0.5);
\draw[->](0,0)--(0,-0.5);
\coordinate (I) at (0,0);
\fill (I) circle (1pt) node[below right] {$0$};
\coordinate (A) at (-4,0);
\fill[blue] (A) circle (1pt) node[below] {$\xi_4$};
\coordinate (b) at (-1,0);
\fill[blue] (b) circle (1pt) node[below] {$\xi_3$};
\coordinate (e) at (4,0);
\fill[blue] (e) circle (1pt) node[below] {$\xi_1$};
\coordinate (f) at (1,0);
\draw[thick,red](1,0) circle (0.4);
\fill[blue] (f) circle (1pt) node[below] {$\xi_2$};
\draw[thick,red](4,0) circle (0.4);
\draw[thick,red](-1,0) circle (0.4);
\draw[thick,red](-4,0) circle (0.4);
\coordinate (c) at (-2,0);
\fill[red] (c) circle (1pt) node[below] {\scriptsize$-1$};
\coordinate (d) at (2,0);
\fill[red] (d) circle (1pt) node[below] {\scriptsize$1$};
\end{tikzpicture}
}
\subfigure[]{
\begin{tikzpicture}
\draw[dashed](-6.8,0)--(7,0)node[right]{ Re$z$};
\draw [-latex](7,0)--(7.1,0);
\coordinate (I) at (0,0);
\fill (I) circle (1pt) node[below right] {$0$};
\coordinate (c) at (-3,0);
\fill[red] (c) circle (1pt) node[below] {\scriptsize$-1$};
\coordinate (D) at (3,0);
\fill[red] (D) circle (1pt) node[below] {\scriptsize$1$};
\draw(-0.575,0.187)--(-0,0.7);
\draw[->](-0.575,0.187)--(-0.4,0.35);
\draw(-1.03,0.186)--(-1.5,0.6);
\draw[-<](-1.03,-0.186)--(-1.15,-0.3);
\draw(-0.575,-0.187)--(-0,-0.7);
\draw[-<](-1.03,0.186)--(-1.15,0.3);
\draw(0.575,0.187)--(0,0.7);
\draw[-<](0.575,0.187)--(0.4,0.35);
\draw(1.03,0.186)--(1.5,0.6);
\draw[->](1.03,-0.186)--(1.15,-0.3);
\draw(0.575,-0.187)--(0,-0.7);
\draw[->](1.03,0.186)--(1.15,0.3);
\draw(1.03,-0.186)--(1.5,-0.6);
\draw[-<](0.575,-0.187)--(0.4,-0.35);
\draw(-1.03,-0.186)--(-1.5,-0.6);
\draw[->](-0.575,-0.187)--(-0.4,-0.35);
\draw(-1.97,0.187)--(-1.5,0.6);
\draw[-<](-2.43,0.187)--(-2.65,0.35);
\draw(-1.97,-0.187)--(-1.5,-0.6);
\draw[->](-1.97,-0.187)--(-1.85,-0.3);
\draw(-2.43,0.187)--(-3.1,0.7);
\draw[->](-1.97,0.187)--(-1.85,0.3);
\draw(-2.43,-0.187)--(-3.1,-0.7);
\draw[-<](-2.43,-0.187)--(-2.65,-0.35);
\draw(1.97,0.187)--(1.5,0.6);
\draw[->](2.43,0.187)--(2.65,0.35);
\draw(1.97,-0.187)--(1.5,-0.6);
\draw[-<](1.97,-0.187)--(1.85,-0.3);
\draw(2.43,0.187)--(3.1,0.7);
\draw[-<](1.97,0.187)--(1.85,0.3);
\draw(2.43,-0.187)--(3.1,-0.7);
\draw[->](2.43,-0.187)--(2.65,-0.35);
\draw(-5.64,0.18)--(-6.5,0.9);
\draw[-<](-5.64,0.18)--(-5.96,0.446);
\draw(-5.16,0.18)--(-4.7,0.6);
\draw[->](-5.16,-0.18)--(-4.93,-0.39);
\draw(-5.64,-0.18)--(-6.5,-0.9);
\draw[->](-5.16,0.18)--(-4.93,0.39);
\draw(-5.16,-0.18)--(-4.7,-0.6);
\draw[-<](-5.64,-0.18)--(-5.96,-0.446);
\draw(-3.75,0.187)--(-3.1,0.7);
\draw[->](-3.75,0.187)--(-3.55,0.35);
\draw(-4.25,0.187)--(-4.7,0.6);
\draw[-<](-4.25,-0.187)--(-4.42,-0.35);
\draw(-3.75,-0.187)--(-3.1,-0.7);
\draw[-<](-4.25,0.187)--(-4.42,0.35);
\draw(-4.25,-0.187)--(-4.7,-0.6);
\draw[->](-3.75,-0.187)--(-3.55,-0.35);
\draw(3.75,0.187)--(3.1,0.7);
\draw[-<](3.75,0.187)--(3.55,0.35);
\draw(4.25,0.187)--(4.7,0.6);
\draw[->](4.25,-0.187)--(4.42,-0.35);
\draw(3.75,-0.187)--(3.1,-0.7);
\draw[->](4.25,0.187)--(4.42,0.35);
\draw(4.25,-0.187)--(4.7,-0.6);
\draw[-<](3.75,-0.187)--(3.55,-0.35);
\draw[->](-1.5,0)--(-1.5,0.6);
\draw[->](-1.5,0)--(-1.5,-0.6);
\draw[->](-4.7,0)--(-4.7,0.6);
\draw[->](-4.7,0)--(-4.7,-0.6);
\draw[->](-3.1,0)--(-3.1,0.7);
\draw[->](-3.1,0)--(-3.1,-0.7);
\draw(5.64,0.18)--(6.5,0.9);
\draw[->](5.64,0.18)--(5.96,0.446);
\draw(5.16,0.18)--(4.7,0.6);
\draw[-<](5.16,-0.18)--(4.93,-0.39);
\draw(5.64,-0.18)--(6.5,-0.9);
\draw[-<](5.16,0.18)--(4.93,0.39);
\draw(5.16,-0.18)--(4.7,-0.6);
\draw[->](5.64,-0.18)--(5.96,-0.446);
\draw[->](1.5,0)--(1.5,0.6);
\draw[->](1.5,0)--(1.5,-0.6);
\draw[->](4.7,0)--(4.7,0.6);
\draw[->](4.7,0)--(4.7,-0.6);
\draw[->](3.1,0)--(3.1,0.7);
\draw[->](3.1,0)--(3.1,-0.7);
\draw[->](0,0)--(0,0.7);
\draw[->](0,0)--(0,-0.7);
\coordinate (A) at (-5.4,0);
\fill[blue] (A) circle (1pt) node[below] {$\xi_8$};
\draw[thick,red](-5.4,0) circle (0.3);
\coordinate (b) at (-4,0);
\draw[thick,red](-4,0) circle (0.3);
\fill[blue] (b) circle (1pt) node[below] {$\xi_7$};
\coordinate (C) at (-0.8,0);
\draw[thick,red](-0.8,0) circle (0.3);
\fill[blue] (C) circle (1pt) node[below] {$\xi_5$};
\coordinate (d) at (-2.2,0);
\draw[thick,red](-2.2,0) circle (0.3);
\fill[blue] (d) circle (1pt) node[below] {$\xi_6$};
\coordinate (E) at (5.4,0);
\draw[thick,red](5.4,0) circle (0.3);
\fill[blue] (E) circle (1pt) node[below] {$\xi_1$};
\coordinate (R) at (4,0);
\draw[thick,red](4,0) circle (0.3);
\fill[blue] (R) circle (1pt) node[below] {$\xi_2$};
\coordinate (T) at (0.8,0);
\draw[thick,red](0.8,0) circle (0.3);
\fill[blue] (T) circle (1pt) node[below] {$\xi_4$};
\coordinate (Y) at (2.2,0);
\draw[thick,red](2.2,0) circle (0.3);
\fill[blue] (Y) circle (1pt) node[below] {$\xi_3$};
\end{tikzpicture}
}
\caption{ ($a$) and ($b$) denote the jump contour $\Sigma^{err}$ as $\xi\in(0,2)$ and $\xi\in(-1/4,0)$, respectively. The red circles represent the jump contour $U(\xi)$ at each phase points $\xi_j$.}\label{figerror}
\end{figure}
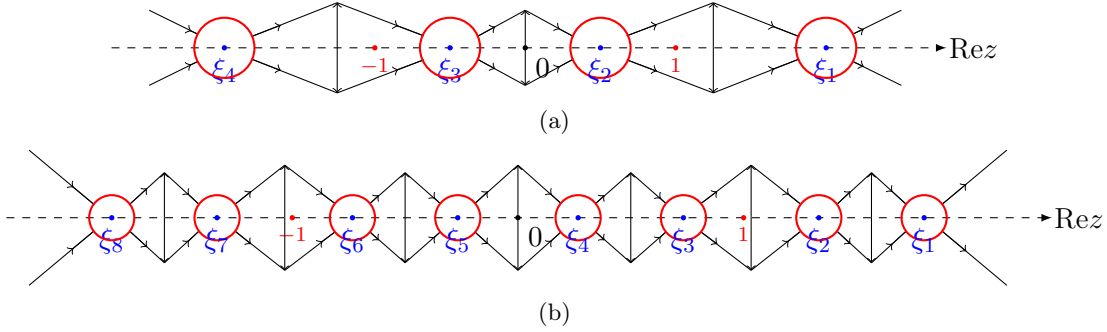

From the small norm theory of RH problem \cite{KMM-2003,Deift-1994,DZ-CPAM},  it is necessary to further analyze the estimation of the jump matrix $V^{err}$ of the error function $E(z)$. It follows from Proposition \ref{pp-est-V2-phase} that
\begin{align}\label{est-verr}
\left|\left|V^{err}-I\right|\right|_{L^2}\lesssim \left\{\begin{aligned}
&e^{-tK}, \ z\in\Sigma_{kj}\setminus U(\xi),\\
&e^{-tK'}, \ z\in\Sigma_{j}^{\pm},\\
&t^{-1/2}, \ z\in\partial U(\xi),
\end{aligned}\right.
\end{align}
where it takes advantage of the fact that $M^{(sol)}$ is bounded for $z\in\partial U(\xi)$, which implies that the solution to RH problem \ref{RHP-error}  exists and is unique expressed by
\begin{align}\label{sol-err}
E(z)=I+\frac{1}{2\pi i}\int_{\Sigma^{err}}\frac{(I+\rho(s))(V^{err}-I)}{s-z}ds,
\end{align}
where $\rho\in L^{2}(\Sigma^{err})$ is the unique solution of $(1-C_{V^{err}})\rho=C_{V^{err}}I$ and $C_{V^{err}}: L^{2}(\Sigma^{err})\to L^{2}(\Sigma^{err})$  is integral operator on $\Sigma^{err}$
\begin{align}
    C_{V^{err}} (f)(z)=C_{-}{f(V^{err}-I)}=\lim_{s\rightarrow z}\frac{1}{2\pi i}\int_{\Sigma^{err}}\frac{f(s)(V^{err}(s)-I)}{s-z}ds.
\end{align}

Then, an application of inequality \eqref{est-verr} yields
\begin{align}
    \Vert C_{V^{err}}\Vert \leqslant \Vert C_{-} \Vert_{L^{2}\rightarrow L^{2}}\Vert V^{err}-I \Vert_{L^{2}}
    \lesssim O(t^{-\frac{1}{2}}),
\end{align}
which indicates that as $t\to+\infty$, $1-C_{V^{err}}$ is reversible and $\rho$ exists and in unique. Also
\begin{align}
    \Vert \rho \Vert_{L^{\infty}{(\Sigma^{err}})}\lesssim \frac{\Vert C_{V^{err}}\Vert}{1-\Vert C_{V^{err}}\Vert}\lesssim t^{-\frac{1}{2}}.
\end{align}

In order to recover the solution $q(x,t)$ of mCH equation, we need to further analyze the asymptotic behavior of error function $E(z)$ as $z\to i$, which is reflected in the following proposition.
\begin{prop}\label{asy-err}
As $z\to i$, the error function $E(z)$ satisfies that
\begin{align}
E(z)=E(i)+E_{1}(z-i)+O((z-i)^2),
\end{align}
where $E(i)$ and $E_{1}$ are determined by
\begin{subequations}\label{express-err}
\begin{align}
&E(i)=I+\frac{1}{2\pi i}\int_{\Sigma^{(E)}}\dfrac{\left( I+\varpi(s)\right) (V^{err}-I)}{s-i}ds,\\
&E_1=-\frac{1}{2\pi i}\int_{\Sigma^{(E)}}\dfrac{\left( I+\varpi(s)\right) (V^{err}-I)}{(s-i)^2}ds,
\end{align}
\end{subequations}
with the asymptotic behaviors
\begin{subequations}\label{asy-ei}
\begin{align}
&E(i)=I+t^{-1/2}f_1+O(t^{-1}),\\
&E_1=t^{-1/2}f_2+O(t^{-1}),
\end{align}
\end{subequations}
where
\begin{subequations}\label{asy-ff}
\begin{align}
&f_1=-\sum_{j=1}^{n(\xi)}\frac{1}{(1-\xi_j^{-2})(\xi_j-i)}M^{(sol)}(\xi_j)A_{j}\sigma_2
M^{(sol)}(\xi_j)\sigma_2,\\
&f_2=\sum_{j=1}^{n(\xi)}\frac{1}{(1-\xi_j^{-2})(\xi_j-i)^2}M^{(sol)}(\xi_j)A_{j}\sigma_2
M^{(sol)}(\xi_j)\sigma_2.\label{asy-f2}
\end{align}
\end{subequations}

For convenience, let denote the following notations
\begin{align}
h_{11}=-&\left[(f_2^{21}M^{(sol)}_{12}(i)+f_2^{21}M^{(sol)}_{12}(i))T(i)+(f_{1}^{11}
M^{(sol)}_{1,12}+f_{1}^{12}M^{(sol)}_{1,22})T(i) \right.\nonumber \\&+
\left.(f_{1}^{11}
M^{(sol)}_{12}(i)+f_{1}^{12}M^{(sol)}_{22}(i))T(i)T_0\right]
(f_{1}^{11}M^{(sol)}_{11}(i)+f_{1}^{12}M^{(sol)}_{21}(i))T(i) \nonumber \\
&-\left[(f_2^{21}M^{(sol)}_{11}(i)+f_{2}^{22}M^{(sol)}_{22}(i))T(i)^{-1}+
(f_1^{21}M^{(sol)}_{1,11}+f_{1}^{22}M^{(sol)}_{1,21})T(i)^{-1}\right.\nonumber \\&
-\left.(f_{1}^{21}
M^{(sol)}_{11}(i)+f_{1}^{22}M^{(sol)}_{21}(i))T(i)^{-1}T_0\right]
(f_{1}^{11}M^{(sol)}_{11}(i)+f_{1}^{12}M^{(sol)}_{21}(i))^{-1}T(i)^{-1},
\end{align}
and
\begin{align}
h_{12}=\ln\left[(1+t^{-1/2}f_{1}^{11})+\frac{M^{(sol)}_{21}(i)}{M^{(sol)}_{11}(i)}
t^{-1/2}f_{1}^{12}\right],
\end{align}
where $M^{(sol)}_{1}$, $T(i)$, and $T_0$ are defined by \eqref{exp-Msol-i}, Proposition \ref{PP-T}, as well as $f_{i}^{jk}$ ($i,j,k=1,2$) denote the elements of the matrices $f_1$ and $f_2$.

\end{prop}

\begin{proof}
The proof of expression \eqref{express-err} is trivial. The following mainly proves \eqref{asy-ei} and \eqref{asy-ff}. Applying the expression \eqref{sol-err} yields
\begin{align}
E(i)&=I+\frac{1}{2\pi i}\left(\oint_{\partial U(\xi)}\frac{V^{err}(s)-I}{s-i}\ ds+
\int_{\Sigma^{err}\setminus\partial U(\xi)}\frac{V^{err}(s)-I}{s-i}\ ds +\int_{\Sigma^{err}}
\frac{\rho(s)(V^{err}(s)-I)}{s-i}\ ds\right)\nonumber\\
&=I+\frac{1}{2\pi i}\oint_{\partial U(\xi)}\frac{V^{err}(s)-I}{s-i}\ ds+O(t^{-1})\nonumber \\
&=I+\frac{1}{2\pi i}\oint_{\partial U(\xi)}
\frac{M^{(sol)}(s)(M^{(mod)}-I)M^{(sol)}(s)^{-1}}{s-i}\ ds+O(t^{-1})\nonumber \\
&=I+t^{-1/2}\frac{1}{2\pi i}\sum_{j=1}^{n(\xi)}\oint_{\partial U(\xi)}
\frac{M^{(sol)}(s)A_j(\xi)M^{(sol)}(s)^{-1}}{(s-i)(s-\xi_j)}\ ds+O(t^{-1})\nonumber \\
&=I-t^{-1/2}\sum_{j=1}^{n(\xi)}
\frac{M^{(sol)}(\xi_j)A_j(\xi)M^{(sol)}(\xi_j)^{-1}}{(\xi_j-i)}\ ds+O(t^{-1})\nonumber \\
&=I-t^{-1/2}\sum_{j=1}^{n(\xi)}
\frac{M^{(sol)}(\xi_j)A_j(\xi)\sigma_2
M^{(sol)}(\xi_j)\sigma_2}{(\xi_j-i)(1-\xi_j^{-2})}\ ds+O(t^{-1})\triangleq
I+t^{-1/2}f_1+O(t^{-1}).
\end{align}
Similarly, the expression \eqref{asy-f2} can be  obtained in the same way.

\end{proof}

\subsection{Analysis on pure $\overline{\partial}$-problem}

Similar to Section \ref{subsec4.4}, from this section we will analyze the pure $\overline{\partial}$-problem $m^{(3)}(z)$
\begin{align}\label{sol-m3}
m^{(3)}(z)= M^{(2)}(z)M^{(2)}_{rhp}(z)^{-1},
\end{align}
in the presence of steady-state phase points, that is, to remove soliton components from $M^{(2)}(z)$, which solves the pure $\overline{\partial}$-problem \ref{RHP-Dbar} replaced the condition $W^{(3)}(z)=M^{(rhp)}(z)\overline{\partial}R^{(2)}(z)M^{(rhp)}(z)^{-1}$ by $W^{(3)}(z)=M^{(2)}_{rhp}(z)\overline{\partial}R^{(2)}(z)M^{(2)}_{rhp}(z)^{-1}$. Then the solution to equation \eqref{sol-m3} is given by
\begin{equation}\label{sol-m31}
m^{(3)}(z)=I+\frac{1}{\pi}\iint_\mathbb{C}\dfrac{m^{(3)}(s)W^{(3)} (s)}{s-z}dA(s),
\end{equation}
which can be rewritten as an operator equation
\begin{align}\label{m3-oper}
m^{(3)}(z)=I\cdot(I-\mathbb{P}_z)^{-1},
\end{align}
where the operator $\D{P}_z$ is defined by \eqref{M3-oper}.

\begin{prop}\label{pp-oper-m3}
As $t\to\infty,$ the operator $\D{P}_z$ satisfies the following estimation
\begin{align}
\left|\left|\D{P}_z\right|\right|_{L^{\infty}\to L^{\infty}}\lesssim t^{-1/4}.
\end{align}
which indicates that the operator $(I-\mathbb{P}_z)^{-1}$ exists.
\end{prop}
\begin{proof}
Taking region $\xi\in(-1/4,0)$ as an example, consider that for any $f\in L^{\infty}$
\begin{align}
||\mathbb{P}_zf||_{L^{\infty}}\leq||f||_{L^{\infty}}\frac{1}{\pi}
\int\int_{\D{C}}\frac{|W^{(3)}(s)|}{|s-z|}\ dA(s),
\end{align}
where $W^{(3)}(s)=M^{(sol)}(s)\overline{\partial}R^{(2)}(s)M^{(sol)}(s)^{-1}$ and $W^{(3)}(s)\equiv0$ as out $\overline{\Omega}$. On the other hand, $M^{(sol)}(s)$ and $M^{(sol)}(s)^{-1}$ are bounded as $z\in\overline{\Omega}$. Let $z=\alpha+i\beta$ and $s-\xi_1=u+iv\in\Omega_{11}$, then
\begin{align}
||f||_{L^{\infty}}\frac{1}{\pi}
\int\int_{\Omega_{11}}\frac{|W^{(3)}(s)|}{|s-z|}\ dA(s)
\leq||f||_{L^{\infty}}\frac{1}{\pi}
\int\int_{\Omega_{11}}\frac{|\overline{\partial}R_{11}(s)|}{|s-z|}\ dA(s)
\leq J_5+J_6,
\end{align}
where
\begin{subequations}
\begin{align}
&J_5=\int\int_{\Omega_{11}}\frac{|p_{11}'(s)|e^{2t\Im\theta}}{|s-z|}\ dudv,\label{J5}\\
&J_6=\int\int_{\Omega_{11}}\frac{|s-\xi_1|^{-1/2}e^{2t\Im\theta}}{|s-z|}\ dudv.\label{J6}
\end{align}
\end{subequations}

For $J_5$,  applying H\"{o}lder inequality obtains
\begin{align}
J_5&\leq\int_0^{\infty}\left\||p_{11}'(s)|\right\|_{L^{2}(v,\infty)}\left\||s-z|^{-1}\right\|
_{L^{2}(v,\infty)}
e^{-4\tau(\xi)tv^{2}\frac{v+\xi_1}{4+\xi_1^2+v^2}}\ dv,\nonumber \\
&\lesssim \int_0^{\infty}|v-y|^{-1/2}e^{-4\tau(\xi)tv^{2}\frac{v+\xi_1}{4+\xi_1^2+v^2}}\ dv\nonumber \\
&=\int_0^{y}|v-y|^{-1/2}e^{-4\tau(\xi)tv^{2}\frac{v+\xi_1}{4+\xi_1^2+v^2}}\ dv
+\int_y^{\infty}|v-y|^{-1/2}e^{-4\tau(\xi)tv^{2}\frac{v+\xi_1}{4+\xi_1^2+v^2}}\ dv
\triangleq J_5^1+J_5^2.
\end{align}
For $J_5^1$, it follows from $e^{-s}\leq s^{-1/4}$ that
\begin{align}
J_5^1&\leq\int_{0}^y(y-v)^{-1/2}v^{-1/2}t^{-1/4}\ dv\nonumber \\
&=\int_{0}^yy^{-1/2}(1-v/y)^{-1/2}(v/y)^{-1/2}y^{-1/2}t^{-1/4}\ dv\nonumber \\
&\overset{w=v/y}{=}\int_{0}^1(1-w)^{-1/2}w^{-1/2}\ dwt^{-1/4} \lesssim t^{-1/4}.\label{J51}
\end{align}

On the other hand, for $J_5^2$
\begin{align}
J_5^2&\overset{\kappa=v/y}{=}\int_{0}^{\infty}\kappa^{-1/2}e^{-2\tau(\xi)(\kappa+y)^2
\frac{\kappa+y+\xi_1}{4+\xi_1^2+(\kappa+y)^2}}\ d\kappa\nonumber \\
&\leq\int_{0}^{\infty}\kappa^{-1/2}e^{-2\tau(\xi)ty\frac{\xi_1+y}{4+\xi_1^2+y^2}} d\kappa
\lesssim e^{-2\tau(\xi)ty^2\frac{\xi_1+y}{4+\xi_1^2+y^2}}.\label{J52}
\end{align}

Observing the fact that $\left\||s-\xi_1|^{-1/2}\right\|_{L^p}\lesssim v^{-1/p-1/2}$ and $\left\||z-s|^{-1}\right\|_{L^q}\lesssim |v-y|^{1/q-1}$ for $p>2$ satisfying $1/p+1/q=1$.
Then
\begin{align}
J_6&\lesssim \int_0^{\infty}\left\||s-\xi_1|^{-1/2}\right\|_{L^p}\left\||z-s|^{-1}\right\|_{L^q}
e^{-4\tau(\xi)tv^{2}\frac{v+\xi_1}{4+\xi_1^2+v^2}}\ dv\nonumber \\
&\lesssim \int_0^{\infty} v^{-1/p-1/2}|v-y|^{1/q-1}e^{-4\tau(\xi)tv^{2}\frac{v+\xi_1}{4+\xi_1^2+v^2}}\ dv\nonumber\\
&=\int_{0}^yv^{-1/p-1/2}(y-v)^{1/q-1}e^{-4\tau(\xi)tv^{2}\frac{v+\xi_1}{4+\xi_1^2+v^2}}\ dv+
\int_{y}^{\infty}v^{-1/p-1/2}(v-y)^{1/q-1}e^{-4\tau(\xi)tv^{2}\frac{v+\xi_1}{4+\xi_1^2+v^2}}\ dv
\end{align}
Similar to the proof of inequality \eqref{J51} and \eqref{J52}, one has $J_6\lesssim t^{-1/4}$.
\end{proof}

Since restoring the solution $q(x,t)$ of the mCH equation needs to consider the asymptotic behavior at $z=i$, which is reflected in the following proposition.
\begin{prop}
The solution to \eqref{sol-m3} obeys the estimation as $t\to\infty$
\begin{align}\label{est-m3i}
\left\|m^{(3)}-I\right\|=\left\|\frac{1}{\pi}\int\int_{\D{C}}
\frac{m^{(3)}(s)W^{(3)}(s)}{s-i}\ dA(s)\right\|\lesssim t^{-3/4}.
\end{align}
Then, The expansion at $z=i$ can be expressed as
\begin{align}
m^{(3)}(z)=m^{(3)}(i)+m^{(3)}_1(x,t)(z-i)+O((z-i)^2),
\end{align}
where
\begin{align}\label{expansion-m3}
m^{(3)}_1(x,t)=\frac{1}{\pi}\int\int_{\D{C}}
\frac{m^{(3)}(s)W^{(3)}(s)}{(s-i)^2}\ dA(s),
\end{align}
with
\begin{align}\label{m3-est-34}
\left|m^{(3)}_1(x,t)\right|\lesssim t^{-3/4}.
\end{align}
\end{prop}
\begin{proof}
Similar to the proof of Proposition \ref{pp-oper-m3}, we still take $z\in\Omega_{11}$ in region $\xi\in(-1/4,0)$ as an example, and first consider inequality \eqref{est-m3i}. Observing the fact that $\left\|m^{(3)}\right\|_{L^{\infty}}\leq 1$. Let $s-\xi_1=u+iv$ and $z=x+iy\in\Omega_{11}$, then
\begin{align}
\frac{1}{\pi}\int\int_{\D{C}}\frac{\left|m^{(3)}(s)W^{(3)}(s)\right|}{|s-i|}
\leq J_7+J_8,
\end{align}
where
\begin{subequations}
\begin{align}
&J_7=\int\int_{\Omega_{11}}\frac{|p_{11}'(s)|e^{2t\Im\theta}}{|s-i|}\ dudv,\label{J7}\\
&J_8=\int\int_{\Omega_{11}}\frac{|s-\xi_1|^{-1/2}e^{2t\Im\theta}}{|s-i|}\ dudv.\label{J8}
\end{align}
\end{subequations}

For the integral equation $J_7$, it follows from H\"{o}lder inequality and $|s-i|^{-1}$ has nonzero maximum that
\begin{align}
J_7&\leq\int_0^{\infty}\int_v^{\infty}\left|p_{11}'(s)\right|e^{-\tau(\xi)tv\frac{u^2+v^2+2u\xi_1}
{4+|s|^2}}\ dudv\nonumber \\
&\lesssim \int_0^{\infty}\left\|p_{11}'(s)\right\|_{L^{2}}
\left\|e^{-\tau(\xi)tvu\frac{v+2\xi_1}{
4+v^2+(v+\xi_1)^2}}\right\|_{L^{2}}e^{-\tau(\xi)tv^2\frac{v+2\xi_1}{
4+v^2+(v+\xi_1)^2}}\ dudv\nonumber \\
&\lesssim t^{-1}\int_0^{\infty}\left(\frac{v^2+2v\xi_1}{4+v^2+(v+\xi_1)^2}\right)^2
e^{-\tau(\xi)tv^2\frac{v+2\xi_1}{
4+v^2+(v+\xi_1)^2}}\ dudv\nonumber \\
&=t^{-1}\left[\left(\int_{0}^1+\int_1^{\infty}\right)
\left(\frac{v^2+2v\xi_1}{4+v^2+(v+\xi_1)^2}\right)^2
e^{-\tau(\xi)tv^2\frac{v+2\xi_1}{
4+v^2+(v+\xi_1)^2}}\ dv\right]\nonumber \\
&\triangleq J_7^1+J_7^2.\label{J712}
\end{align}
For the first integral $J_7^1$, that is $v\in(0,1)$, then one has $\frac{v^2+2v\xi_1}{4+v^2+(v+\xi_1)^2}\leq v^{-1}$. Therefore,
\begin{align}
J_7^1\lesssim \int_{0}^1v^{-1/2}e^{-tv^2}\ dv\lesssim
\int_{0}^1t^{-1/4}e^{-(t^{1/4}v^{1/2})^4}\ dt^{1/4}v^{1/2}\lesssim t^{-1/4}.
\end{align}
Similarly, direct calculation shows
\begin{align}
J_7^2\lesssim \int_1^{\infty}e^{-t(v+\xi_1)}\lesssim t^{-1}.
\end{align}
Now our attention turns to proving $J_8$. For $2<p<4$ such that $1/p+1/q=1$, applying H\"{o}lder inequality yields
\begin{align}
J_8&\lesssim \int_0^{\infty}\left\||s-\xi_1|^{-1/2}\right\|_{L^{p}}
\left\|e^{-\tau(\xi)tvu\frac{v+2\xi_1}{
4+v^2+(v+\xi_1)^2}}\right\|_{L^{q}}e^{-\tau(\xi)tv^2\frac{v+2\xi_1}{
4+v^2+(v+\xi_1)^2}}\ dv\nonumber \\
&=t^{-\frac{1}{q}}\left[\left(\int_{0}^1+\int_1^{\infty}\right)v^{\frac{1}{p}-\frac{1}{2}}
\left(\frac{v^2+2v\xi_1}{4+v^2+(v+\xi_1)^2}\right)^{-\frac{1}{q}}
e^{-q\tau(\xi)tv^2\frac{v+2\xi_1}{
4+v^2+(v+\xi_1)^2}}\ dv\right]\nonumber \\
&\triangleq J_8^1+J_8^2.
\end{align}
Similar to the proof of \eqref{J712}, one has $J_8\leq t^{-3/4}$. On the other hand, because $|(s-i)^{-1}|$ is bounded, inequality \eqref{m3-est-34} can be obtained by using results \eqref{est-m3i} and \eqref{expansion-m3}.
\end{proof}

\section{Long-time asymptotic behaviors}\label{sec6}
The main purpose of this section is to give the asymptotic behavior of the mCH equation depending on $\xi$ with nonzero boundary conditions. First, a series of deformations of the original RH problem \ref{RHP-M} are reviewed, including \eqref{M1}, \eqref{M2}, \eqref{M3} and \eqref{Ez}. Therefore, one has
\begin{align}
M(z)=M^{(3)}(z)M^{(err)}(z)M^{(sol)}(z)R^{(2)}(z)^{-1}T(z)^{-\sigma_3}
\end{align}
In order to reconstruct the solution of the mCH equation, consider $z\to i$ outside the region $\overline{\Omega}$, then there is $R^{(2)}(z)\equiv1$. It follows from Propositions \ref{PP-T} and \ref{pp-M3i}
\begin{align}
M(z)=&\left(M^{(3)}(i)+M^{(3)}_1(z)(z-i)\right)M^{(err)}(z)M^{(sol)}(z)
T(i)^{-\sigma_3}\nonumber \\
&\times\left(1-\frac{1}{2\pi i}\int_{\Sigma_b(\xi)}\frac{\ln(1-|r(s)|^2)}{(s-i)^2}\ ds(z-i)\right)^{-\sigma_3}+O((z-i)^2)
\end{align}

For the regions $\xi\in(-\infty,-1/4)\cup(2,+\infty)$, the corresponding steady-state phase point does not exist, then
\begin{align}
M(z)=M^{(sol)}(i)
T(i)^{-\sigma_3}
\left(1-\frac{1}{2\pi i}\int_{\Sigma_b(\xi)}\frac{\ln(1-|r(s)|^2)}{(s-i)^2}\ ds(z-i)\right)^{-\sigma_3}+O((z-i)^2)+O(t^{-1/4}),
\end{align}
and
\begin{align}
M(z)=M^{(sol)}(i)T(i)^{-\sigma_3}+O(t^{-1/4}).
\end{align}

Applying the reconstruction formula \eqref{q-sol} yields
\begin{align}
q(x,t)&=-\partial_zM_{12}(z)|_{z=i}M_{11}(i)-\partial_zM_{21}(z)|_{z=i}M_{11}(i)^{-1}\nonumber \\
=&-\partial_zM^{(sol)}_{12}(z)|_{z=i}M^{(sol)}_{11}(i)-
\partial_zM^{(sol)}_{21}(z)|_{z=i}M^{(sol)}_{11}(i)^{-1}+O(t^{-1/4})\nonumber \\
=&q_{sol}(x,t,\Lambda)+O(t^{-1/4}),
\end{align}
and
\begin{align}
x(y,t)=y+2\ln(T(i))+c_++O(t^{-1/4}).
\end{align}
Similarly, the same treatment can be done for the regions $\xi\in(-1/4,0)\cup(0,2)$, which is summarized as the following theorem.

\begin{thm}
Let $q(x,t)$ be the solution of mCH equation \eqref{mCH-2} corresponding to initial data $q_0(x)\in H^{2,1}(\D{R})$ and assume that $q_0(x)$ is generic. Let $\left\{r(z),\{\eta_n,c_n\}_{n=1}^{4N_1+2N_2}\right\}$ be the scattering data generated by the
initial data $q_0(x)$ and $q_{sol}(x,t,\Lambda)$ denote the $N(\Lambda)$-soliton solution with the modified scattering data $\sigma_d^{\Lambda}=\left\{0,\{\eta_n,c_nT^{2}(\eta_n)\}_{n\in\Lambda}\right\}$ shown in Corollary \ref{cor-sol}. Then the solution to the mCH equation can be described as follow: as $t\to+\infty$
\begin{enumerate}[(I)]
\item As $\xi\in(-\infty,-1/4)\cup(2,+\infty)$, then
\begin{align}
q(x,t)&=q_{sol}(x,t,\Lambda)+O(t^{-1/4}),\\
x(y,t)&=y+2\ln(T(i))+c_++O(t^{-1/4}),
\end{align}
where the parameters $q_{sol}(x,t,\Lambda)$, $c_+$ and $T(i)$ are defined in Corollary \ref{cor-sol} and Proposition \ref{PP-T}.
\item As $\xi\in(-1/4,0)\cup(0,2)$, then
\begin{align}
q(x,t)&=q_{sol}(x,t,\Lambda)(1+T_0)+t^{-1/2}h_{11}+O(t^{-3/4}),\\
x(y,t)&=y+2\ln(T(i))+c_+'+t^{-1/2}h_{12}+O(t^{-3/4})
\end{align}
where the parameters $h_{11}$, $h_{12}$ and $\ln(T(i))$ are defined in Proposition \ref{asy-err} and Proposition \ref{PP-T}.
\end{enumerate}
\end{thm}

\section*{Acknowledgements}\quad

This work was supported by the National Natural Science Foundation of China under Grant No. 11975306, the Natural Science Foundation of Jiangsu Province under Grant No. BK20181351, the Six Talent Peaks Project in Jiangsu Province under Grant No. JY-059, and the Fundamental Research Fund for the Central Universities under the Grant Nos. 2019ZDPY07 and 2019QNA35.

\appendix
\renewcommand\thefigure{\Alph{section}\arabic{figure}}
\setcounter{figure}{0}
\renewcommand{\theequation}{\thesection.\arabic{equation}}
\setcounter{equation}{0}

\section{The proof of Proposition \ref{pp-phase-sol}}\label{AppA}

In this appendix, we will prove in detail the match with parabolic cylindrical function at phase point $\xi_1$ for $\xi\in(-1/4,0)\cup(0,2)$.
\subsection{Local model for $\xi\in(0,2)$}\label{AppA-02}
Now let's first consider region $\xi\in(0,2)$, which corresponds to $\theta''(\xi_1)>0$.
\begin{RHP}\label{RHP-A1}
Find a $2\times2$ matrix $M^{(pc),\xi_1}(\zeta)$ such that
\begin{enumerate}[(I)]
\item $M^{(pc),\xi_1}(\zeta)$ is analytic in analytic in $\mathbb{C}\setminus\Sigma^{pc}$,
where
$\Sigma^{pc}=\{\mathbb{R}e^{i\varphi}\}\cup\{\mathbb{R}e^{i(\pi-\varphi)}\}$ shown in Fig.\ref{Fig-phase3}.
\item Jump condition:
\begin{align}\label{Q1}
M^{(pc),\xi_1}_+(\zeta)=M^{(pc),\xi_1}_-(\zeta)V^{PC}(\zeta),~~\zeta\in\Sigma^{pc},
\end{align}
where
\begin{align}\label{Q2}
V^{PC}(\zeta)=\left\{\begin{aligned}
&\zeta^{iv\widehat{\sigma}_3}e^{-\frac{i\zeta^2}{4}\widehat{\sigma}_3}\left(
   \begin{array}{cc}
     1 & 0 \\
     -\overline{r}_{\xi_1} & 1 \\
   \end{array}
 \right),&&\zeta\in\mathbb{R}^+e^{i\varphi},\\
&\zeta^{iv\widehat{\sigma}_3}e^{-\frac{i\zeta^2}{4}\widehat{\sigma}_3}\left(
  \begin{array}{cc}
    1 & \frac{r_{\xi_1}}{1-|r_{\xi_1}|^2}  \\
    0 & 1 \\
  \end{array}
\right),&&\zeta\in\mathbb{R}^+e^{i(\pi-\varphi)},\\
&\zeta^{iv\widehat{\sigma}_3}e^{-\frac{i\zeta^2}{4}\widehat{\sigma}_3}\left(
  \begin{array}{cc}
    1 & 0 \\
    -\frac{r_{\xi_1}}{1-|r_{\xi_1}|^2} & 1 \\
  \end{array}
\right),&&\zeta\in\mathbb{R}^+e^{i(-\pi+\varphi)},\\
&\zeta^{iv\widehat{\sigma}_3}e^{-\frac{i\zeta^2}{4}\widehat{\sigma}_3}\left(
   \begin{array}{cc}
     1 & -r_{\xi_1} \\
     0 & 1 \\
   \end{array}
 \right),&&\zeta\in\mathbb{R}^+e^{-i\varphi}.
 \end{aligned}\right.
\end{align}
\item Asymptotic behavior:
\begin{align}\label{Q3}
M^{(pc),\xi_1}(\zeta)=I+\frac{M^{(pc),\xi_1}_1(\zeta)}{\zeta}+O(\zeta^{-2}),
~~~\zeta\rightarrow\infty.
\end{align}
\end{enumerate}
\end{RHP}

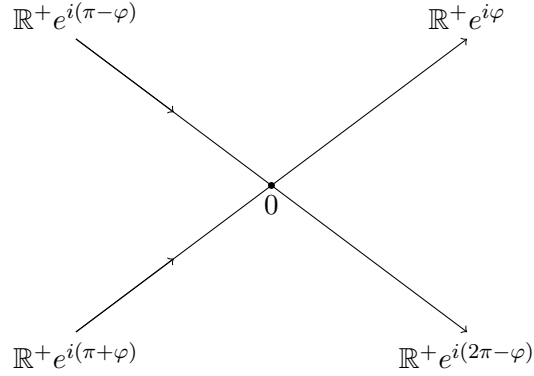
\begin{figure}[H]
\centering
\begin{tikzpicture}[scale=1.3]
\draw[->](0,0)--(2,1.5)node[above]{$\mathbb{R}^+e^{ i\varphi }$};
\draw(0,0)--(-2,1.5)node[above]{$\mathbb{R}^+e^{ i(\pi-\varphi) }$};
\draw(0,0)--(-2,-1.5)node[below]{$\mathbb{R}^+e^{ i(\pi+\varphi) }$};
\draw[->](0,0)--(2,-1.5)node[below]{$\mathbb{R}^+e^{i(2\pi-\varphi) }$};
\draw[->](-2,-1.5)--(-1,-0.75);
\draw[->](-2,1.5)--(-1,0.75);
\coordinate (A) at (1,0.5);
\coordinate (B) at (1,-0.5);
\coordinate (G) at (-1,0.5);
\coordinate (H) at (-1,-0.5);
\coordinate (I) at (0,0);
\fill (I) circle (1pt) node[below] {$0$};
\end{tikzpicture}
\caption{The contour $\Sigma^{pc}$ for the case of $\xi_1$ in the region $\xi\in(0,2)$.}\label{Fig-phase3}
\end{figure}

The next goal is to further reduce RH problem \ref{RHP-A1} to a model RH problem, whose solution can be given by using the asymptotic properties of parabolic cylindrical functions. Like the standard NLS equation and mKdV equation (see \cite{DZ-CMP-1994,DZ-AM1993,Its-1981}), we make the following transformation
\begin{align}
M^{(pc),\xi_1}(\zeta)=\Psi(\zeta)\mathcal {P}\zeta^{-iv\sigma_3}e^{\frac{i}{4}\zeta^2\sigma_3},
\end{align}
where
\begin{align*}
\mathcal {P}(\xi)=\left\{\begin{aligned}
&\left(
  \begin{array}{cc}
    1 & 0 \\
    \overline{r}_{\xi_1} & 1 \\
  \end{array}
\right),&&\arg\zeta\in(0,\varphi),\\
&\left(
  \begin{array}{cc}
    1 & r_{\xi_1} \\
    0 & 1 \\
  \end{array}
\right),&&\arg\zeta\in(2\pi-\varphi),\\
&\left(
   \begin{array}{cc}
     1 & 0 \\
     -\frac{\overline{r}_{\xi_1}}{1+|r_{\xi_1}|^2} & 1 \\
   \end{array}
 \right),&&\arg\zeta\in(\pi,\pi+\varphi),\\
&\left(
   \begin{array}{cc}
     1 & -\frac{r_{\xi_1}}{1+|r_{\xi_1}|^2} \\
     0 & 1 \\
   \end{array}
 \right),&&\arg\zeta\in(\pi-\varphi,\pi).
\end{aligned}\right.
\end{align*}
Through calculation, it is known that matrix $\Psi(\zeta)$ has the same constant jump matrix $V^{\Psi}$ along the positive real axis and the negative real axis. Additionally, the matrix $\Psi(\zeta)$ satisfies the following RH problem.
\begin{RHP}\label{RHP-A2}
The matrix $\Psi(\zeta)$ is analytic in $\D{C}\setminus\D{R}$ with the following properties
\begin{enumerate}[(I)]
\item For $\zeta\in\D{R}$, the continuous values $\Psi_{\pm}(\zeta)$ on $\D{R}$ satisfy that
\begin{align}\label{Q4}
\Psi_{+}(\zeta)=\Psi_{-}(\zeta)V^{\Psi}(\zeta), ~~~~\zeta\in\mathbb{R},
\end{align}
where $V^{\Psi}(\zeta)=\left(
                         \begin{array}{cc}
                           1-|r_{\xi_1}|^2 & r_{\xi_1} \\
                           -\overline{r}_{\xi_1} & 1 \\
                         \end{array}
                       \right).$
\item Asymptotic behavior:

\end{enumerate}
\end{RHP}

\begin{prop}\label{pp-ode}
The entries of the matrix $\Psi(\zeta)=(\psi)_{2\times2}$ obey ordinary differential equation system as follows.
\begin{align}
&\frac{d^2\psi_{11}(\zeta)}{d\zeta^2}
+\left(\frac{i}{2}+\frac{\zeta^2}{4}-\beta^{\xi_1}_{21}
\beta^{\xi_1}_{12}\right)\psi_{11}=0,\
\frac{d^2\psi_{21}(\zeta)}{d\zeta^2}+\left(-\frac{i}{2}+\frac{\zeta^2}{4}-\beta^{\xi_1}_{21}
\beta^{\xi_1}_{12}\right)\psi_{21}=0,\label{psi11}\\
&\frac{d^2\psi_{12}(\zeta)}{d\zeta^2}+\left(\frac{i}{2}+\frac{\zeta^2}{4}-\beta^{\xi_1}_{21}
\beta^{\xi_1}_{12}\right)\psi_{12}=0,\
\frac{d^2\psi_{22}(\zeta)}{d\zeta^2}+\left(-\frac{i}{2}+\frac{\zeta^2}{4}-\beta^{\xi_1}_{21}
\beta^{\xi_1}_{12}\right)\psi_{22}=0,
\end{align}
\end{prop}
\begin{proof}
Differentiating \eqref{Q4} and combining with $\frac{1}{2}i\zeta\sigma_3\Psi_+=\frac{1}{2}i\zeta\sigma_3\Psi_-$ obtains
\begin{align}\label{Q5}
\left(\frac{d\Psi}{d\zeta}+\frac{i}{2}\zeta\sigma_3\Psi\right)_{+}=
\left(\frac{d\Psi}{d\zeta}+\frac{i}{2}\zeta\sigma_3\Psi\right)_{-}V^{\Psi}(\zeta).
\end{align}
Observing that $\det V^{\Psi}(\zeta)\equiv1$, then one has $\det \Psi_+=\det \Psi_-$ and $\det \Psi$ is analytic in the $\zeta$-plane. Applying Liouville theorem yields that $\det \Psi=1$ and $\Psi^{-1}$ Additionally, the matrix-valued function $\left(\frac{d\Psi}{d\zeta}+\frac{i}{2}\zeta\sigma_3\Psi\right)(\Psi)^{-1}$ has no jump condition along the real axis $\mathbb{R}$ and is an entire function with respect to $\zeta$. By direct calculation, we can get
\begin{align}\label{Q7}
\begin{split}
\left(\frac{d\Psi}{d\zeta}+\frac{i}{2}\zeta\sigma_3\Psi\right)(\Psi)^{-1}&=
\left(\frac{d M^{(pc),\xi_1}(\zeta)}{d\zeta}+M^{(pc),\xi_1}(\zeta)\frac{iv\sigma_3}{\zeta}\right)
M^{(pc),\xi_1}(\zeta)^{-1}\\&-\frac{i}{2}\zeta\left[M^{(pc),\xi_1}(\zeta),\sigma_3\right]
M^{(pc),\xi_1}(\zeta)^{-1}.
\end{split}
\end{align}
It follows from Liouville theorem that
there exists a constant matrix $\beta^{\xi_1}$ such that
\begin{align}\label{Q8}
\beta^{\xi_1}=-\frac{i}{2}\left[M^{(pc),\xi_1}(\zeta),\sigma_3\right]=
\left(
  \begin{array}{cc}
    0 & i[M^{(pc),\xi_1}(\zeta)]_{12} \\
    -i[M^{(pc),\xi_1}(\zeta)]_{21} & 0 \\
  \end{array}
\right)\triangleq \left(
                    \begin{array}{cc}
                      0 & \beta_{12}^{\xi_1} \\
                      \beta_{21}^{\xi_1} & 0 \\
                    \end{array}
                  \right).
\end{align}
Also using Liouville theorem again, one has
\begin{align}\label{Q9}
\left(\frac{d\Psi}{d\zeta}+\frac{i}{2}\zeta\sigma_3\Psi\right)=\beta^{\xi_1}\Psi.
\end{align}
Expanding \eqref{Q9} yields that
\begin{align*}
&\frac{d\psi_{11}}{d\zeta}+\frac{i}{2}\zeta\psi_{11}=\beta_{12}^{\xi_1}\psi_{21},\
\frac{d\psi_{21}}{d\zeta}-\frac{i}{2}\zeta\psi_{21}=\beta_{21}^{\xi_1}\psi_{11},\\
&\frac{d\psi_{12}}{d\zeta}+\frac{i}{2}\zeta\psi_{12}=\beta_{12}^{\xi_1}\psi_{22},\
\frac{d\psi_{22}}{d\zeta}-\frac{i}{2}\zeta\psi_{22}=\beta_{21}^{\xi_1}\psi_{12}.
\end{align*}
Direct calculation can get the results of Proposition \ref{pp-ode}.

\end{proof}

Attention now turns to solving $\beta_{12}^{\xi_1}$ and $\beta_{21}^{\xi_1}$. As in \cite{DZ-AM1993}, the parabolic cylinder (PC) equation
\begin{align}\label{Q10}
y^{''}+\left(\frac{1}{2}-\frac{z^2}{4}-a\right)y=0,
\end{align}
where the PC functions $D_a(z)$, $D_a(-z)$, $D_{-a-1}(iz)$ and $D_{-a-1}(-iz)$ are entire functions for any $a$ with the asymptotic behavior as $z\rightarrow\infty$
\begin{equation}\label{asy-Da}
D_{a}(z)=\left\{
             \begin{aligned}
             &z^a e^{-\frac{z^2}{4}}\left(1+O(z^{-2})\right),  &&\vert{\rm arg}z\vert<\frac{3\pi}{4},\\
             &z^a e^{-\frac{z^2}{4}}\left(1+O(z^{-2})\right)-\frac{\sqrt{2\pi}}
             {\Gamma(-a)}e^{ia\pi}z^{-a-1}e^{z^2/4}\left(1+O(z^{-2})\right),  && \frac{\pi}{4}<{\rm arg}z<\frac{5\pi}{4},\\
             &z^a e^{-\frac{z^2}{4}}\left(1+O(z^{-2})\right)-\frac{\sqrt{2\pi}}
             {\Gamma(-a)}e^{-ia\pi}z^{-a-1}e^{z^2/4}\left(1+O(z^{-2})\right), && -\frac{5\pi}{4}<{\rm arg}z<-\frac{\pi}{4}.\\
             \end{aligned}
             \right.
\end{equation}
\begin{prop}\label{PA2}
RH problem \ref{RHP-A2} has the unique solution with the following forms
\begin{align}\label{Q11}
\Psi(\zeta)=\left(
              \begin{array}{cc}
                e^{-\frac{3\pi}{4}v(\xi_1)}D_{iv(\xi_1)}(\zeta e^{-\frac{3\pi i}{4}}) &
                \frac{e^{\frac{\pi}{4}(v(\xi_1)-i)}}{\beta_{21}^{\xi_1}}(-iv(\xi_1))
                D_{-iv(\xi_1)-1}(\zeta e^{-\frac{\pi i}{4}}) \\
                \frac{e^{\frac{\pi}{4}(v(\xi_1)-i)}}{\beta_{21}^{\xi_1}}iv(\xi_1)
                D_{iv(\xi_1)-1}(\zeta e^{-\frac{\pi i}{4}}) &
                 e^{\frac{\pi}{4}v(\xi_1)}D_{-iv(\xi_1)}(\zeta e^{-\frac{\pi i}{4}}) \\
              \end{array}
            \right),
\end{align}
for $\Im(\zeta)>0$ and for $\Im(\zeta)<0$
\begin{align}\label{Q12}
\Psi(\zeta)=\left(
              \begin{array}{cc}
                e^{\frac{\pi}{4}v(\xi_1)}D_{iv(\xi_1)}(\zeta e^{\frac{\pi i}{4}}) &
                -\frac{iv(\xi_1)}{\beta_{21}^{\xi_1}} e^{-\frac{3\pi}{4}(v(\xi_1)-i)}
                D_{-iv(\xi_1)-1}(\zeta e^{\frac{3\pi i}{4}}) \\
                \frac{iv(\xi_1)}{\beta_{12}^{\xi_1}}e^{\frac{\pi}{4}(v(\xi_1)+i)}
                D_{iv(\xi_1)-1}(\zeta e^{\frac{\pi i}{4}}) &
                 e^{-\frac{3\pi}{4}v(\xi_1)}D_{-iv(\xi_1)}(\zeta e^{\frac{3\pi i}{4}}) \\
              \end{array}
            \right).
\end{align}
\end{prop}
\begin{proof}
Let $v=\beta_{12}^{\xi_1}\beta_{21}^{\xi_1}$,  then the first equation of \eqref{psi11} can be rewritten as
\begin{align}\label{Q13}
\frac{d^2\psi_{11}}{d\eta^2}+\left(\frac{1}{2}-\frac{\eta^2}{4}+iv\right)\psi_{11}=0,
\end{align}
under the new variable $\eta=\zeta e^{-\frac{3\pi i}{4}}$ for $\Im(\zeta)>0$, which has a set of linearly independent solutions for some constants $c_1$ and $c_2$
\begin{align*}
\psi_{11}=c_1D_a(e^{-\frac{3\pi i}{4}}\zeta)+c_1D_{a}(-e^{-\frac{3\pi i}{4}}\zeta).
\end{align*}
Due to $\zeta\in\mathbb{C}^+$, thus $-\frac{3\pi}{4}<\arg\eta<\frac{\pi}{4}$.
Combining the asymptotic behavior $\Psi$ with  the properties of $D_a(z)$ yields the result
\begin{align}
\psi_{11}=e^{-\frac{3\pi}{4}v(\xi_1)}D_{iv(\xi_1)}(e^{-\frac{3\pi i}{4}}\zeta).
\end{align}
On the other hand, applying \eqref{psi11} yields
\begin{align}\label{Q14}
\psi_{21}&=\frac{1}{\beta_{21}^{\xi_1}}\left[e^{-\frac{3\pi}{4}v(\xi_1)}\partial(D_{iv(\xi_1)}
(e^{-\frac{3\pi i}{4}}\zeta))+\frac{i\zeta}{2}e^{-\frac{3\pi}{4}v(\xi_1)}D_{iv(\xi_1)}
(e^{-\frac{3\pi i}{4}}\zeta)\right]\nonumber\\
&=\frac{1}{\beta_{21}^{\xi_1}}e^{-\frac{3\pi}{4}(v(\xi_1)+i)}\left[\partial(D_{iv(\xi_1)}
(e^{-\frac{3\pi i}{4}}\zeta))+
\frac{\zeta}{2}e^{-\frac{3\pi i}{4}}D_{iv(\xi_1)}(e^{-\frac{3\pi i}{4}}\zeta)\right]\nonumber\\
&=\frac{1}{\beta_{21}^{\xi_1}}e^{-\frac{3\pi}{4}(v(\xi_1)+i)}D_{iv(\xi_1)-1}(e^{-\frac{3\pi i}{4}}\zeta).
\end{align}

Similarly, for $\psi_{22}$, taking the new variable $\eta=e^{-\frac{\pi i}{4}}\zeta$
\begin{align*}
\psi_{22}''+\left(-\frac{\zeta^2}{4}-iv+\frac{1}{2}\right)\psi_{22}=0,
\end{align*}
also
\begin{align*}
\psi_{22}=e^{\frac{\pi v}{4}}D_{-iv}(e^{\frac{-\pi}{4}}\zeta),\ \zeta\in\D{C}^+.
\end{align*}
Then, one has
\begin{align*}
\psi_{12}=\frac{1}{\beta_{21}}e^{\frac{\pi v}{4}}\left(\partial_{\zeta}(D_{-iv}(\zeta e^{-\frac{i\pi}{4}}))-\frac{i\zeta}{2}D_{-iv}(\zeta e^{-\frac{i\pi}{4}})\right).
\end{align*}

As $\Im(\zeta)<0$, we take the transformations $\eta=\zeta e^{i\pi/4}$ and $\eta=\zeta e^{3\pi i/4}$ to solve $\psi_{11}$ $\psi_{21}$ and $\psi_{22}$ $\psi_{12}$ by repeating the above process, respectively.
\end{proof}

Next, we turn our attention to solving $\beta_{12}^{\xi_1}$. It follows from \eqref{Q4} that
\begin{align}
-\bar{r}_{\xi_1}&=\psi_{11}^-\psi_{21}^+-\psi_{21}^-\psi_{11}^+ \nonumber \\
&=e^{\frac{\pi}{4}v(\xi_1)}D_{iv(\xi_1)}(e^{\frac{\pi i}{4}}\zeta)\frac{e^{-\frac{3\pi v(\xi_1)}{4}}}{\beta_{12}^{\xi_1}}\left[\partial_{\zeta}(D_{iv(\xi_1)}(e^{-\frac{3\pi i}{4}}\zeta))+\frac{i\zeta}{2}(D_{iv(\xi_1)}(e^{-\frac{3\pi i}{4}}\zeta))\right. \nonumber \\
&-e^{-\frac{3\pi v(\xi_1)}{4}}\frac{e^{\frac{\pi v(\xi_1)}{4}}}{\beta_{12}^{\xi_1}}D_{iv(\xi_1)}(e^{-\frac{3\pi i}{4}}\zeta)
\left.\partial_{\zeta}(D_{iv(\xi_1)}
(e^{\frac{\pi i}{4}}\zeta))+\frac{i\zeta}{2}D_{iv(\xi_1)}(e^{\frac{\pi i}{4}}\zeta)  \right]\nonumber \\
&=\frac{1}{\beta_{12}^{\xi_1}}\frac{\sqrt{2\pi}e^{\frac{\pi i}{4}}
e^{-\frac{\pi v(\xi_1)}{2}}}{\Gamma(-iv(\xi_1))},
\end{align}
which implies that
\begin{align}
&\beta_{12}^{\xi_1}=-\frac{\sqrt{2\pi}e^{\frac{\pi i}{4}}
e^{-\frac{\pi v(\xi_1)}{2}}}{\bar{r}_{\xi_1}\Gamma(-iv(\xi_1))},\ \beta_{12}^{\xi_1}\beta_{21}^{\xi_1}=v(\xi_1),\\
&\arg(\beta_{12}^{\xi_1})=\frac{\pi}{4}+\arg(r_{\xi_1})-\arg(\Gamma(-iv(\xi_1))).
\end{align}

\subsection{Local model for $\xi\in(-1/4,0)$}
Now let's first consider region $\xi\in(-1/4,0)$, which corresponds to $\theta''(\xi_1)<0$.
\begin{RHP}\label{RHP-A3}
The matrix $M^{(pc),\xi_1}(\zeta)$ solves RH problem \ref{RHP-A1} by replaced the jump condition

\begin{align}\label{Q15}
V^{PC}(\zeta)=\left\{\begin{aligned}
&\zeta^{-iv\widehat{\sigma}_3}e^{\frac{i\zeta^2}{4}\widehat{\sigma}_3}\left(
  \begin{array}{cc}
    1 & \frac{r_{\xi_1}}{1-|r_{\xi_1}|^2}  \\
    0 & 1 \\
  \end{array}
\right),&&\zeta\in\mathbb{R}^+e^{i\varphi},\\
&\zeta^{-iv\widehat{\sigma}_3}e^{\frac{i\zeta^2}{4}\widehat{\sigma}_3}\left(
  \begin{array}{cc}
    1 & 0  \\
    -\overline{r}_{\xi_1} & 1 \\
  \end{array}
\right),&&\zeta\in\mathbb{R}^+e^{i(\pi-\varphi)},\\
&\zeta^{-iv\widehat{\sigma}_3}e^{\frac{i\zeta^2}{4}\widehat{\sigma}_3}\left(
  \begin{array}{cc}
    1 & r_{\xi_1} \\
    0 & 1 \\
  \end{array}
\right),&&\zeta\in\mathbb{R}^+e^{i(-\pi+\varphi)},\\
&\zeta^{-iv\widehat{\sigma}_3}e^{\frac{i\zeta^2}{4}\widehat{\sigma}_3}\left(
   \begin{array}{cc}
     1 & 0 \\
     -\frac{\overline{r}_{\xi_1}}{1-|r_{\xi_1}|^2} & 1 \\
   \end{array}
 \right),&&\zeta\in\mathbb{R}^+e^{-i\varphi},
 \end{aligned}\right.
\end{align}
and the jump contour is shown in Fig.\ref{phas2}.
\begin{figure}[H]
\centering
\begin{tikzpicture}[scale=1.3]
\draw[->](0,0)--(2,1.2)node[above]{$\mathbb{R}^+e^{\varphi i}$};
\draw(0,0)--(-2,1.2)node[above]{$\mathbb{R}^+e^{(\pi-\varphi)i}$};
\draw(0,0)--(-2,-1.2)node[below]{$\mathbb{R}^+e^{(-\pi+\varphi)i}$};
\draw[->](0,0)--(2,-1.2)node[below]{$\mathbb{R}^+e^{-\varphi i}$};
\draw[dashed](-2,0)--(2,0)node[right]{\scriptsize $\Re\zeta$};
\draw[->](-2,-1.2)--(-1,-0.6);
\draw[->](-2,1.2)--(-1,0.6);
\coordinate (A) at (-1.2,0.5);
\coordinate (B) at (-1.2,-0.5);
\coordinate (G) at (1.4,0.5);
\coordinate (H) at (1.4,-0.5);
\coordinate (I) at (0,0);
\fill (I) circle (1pt) node[below] {$0$};
\end{tikzpicture}
\caption{The contour $\Sigma^{pc}$ for the case of $\xi_1$ in the region $\xi\in(-1/4,0)$.}
	\label{phas2}
\end{figure}
\end{RHP}

Similarly, we take the transformation as \cite{DZ-CMP-1994,DZ-AM1993,Its-1981}
\begin{align}
M^{(pc),\xi_1}(\zeta)=\Psi(\zeta)\mathcal {P}\zeta^{-iv\sigma_3}e^{\frac{i}{4}\zeta^2\sigma_3},
\end{align}
where
\begin{align*}
\mathcal {P}(\xi)=\left\{\begin{aligned}
&\left(
  \begin{array}{cc}
    1 & -\frac{\overline{r}_{\xi_1}}{1+|r_{\xi_1}|^2} \\
    0 & 1 \\
  \end{array}
\right),&&\arg\zeta\in(0,\varphi),\\
&\left(
  \begin{array}{cc}
    1 & r_{\xi_1} \\
    -\frac{r_{\xi_1}}{1+|r_{\xi_1}|^2} & 1 \\
  \end{array}
\right),&&\arg\zeta\in(-\varphi,0),\\
&\left(
   \begin{array}{cc}
     1 & r_{\xi_1} \\
     0 & 1 \\
   \end{array}
 \right),&&\arg\zeta\in(-\pi,-\pi+\varphi),\\
&\left(
   \begin{array}{cc}
     1 & 0 \\
     \overline{r}_{\xi_1} & 1 \\
   \end{array}
 \right),&&\arg\zeta\in(\pi-\varphi,\pi).
\end{aligned}\right.
\end{align*}

\begin{RHP}\label{RHP-A4}
The matrix $\Psi$ solves RH problem \ref{RHP-A3} with the asymptotic behavior replaced by
\begin{align}
\Psi(\zeta)=\left(I+\frac{\psi_1}{\zeta}+O(\zeta^{-2})\right)\zeta^{iv\sigma_3
}e^{-\frac{i}{4}\zeta^2\sigma_3},~~~~\zeta\rightarrow\infty.
\end{align}
\end{RHP}

Similar to Section \ref{AppA-02}, the matrix $\Psi$ satisfied by RH problem \ref{RHP-A4} follows a ordinary differential equation system, which can be solved by PC model.
\begin{prop}\label{pp-ode40}
The entries of matrix $\Psi=(\psi)_{2\times2}$ satisfy the following system
\begin{align}
\frac{d\psi_{11}}{d\zeta}-\frac{i\zeta}{2}\psi_{11}=\beta_{12}\psi_{21},\
\frac{d\psi_{21}}{d\zeta}+\frac{i\zeta}{2}\psi_{21}=\beta_{21}\psi_{11},\\
\frac{d\psi_{12}}{d\zeta}-\frac{i\zeta}{2}\psi_{12}=\beta_{12}\psi_{22},\
\frac{d\psi_{22}}{d\zeta}+\frac{i\zeta}{2}\psi_{22}=\beta_{21}\psi_{12}.
\end{align}
\end{prop}

It follows from PC equation and asymptotic behavior \eqref{asy-Da} that we obtain the unique solution to RH problem \ref{RHP-A4}.

\begin{prop}\label{pA4}
RH problem \ref{RHP-A4} has the unique solution with the following forms
\begin{align}
\Psi(\zeta)=\left(
              \begin{array}{cc}
                e^{\frac{\pi}{4}v(\xi_1)}D_{-iv(\xi_1)}(\zeta e^{-\frac{\pi i}{4}}) &
                \frac{e^{-\frac{3\pi}{4}(v(\xi_1)+i)}}{\beta_{21}^{\xi_1}}(iv(\xi_1))
                D_{iv(\xi_1)-1}(\zeta e^{-3\frac{\pi i}{4}}) \\
                -\frac{e^{\frac{\pi}{4}(v(\xi_1)-i)}}{\beta_{12}^{\xi_1}}iv(\xi_1)
                D_{-iv(\xi_1)-1}(\zeta e^{-\frac{\pi i}{4}}) &
                 e^{-\frac{3\pi}{4}v(\xi_1)}D_{iv(\xi_1)}(\zeta e^{-\frac{3\pi i}{4}}) \\
              \end{array}
            \right),
\end{align}
for $\Im(\zeta)>0$ and for $\Im(\zeta)<0$
\begin{align}
\Psi(\zeta)=\left(
              \begin{array}{cc}
                e^{\frac{5\pi}{4}v(\xi_1)}D_{-iv(\xi_1)}(\zeta e^{-\frac{5\pi i}{4}}) &
                \frac{iv(\xi_1)}{\beta_{21}^{\xi_1}} e^{-\frac{\pi}{4}(v(\xi_1)+i)}
                D_{iv(\xi_1)-1}(\zeta e^{-\frac{7\pi i}{4}}) \\
                -\frac{iv(\xi_1)}{\beta_{12}^{\xi_2}}e^{\frac{5\pi}{4}(v(\xi_1)-i)}
                D_{-iv(\xi_1)-1}(\zeta e^{-\frac{5\pi i}{4}}) &
                 e^{-\frac{7\pi}{4}v(\xi_1)}D_{iv(\xi_1)}(\zeta e^{-\frac{7\pi i}{4}}) \\
              \end{array}
            \right).
\end{align}
\end{prop}

Utilizing the equality $D'_a(z)+\frac{z}{2}D_{a}(z)=aD_{a-1}(z)$ and from the jump condition \eqref{Q5}, one has
\begin{align}
-\overline{r}_{\xi_1}&=\psi_{-,11}\psi_{+,21}-\psi_{-,21}\psi_{+,11}\nonumber \\
&=\frac{e^{\frac{3\pi}{2}v(\xi_1)}}{\beta_{12}^{\xi_1}}\frac{\sqrt{2\pi}e^{-\frac{5\pi i}{4}}}
{\Gamma(iv(\xi_1))},
\end{align}
from which we obtain
\begin{align}
&\beta_{12}^{\xi_1}=-\frac{e^{\frac{3\pi}{2}v(\xi_1)}}{\beta_{12}^{\xi_1}}\frac{\sqrt{2\pi}e^{-\frac{5\pi i}{4}}}{\overline{r}_{\xi_1}\Gamma(iv(\xi_1))},\ \beta_{12}^{\xi_1}\beta_{21}^{\xi_1}=v(\xi_1),\\
&\arg(\beta_{12}^{\xi_1})=-\frac{5\pi}{4}+\arg(r_{\xi_1})-\arg(\Gamma(iv(\xi_1))).
\end{align}

\bibliographystyle{plain}

\end{document}